\documentclass[reqno,10pt]{amsart}
\usepackage{tikz}
\usepackage{overpic}
\usepackage{amssymb}
\usepackage{amsmath}
\usepackage{amsthm}
\usepackage{mathrsfs}
\usepackage{pgf}
 \usepackage{color}
\usepackage{graphicx}   
\usepackage{hyperref}
\usepackage[titletoc]{appendix}

\definecolor{ddorange}{rgb}{1,0.5,0}
\definecolor{ddcyan}{rgb}{0,0.2,1.0}
\newcommand{\GGG}{\color{black}}
\newcommand{\CCC}{\color{black}}
\newcommand{\EEE}{\color{black}}

\usepackage{amsmath}
  \usepackage{paralist}
\usepackage{graphicx}  
\usepackage{psfrag}
\usepackage{comment}
\usepackage{esvect}

\usepackage{mathtools}

\usepackage[T1]{fontenc}
\usepackage[utf8]{inputenc}

\newtheorem{theorem}{Theorem}[section]
\newtheorem{corollary}[theorem]{Corollary}

\newtheorem{lemma}[theorem]{Lemma}
\newtheorem{proposition}[theorem]{Proposition}

\theoremstyle{definition}
\newtheorem{definition}[theorem]{Definition}
\newtheorem{remark}[theorem]{Remark}

\newcommand{\R}{\mathbb{R}}

\newcommand{\N}{\mathbb{N}}
\newcommand{\Z}{\mathbb{Z}}

\newcommand{\eps}{\varepsilon}

\newcommand{\ove}{\overline}

\newcommand{\Ld}{\mathcal{L}^d}

\newcommand{\Rd}{\R^d}
\newcommand{\sm}{\setminus}
\newcommand{\hd}{\mathcal{H}^{d-1}}
\newcommand{\sps}{\sigma^p_{\mathrm{\rm sym}}}
\newcommand{\Mdd}{{\mathbb{M}^{d\times d}_{\rm sym}}}
\newcommand{\Sd}{{\mathbb{S}^{d-1}}}
\newcommand{\xy}{^\xi_y}
\newcommand{\dx}{\, \mathrm{d} x}
\renewcommand{\dh}{\, \mathrm{d} \hd}

\newcommand{\M}{\mathfrak{M}}
\newcommand{\weak}{\rightharpoonup}

\newcommand{\km}{k^{-1}}
\newcommand{\tq}{{\tilde{q}}}
\newcommand{\qz}{{q_z^k}}
\newcommand{\tqz}{{\tilde{q}_z^k}}
\newcommand{\Qz}{{Q_z^k}}

\renewcommand{\d}{\mathrm{d}}
\newcommand{\oho}{\Omega_h^{0}}
\newcommand{\ohi}{\Omega_h^{1}}

\newcommand{\wu}{\widetilde{u}}

\newcommand{\eop}{\nopagebreak\hspace*{\fill}$\Box$\smallskip}

\newcommand{\dod}{{\partial_D \Omega}}
\newcommand{\don}{{\partial_N \Omega}}
\newcommand{\dom}{{\partial \Omega}}

\newcommand{\mres}{\mathbin{\vrule height 1.6ex depth 0pt width
0.13ex\vrule height 0.13ex depth 0pt width 1.3ex}}

\DeclareMathOperator*{\aplim}{ap\,lim}

\usepackage{latexsym,amsfonts}
\usepackage{mathrsfs}
\usepackage{centernot}
\usepackage{paralist}

\usepackage{eso-pic}  
\usepackage{pifont}
\usepackage{fixmath}

\numberwithin{equation}{section}

\begin{document}

\title[Equilibrium configurations of epitaxially strained films and material voids]{Equilibrium configurations for epitaxially strained films and material voids in three-dimensional linear elasticity}

\author{Vito Crismale}
\address[Vito Crismale]{CMAP, \'Ecole Polytechnique, 91128 Palaiseau Cedex, France}
\email[Vito Crismale]{vito.crismale@polytechnique.edu}

\author{Manuel Friedrich}
\address[Manuel Friedrich]{Applied Mathematics M\"unster, University of M\"unster\\
Einsteinstrasse 62, 48149 M\"unster, Germany.}
\email{manuel.friedrich@uni-muenster.de}

%
%
%
%
%
%
%
%
%
%
%
%
%
%
%
%

\begin{abstract}
 We extend the results about existence of minimizers, relaxation, and approximation proven by Chambolle et al.\ in 2002 and 2007 for an energy related to  epitaxially strained crystalline films,   and by Braides, Chambolle, and Solci in 2007 for a class of energies defined on pairs of  function-set.  We study these models in the framework of three-dimensional linear elasticity, where 
 a major obstacle to overcome
 is the lack of any \emph{a priori} assumption on the integrability properties of  displacements.  As a key tool for the proofs,  we introduce a new notion of convergence for $(d{-}1)$-rectifiable sets that are jumps of $GSBD^p$ functions, called $\sps$-convergence.  
\end{abstract}



\maketitle

%

%

\section{Introduction}
The last years have witnessed a remarkable progress in the mathematical and physical literature towards the understanding of  stress driven rearrangement instabilities (SDRI),  i.e.,  morphological instabilities of interfaces between elastic phases generated by the
competition between elastic and surface energies  of (isotropic or anisotropic) perimeter type.  Such  phenomena are  for instance  observed in the formation of material voids inside elastically stressed solids.  Another example is  hetero-epitaxial growth of elastic thin films,  when thin layers of highly strained hetero-systems, such as InGaAs/GaAs or SiGe/Si, are deposited onto a substrate:  in case  of  a mismatch between the lattice parameters of the two crystalline solids,  the free surface of the film is flat until a critical value of the thickness is reached, after which the free surface becomes corrugated   (see e.g.\ \cite{AsaTil72, GaoNix99, Grin86, Grin93,  SieMikVoo04, Spe99} for some physical and numerical literature). 

 From a mathematical point of view,   the common feature of functionals describing SDRI  is  the presence of both  stored elastic bulk  and surface energies. In the static setting, problems arise concerning existence, regularity,  and stability of equilibrium configurations obtained by energy minimization.   The analysis of these issues  is by now mostly developed 
in dimension two  only.

 Starting  with the seminal work by {\sc Bonnetier and Chambolle} \cite{BonCha02} who proved existence of equilibrium configurations, several results  have been obtained    in  \cite{BelGolZwi15, Bon12, FonFusLeoMor07, FonPraZwi14, FusMor12,  GolZwi14}  for  hetero-epitaxially   strained  elastic thin films in 2d. We also refer to   \cite{DavPio18, DavPio17,  KrePio19} for related  energies  and  to  \cite{KhoPio19} for a unified model for SDRI. In the three dimensional setting, results 
 are  limited to the geometrically nonlinear setting  or to linear elasticity  under   antiplane-shear assumption  \cite{Bon13, ChaSol07}.   In a similar fashion, regarding the study of material voids in elastic solids, there are  works about  existence and regularity in dimension two \cite{capriani, FonFusLeoMil11} and a relaxation result in higher dimensions \cite{BraChaSol07}  for nonlinearly elastic energies  or  in  linear elasticity  under  antiplane-shear assumption. \GGG Related to \cite{BraChaSol07}, we also mention a similar relaxation result in presence of obstacles \cite{FocGel08}, and the study of homogenization in periodically perforated domains, cf.\ e.g.\ \cite{CagSca11, FocGel07}. \EEE


The goal of the present paper is to extend the results about relaxation, existence, and approximation obtained for energies related to material voids \cite{BraChaSol07}  and to epitaxial growth \cite{BonCha02, ChaSol07}, respectively, to the case of  linear elasticity in arbitrary space dimensions. As  already observed in \cite{ChaSol07}, the main obstacle for deriving such generalizations lies in the fact that a deep understanding of the function space of \emph{generalized special functions of bounded deformation} ($GSBD$) is necessary. Indeed, our strategy is  based extensively on using the theory on $GSBD$ functions which, initiated by {\sc Dal Maso} \cite{DM13},  was developed over the last years, see e.g.\ \cite{ChaConIur17, CC17, CC18, ConFocIur15,  CFI16ARMA,    CFI17DCL, CFI17Density,  Cri19, Fri17M3AS, FriPWKorn, FriSol16,   Iur14}. In fact, as a byproduct of our analysis, we introduce two new notions related to this function space, namely (1) a version of the space with functions attaining also the value infinity   and  (2) a novel notion for convergence of rectifiable  sets,   which we  call  $\sps$-convergence.   Let us  stress  that  in this work we consider exclusively a static setting.  For evolutionary models, we mention the recent works \cite{FonFusLeoMor15, FusJulMor18, FusJulMor19, Pio14}.

We now introduce the models under consideration in a slightly simplified way, restricting ourselves to three space dimensions. To describe material voids in elastically stressed solids, we consider the following functional defined on pairs of function-set (see \cite{SieMikVoo04})
\begin{equation}\label{eq: F functional-intro}
F(u,E) = \int_{\Omega \setminus E}  \mathbb{C} \, e(u) : e(u) \, \dx + \int_{\Omega \cap \partial E}  \varphi(\nu_E) \, \d\mathcal{H}^{2}\,,
\end{equation}
where $E \subset \Omega$ represents the (sufficiently smooth) shape of voids within an elastic body with reference configuration $\Omega \subset \R^3$, and  $u$ is an elastic displacement field. The first part of the functional represents the elastic energy depending on the linear strain $e(u) := \frac{1}{2}\big((\nabla u)^T + \nabla u\big)$, where $\mathbb{C}$ denotes the fourth-order  positive  semi-definite  tensor of elasticity coefficients. (In fact, we can incorporate more general elastic energies, see \eqref{eq: F functional} below.) The surface energy depends on a (possibly anisotropic) density $\varphi$   evaluated at the outer normal $\nu_E$ to $E$.  This setting is usually complemented with a volume constraint on the voids $E$ and nontrivial  prescribed Dirichlet boundary conditions for $u$ on a part of $\partial \Omega$. We point out that the boundary conditions are the reason why the solid is elastically  stressed.

A variational model for epitaxially strained films can be regarded as a special case of  \eqref{eq: F functional-intro} and corresponds to the situation   where the material domain is the subgraph of an unknown nonnegative function $h$. More precisely, we assume that the material occupies the region 
$${\Omega_h^+ := \{ x \in \omega \times \R : 0 <  x_3   <   h(x_1,x_2)\}}$$
for a given bounded function $h: \omega \to  [0,\infty) $,  $\omega \subset \R^2$, whose graph represents the free profile of the film. We consider the energy
\begin{equation}\label{eq: G functional-intro} 
G(u,h) = \int_{\Omega_h^+} \mathbb{C} \, e(u) : e(u)   \, \dx + \int_{\omega} \sqrt{1 + |\nabla h(x_1,x_2)|^2} \, \d(x_1,x_2)\,.
\end{equation}
Here, $u$ satisfies  prescribed boundary data on $\omega \times \lbrace 0 \rbrace$ which corresponds to the interface between film and substrate.  This Dirichlet boundary condition models the case of a film growing on an  infinitely rigid substrate and is the reason for the film to be strained. We observe that \eqref{eq: G functional-intro}  corresponds to \eqref{eq: F functional-intro} when $\varphi$ is the Euclidean norm,  $\Omega= \omega {\times} (0, M)$  for some $M>0$ large enough, and $E=\Omega \sm  \Omega^+_h$.

Variants of the  above models \eqref{eq: F functional-intro} and \eqref{eq: G functional-intro} have been studied by {\sc Braides, Chambolle, and Solci} \cite{BraChaSol07} and by  {\sc Chambolle and Solci} \cite{ChaSol07}, respectively, where  the  linearly elastic energy density  $\mathbb{C} \, e(u): e(u)$ is replaced by   an elastic energy satisfying a $2$-growth (or $p$-growth, $p>1$) condition in the full gradient $\nabla u$ with quasiconvex integrands.  These works are devoted to giving a sound mathematical formulation for determining equilibrium configurations.  By means of variational methods and geometric measure theory, they study the relaxation of the functionals in terms of  \emph{generalized functions of bounded variation} ($GSBV$) which allows to incorporate the possible roughness of the geometry of voids or films. Existence of minimizers for the relaxed functionals and  the approximation of (the counterpart of) $G$  through a phase-field $\Gamma$-convergence result are addressed.  In fact, the two articles
 have been written almost  simultaneously with many similarities in both the setting and the proof strategy. 

Therefore, we prefer to  present the extension of  both works    to the $GSBD$ setting (i.e., to three-dimensional linear elasticity) in a single work to allow for a comprehensive study of different applications. 
We now briefly  discuss  our main results. 

 \textbf{(a) Relaxation of $F$}: We first note that, for fixed $E$,  $F(\cdot,E)$ is weakly lower semicontinuous in  $H^1$  
 and, for fixed $u$, $F(u,\cdot)$ can be regarded as a lower semicontinuous functional on sets of finite perimeter. The energy defined on pairs $(u,E)$, however, is not lower semicontinuous since, in a limiting process, the voids $E$  may collapse into a  discontinuity of the displacement $u$. The relaxation has to take this phenomenon into account, in particular collapsed surfaces need to be counted twice in the relaxed energy. Provided that the surface density $\varphi$ is a norm in $\R^3$, we show that  the relaxation takes the form (see Proposition \ref{prop:relF})
\begin{equation}\label{eq: oveF}
\ove{F}(u,E) =  \int_{\Omega \setminus E}  \mathbb{C} \, e(u) : e(u)   \, \dx + \int_{\Omega \cap \partial^* E}  \varphi  (\nu_E)  \, {\rm d}\mathcal{H}^2 + \int_{J_u \cap (\Omega \setminus E)^1} 2\, \varphi(\nu_u)  \, {\rm d}\mathcal{H}^2\,,
\end{equation}
where $E$ is a set of finite perimeter with essential boundary $\partial^* E$, $( \Omega \setminus E)^1$ denotes the set of points  of density $1$ of $\Omega \setminus E$, and $u \in GSBD^2(\Omega)$. Here, $e(u)$  denotes the approximate symmetrized gradient of class $L^2(\Omega; \R^{3 \times 3})$ and $J_u$ is the jump set with corresponding    measure-theoretical normal $\nu_u$. (We refer to Section~\ref{sec:prel} for the definition and  the main properties of this function space. Later, we will also consider more general elastic energies and work with the space $GSBD^p(\Omega)$, $1 <p < \infty$, i.e., $e(u) \in L^p(\Omega; \R^{3 \times 3})$.)  

\textbf{(b) Minimizer for $\ove F$:} In Theorem~\ref{thm:relFDir}, we show that such a relaxation result can also be proved by imposing additionally a volume constraint on $E$ (which reflects mass conservation)  and by prescribing boundary data for $u$. For this version of the relaxed  functional,  we prove the existence of minimizers, see Theorem~\ref{th: relF-extended}. 

 \textbf{(c) Relaxation of $G$}: For the model \eqref{eq: G functional-intro}  describing epitaxially strained crystalline films, we show in Theorem \ref{thm:relG} that the lower semicontinous envelope takes the form 
\begin{equation}\label{eq: oveG}
\ove{G}(u,h) = \int_{\Omega_h^+} \mathbb{C} \, e(u) : e(u)   \, \dx  +\mathcal{H}^{2}(\Gamma_h) + 2 \, \mathcal{H}^{2}(\Sigma)\,, 
\end{equation}
where  $h \in BV(\omega; [0,\infty) )$ and $\Gamma_h$ denotes the (generalized) graph of $h$. Here, $u$ is again  a $GSBD^2$-function and the set $\Sigma \subset \R^3$ is a ``vertical'' rectifiable set describing the discontinuity set of $u$ inside the subgraph $\Omega_h^+$.  Similar to the last term in \eqref{eq: oveF}, this contribution has to be counted twice.  We remark that in \cite{FonFusLeoMor07} the set $\Sigma$ is called ``vertical cuts''.   Also here a volume constraint may be imposed. 

\textbf{(d) Minimizer for $\ove G$:}  In Theorem~\ref{thm:compG}, we  show   compactness for sequences with bounded $G$ energy. In particular, this implies existence of minimizers for $\ove G$ (under  a volume constraint).

\textbf{(e) Approximation for $\ove G$:} 
In Theorem~\ref{thm:phasefieldG}, we finally prove a phase-field $\Gamma$-convergence approximation of $\ove G$. We remark that 
we can  generalize the assumptions on the regularity of the Dirichlet datum. Whereas in \cite[Theorem~5.1]{ChaSol07} the class $H^1 \cap L^\infty$ was considered, we show that it indeed  suffices to assume $H^1$-regularity.

We now provide some information on the proof strategy highlighting  in   particular the additional difficulties compared to \cite{BraChaSol07, ChaSol07}. Here, we will also explain why two new technical tools related to the space $GSBD$ have to be introduced.

\textbf{(a)} The proof of the lower inequality for the relaxation $\ove F$ is closely related to the  analog in \cite{BraChaSol07}: we use an approach by slicing, exploit the lower inequality in one dimension, and a localization method.  To prove  the upper inequality, it is enough to  combine the corresponding upper bound from \cite{BraChaSol07} with a density result for  $GSBD^p$ ($p>1$)  functions \cite{CC17},  slightly adapted for our purposes, see Lemma \ref{le:0410191844}.

\textbf{(b)} We point out that, in \cite{BraChaSol07}, the existence of minimizers has not been addressed due to the  lack of a compactness result. In this sense, our study also delivers a conceptionally new result without corresponding counterpart in \cite{BraChaSol07}. The main difficulty lies in the fact that, for configurations with finite energy \eqref{eq: oveF},  small pieces of the body could be 
disconnected from the bulk part, either by the voids $E$ or  by the jump set $J_u$. Thus, since there are \emph{no a priori bounds} on the displacements, the function $u$ could attain arbitrarily large values on certain components, and this might 
 rule out measure convergence for minimizing sequences. We remark  that truncation methods,  used to remedy this issue in scalar problems, are not applicable in the vectorial setting.  This problem was  solved only recently by general compactness results, both in the $GSBV^p$ and the $GSBD^p$ setting. The result \cite{Fri19} in $GSBV^p$   delivers a \emph{ selection  principle for minimizing  sequences} showing that one can always find at least one minimizing sequence converging in measure. With this, existence of minimizers for the 
energies 
in \cite{BraChaSol07} is immediate. 

Our situation in linear elasticity, however, is more delicate since a comparable strong result is not available in $GSBD$. In \cite[Theorem~1.1]{CC18}, a compactness and lower semicontinuity 
result in $GSBD^p$ is derived relying on the idea that minimizing sequences may ``converge to infinity'' on a set of finite perimeter. In the present work, we refine this result by introducing a topology which induces this kind of nonstandard convergence. To this end, we  need to define the new space $GSBD^p_\infty$ consisting of $GSBD^p$ functions which may also attain the value infinity. With these new techniques at hand, we can prove a general compactness result in $GSBD^p_\infty$ (see Theorem~\ref{thm:compF}) which particularly implies the existence of minimizers for  \eqref{eq: oveF}.

\textbf{(c)} Although the functional $G$ in \eqref{eq: G functional-intro}  is a special case of $F$,  the relaxation result is not an immediate consequence, due to the additional constraint that the domain is the subgraph of a function.  Indeed, in the lower inequality, a further crucial step is needed in the description of the (variational) limit of $\partial\Omega_{h_n}$ when $h_n \to h$ in $L^1(\omega)$. In particular, the vertical set $\Sigma$ has to be identified,  see \eqref{eq: oveG}.   

This issue is connected to the problem of detecting all possible limits of jump sets $J_{u_n}$  of converging sequences $(u_n)_n$ of $GSBD^p$ functions. In the $GSBV^p$ setting, the notion of $\sigma^p$-convergence of sets is used, which has originally been developed by {\sc Dal Maso, Francfort, and Toader} \cite{DMFraToa02} to study quasistatic crack evolution in nonlinear elasticity. (We refer also to the variant \cite{GiaPon06} which is independent of $p$.) In this work, we introduce an analogous notion in the $GSBD^p$ setting which we call $\sps$-convergence.  The definition is a bit more complicated compared to the $GSBV$ setting since it has to be formulated in the frame of $GSBD^p_\infty$ functions possibly attaining the value infinity.   We believe that this notion may be of independent interest and is potentially  helpful to study also other  problems such as quasistatic crack evolution in linear elasticity \cite{FriSol16}. We refer to Section \ref{sec:sigmap} for the definition and properties of  $\sps$-convergence, as well as for a comparison to the corresponding notion in the $GSBV^p$ setting. 

Showing  the upper bound for the relaxation result is considerably more difficult than the analogous bound for $\ove F$. In fact, one has to guarantee that  recovery sequences are made up by sets that are still subgraphs. We stress that this cannot be obtained by some general existence results, but is achieved through a very careful construction (pp.\ \pageref{page:upperineqbeg}-\pageref{page:upperineqend}),  that  follows  only  partially the analogous  one   
  in \cite{ChaSol07}. We believe that   the construction in \cite{ChaSol07}   could indeed be improved by adopting an approach similar to ours,  in order     to take also some pathological situations into account.  
 
\textbf{(d)} To show the existence of minimizers of $G$, the delicate step is to prove that minimizing sequences have subsequences which converge (at least) in measure. In the $GSBV^p$ setting, this is simply obtained by applying a Poincar\'e inequality on vertical slices through the film. The same strategy cannot be pursued in $GSBD^p$ since by slicing in a certain direction not all components can be controlled. As a remedy, we proceed in two steps. We first use the novel compactness result in $GSBD^p_\infty$ to identify a limit which might attain the value infinity on a set of finite perimeter $G_\infty$.  Then, \emph{a posteriori}, we show that actually $G_\infty = \emptyset$,
 see Subsection~\ref{subsec:RelG}  for details.  

 \textbf{(e)}  For the phase-field approximation,   we  combine  a variant of the construction in the upper inequality for $\ove G$  with  the general strategy of the corresponding  approximation  result in \cite{ChaSol07}.  The latter is slightly modified  in order to  proceed without $L^\infty$-bound on the displacements.

The paper is   organized   as follows. In Section~\ref{sec: main results}, we introduce the setting of our two models on material voids in elastic solids and epitaxially strained films. Here, we also present our main relaxation, existence,  and approximation  results. Section~\ref{sec:prel} collects definition and main properties of the function space $GSBD^p$.  In this section, we also  define  the space   $GSBD^p_\infty$ and  show   basic properties. In Section \ref{sec:sigmap} we introduce  the novel notion of $\sps$-convergence  and  prove  
a compactness result for sequences of rectifiable sets with bounded Hausdorff measure.  Section \ref{sec:FFF} is devoted to the analysis of  functionals defined on pairs of function-set. Finally, in Section \ref{sec:GGG} we investigate the model for epitaxially strained films and prove the relaxation, existence, and approximation results.

\section{Setting of the problem and statement of the main results}\label{sec: main results}

In this section, we give the precise definitions of the two energy functionals and present the main relaxation,  existence,  and approximation  results. In the following,  $f\colon \mathbb{M}^{d\times d}\to  [0,\infty) $ denotes a convex function satisfying the growth condition ($| \cdot |$ is the Frobenius norm on $\mathbb{M}^{d\times d}$)
\begin{align}\label{eq: growth conditions}
 c_1 |\zeta^T + \zeta|^p \EEE  &\le f(\zeta) \le c_2 (|\zeta^T + \zeta|^p +1) \ \ \ \ \ \text{for all} \ \zeta \in \mathbb{M}^{d\times d}
\end{align}
 and $f(0) = 0$,   
for some $1 < p < + \infty$. \GGG In particular, the convexity of $f$ and \eqref{eq: growth conditions} imply that $f(\zeta) = f(\frac{1}{2}(\zeta^T + \zeta))$ for all $\zeta \in \mathbb{M}^{d\times d}$. \EEE  For an open subset $\Omega \subset \R^d$, we will denote by $L^0(\Omega;\Rd)$ the space of $\Ld$-measurable functions $v \colon \Omega \to \Rd$ endowed with the topology of the convergence in measure. We let $\M(\Omega)$ be the family of all  $\Ld$-measurable subsets of $\Omega$. 

\subsection{Energies on pairs function-set: material voids in elastically stressed solids}\label{sec: results1}

Let $\Omega\subset \R^d$ be a Lipschitz domain.  We introduce an energy functional defined on pairs function-set.  Given a norm  $\varphi$    on $\Rd$ and $f\colon  \mathbb{M}^{d\times d}\to  [0,\infty) $,  we let $F\colon  L^0( \Omega;  \R^d) \times \M(\Omega) \to \R \cup \lbrace+ \infty \rbrace$   be   defined by
\begin{equation}\label{eq: F functional}
F(u,E) = 
\begin{dcases}
\int_{\Omega \setminus E} 
 f(e(u))\, \dx + \int_{\Omega \cap \partial E}  
\varphi(\nu_E) \, \d\mathcal{H}^{d-1}
\hspace{-0.5em} & \\
& \hspace{-4cm} 
\text{if } \partial E \text{ Lipschitz, } u|_{\Omega\sm \ove E} \in W^{1,p}( \Omega \sm \ove E;\R^d), \,  u|_E=0,\\
+ \infty  \hspace{-0.9em} &\hspace{-4cm}  \text{otherwise,}
\end{dcases}
\end{equation}
where $e(u) := \frac{1}{2}\big((\nabla u)^T + \nabla u\big)$ denotes the symmetrized gradient, and $\nu_E$ the outer normal to $E$. We point out that the energy is determined by $E$ and the values of $u$ on $\Omega \setminus \ove E$. The condition $u|_E=0$ is for definiteness only. We denote by $\ove{F}\colon  L^0(\Omega;\R^d) \times \M(\Omega) \to \R \cup \lbrace+ \infty \rbrace$  the lower semicontinuous envelope of the functional $F$ with respect to the convergence in  measure  for the functions and the $L^1(\Omega)$-convergence of characteristic functions of sets, i.e.,
\begin{equation}\label{eq: first-enve}
\ove{F}(u,E) = \inf\Big\{ \liminf_{n\to\infty} F(u_n,E_n)\colon  \,   u_n \to u \text{ in }L^0(\Omega;\R^d) \text{ and } \chi_{E_n} \to \chi_E \text{ in }L^1(\Omega) \Big\}\,.
\end{equation}
\GGG (We observe that the convergence in $L^0(\Omega;\Rd)$ is metrizable, so the sequential lower semicontinuous envelope coincides with the lower semicontinuous envelope with respect to this convergence.) \EEE
In the following, for any $s \in [0,1]$ and any $E \in \M(\Omega)$, $E^s$ denotes the set of points with density $s$ for $E$. By $\partial^* E$ we indicate its essential boundary,  see \cite[Definition 3.60]{AFP}. For the definition of the space $GSBD^p(\Omega)$, $p>1$, we refer to Section \ref{sec:prel} below. In particular, by $e(u) = \frac{1}{2}((\nabla u)^T + \nabla u)$ we denote the approximate  symmetrized  gradient,  and by $J_u$
the jump set of $u$ with  measure-theoretical normal $\nu_u$. We characterize $\ove F$ as follows.

\begin{proposition}[Characterization of  the lower semicontinuous envelope $\ove{F}$]\label{prop:relF} 
Suppose that $f$ is convex and satisfies \eqref{eq: growth conditions}, and that $\varphi$ is a norm on $\Rd$. Then there holds
\begin{equation*}
\ove{F}(u,E) = \begin{dcases} \int_{\Omega \setminus E} f(e(u))\, \dx + \int_{\Omega \cap \partial^* E} & \varphi  (\nu_E)  \, \dh + \int_{J_u \cap (\Omega \setminus E)^1} 2\, \varphi(\nu_u)  \, \dh \\
&\hspace{-1em}\text{if } u=  u \,  \chi_{E^0}  \in GSBD^p(\Omega) \text{ and }\mathcal{H}^{d-1}(\partial^* E) < +\infty\,,\\
+\infty &\hspace{-1em}\text{otherwise.}
\end{dcases}
\end{equation*}
Moreover, if $\mathcal{L}^d(E)>0$, then for any $(u,E) \in L^0(\Omega;\Rd){\times}\M(\Omega)$ there exists a recovery sequence $(u_n,E_n)_n\subset L^0(\Omega;\Rd){\times}\M(\Omega)$ such that $\mathcal{L}^d(E_n) = \mathcal{L}^d(E)$ for all $n \in \N$.
\end{proposition}
The last property shows that it is possible to incorporate a volume constraint on $E$ in the relaxation result. We now move on to consider  a Dirichlet minimization  problem associated to $F$. We will impose Dirichlet boundary data $u_0 \in W^{1,p}(\R^d;\R^d)$ on a subset $\partial_D \Omega \subset \partial \Omega$. For technical reasons, we suppose that  $\dom=\dod\cup \don\cup N$ with $\dod$ and $\don$ relatively open, $\dod \cap \don =\emptyset$, $\hd(N)=0$, 
$\dod \neq \emptyset$,  $\partial(\dod)=\partial(\don)$, and 
that there exist a small $\ove \delta  >0 $ and $x_0\in \Rd$ such that for every $\delta \in (0,\ove \delta)$   there holds 
\begin{equation}\label{0807170103}
O_{\delta,x_0}(\dod) \subset \Omega\,,
\end{equation}
where $O_{\delta,x_0}(x):=x_0+(1-\delta)(x-x_0)$.  (These assumptions are related to  Lemma \ref{le:0410191844} below.)   In the following, we denote by ${\rm tr}(u)$ the trace of $u$ on $\partial \Omega$ which is well defined for functions in $GSBD^p(\Omega)$, see Section  \ref{sec:prel}. In particular, it is well defined for functions $u$ considered in \eqref{eq: F functional} satisfying $u|_{\Omega\sm \ove E} \in W^{1,p}( \Omega \sm \ove E;\R^d)$ and   $u|_E=0$.  By $\nu_\Omega$ we denote the outer unit normal to $\partial \Omega$.

We now introduce a version of $F$ taking boundary data into account.   Given  $u_0 \in W^{1,p}(\R^d;\R^d)$, we    set   
\begin{equation}\label{eq: FDir functional}
F_{\mathrm{Dir}}(u,E) = 
\begin{cases}
 F(u,E) +   \int_{\dod \cap \partial E} \varphi(\nu_E) \, \d\mathcal{H}^{d-1}   & \text{if } \ \  \mathrm{tr}(u)  =\mathrm{tr}(u_0) \text{ on } \partial_D\Omega \setminus \ove E, \\
+\infty &\text{otherwise.}
\end{cases}
\end{equation}
Similar to  \eqref{eq: first-enve}, we define  the lower semicontinuous envelope  $\ove F_{\mathrm{Dir}}$  by  
\begin{equation}\label{eq: Fdir-rela}
\ove{F}_{\mathrm{Dir}}(u,E) = \Big\{ \liminf_{n\to\infty} F_{\mathrm{Dir}}(u_n,E_n)\colon  \,   u_n \to u \text{ in }L^0(\Omega;\R^d) \text{ and } \chi_{E_n} \to \chi_E \text{ in }L^1(\Omega) \Big\}\,.
\end{equation}
We have the following characterization. 
\begin{theorem}[Characterization of  the lower semicontinuous envelope $\ove{F}_{\mathrm{Dir}}$]\label{thm:relFDir}
Suppose that $f$ is convex and satisfies \eqref{eq: growth conditions}, that $\varphi$ is a norm on $\Rd$, and that  \eqref{0807170103} is satisfied. Then there holds
\begin{equation}\label{1807191836}
\ove{F}_{\mathrm{Dir}}(u,E) = 
\ove{F}(u,E) +  \int\limits_{\dod \cap \partial^* E}  \hspace{-0.5cm} \varphi (\nu_E)    \dh + \int\limits_{ \{ \mathrm{tr}(u)  \neq \mathrm{tr}(u_0)  \} \cap  (\dod \setminus \partial^* E) }   \hspace{-0.5cm} 2 \, \varphi(  \nu_\Omega  ) \dh\,.
\end{equation}
Moreover, if $\mathcal{L}^d(E)>0$, then for any $(u,E) \in L^0(\Omega;\Rd){\times}\M(\Omega)$ there exists a recovery sequence $(u_n,E_n)_n\subset L^0(\Omega;\Rd){\times}\M(\Omega)$ such that $\mathcal{L}^d(E_n) = \mathcal{L}^d(E)$ for all $n \in \N$.
\end{theorem}
The proof of Proposition~\ref{prop:relF}  and Theorem~\ref{thm:relFDir} will be given in Subsection \ref{sec: sub-voids}. \CCC There, we provide also two slight generalizations, see Proposition~\ref{prop:relFinfty}  and Theorem~\ref{thm:relFDirinfty}, namely a relaxation with respect to a weaker convergence in a general space  $GSBD^p_\infty$ (cf.\ \eqref{eq: compact extension}), where functions are allowed to attain  the value infinity.  \EEE
We close this subsection with an existence result for $\ove{F}_{\mathrm{Dir}}$, under a volume constraint for the voids.

\begin{theorem}[Existence of minimizers for $\ove{F}_{\mathrm{Dir}}$]\label{th: relF-extended}
Suppose that $f$ is convex and satisfies \eqref{eq: growth conditions}, and that $\varphi$ is a norm on $\Rd$. Let $m>0$. Then the minimization problem
$$
\inf \big\{ \ove{F}_{\mathrm{Dir}}(u,E)\colon (u,E) \in  L^0(\Omega;\Rd){\times}\M(\Omega), \ \mathcal{L}^d(E) = m  \big\}  
$$
admits solutions. 
\end{theorem}
For the proof, we refer to  Subsection \ref{sec: Fcomp}.  It relies on the lower semicontinuity of $\ove{F}_{\mathrm{Dir}}$ and a compactness result in the general space $GSBD^p_\infty$ (cf.\ \eqref{eq: compact extension}),  see  Theorem~\ref{thm:compF}.

\subsection{Energies on domains with a subgraph constraint: epitaxially strained films}\label{sec: results2}

We now  consider the problem of  displacement  fields in a material domain which is the subgraph of an unknown nonnegative function $h$.  Assuming that $h$ is defined on a Lipschitz domain $\omega \subset \R^{d-1}$,  displacement fields $u$ will be defined on the
subgraph 
$$\Omega_h^+ := \{ x \in \omega \times \R \colon 0  <  x_d   <   h(x')\},$$
where here and in the following we use the notation $x = (x',x_d)$ for $x \in \R^d$. To model Dirichlet boundary data at the flat surface $\omega \times \lbrace 0 \rbrace$, we will suppose that functions are extended to the set $\Omega_h := \{ x \in  \omega \times \R  \colon -1 <  x_d   <  h(x')\}$ and satisfy $u = u_0$ on $\omega{\times}(-1,0)$ for a given function $u_0 \in W^{1,p}(\omega{\times}(-1,0);\R^d)$, $p>1$.  In the application to epitaxially strained films, $u_0$ represents the substrate and   $h$ represents the profile of  the  free surface of the film.

For convenience, we introduce the reference domain $\Omega:=\omega{\times} (-1, M+1)$ for $M>0$. We  define the energy functional $G\colon    L^0(\Omega;\R^d) \times L^1(\omega;[0,M]) \to \R \cup \lbrace + \infty \rbrace$ by  
\begin{equation}\label{eq: Gfunctional}
G(u,h) =  
\int_{\Omega_h^+} f(e(u(x))) \, \dx + \int_{\omega} \sqrt{1 + |\nabla h(x')|^2} \, \dx' 
\end{equation}
if $h \in C^1(\omega;[0,M])$, $u|_{\Omega_h} \in W^{1,p}(\Omega_h;\R^d)$,  $u=0$  in $\Omega\sm \Omega_h$, and $u=u_0$  in $\omega{\times}(-1,0)$, and $G(u,h) = +\infty$ otherwise. Here,  $f\colon  \mathbb{M}^{d\times d}\to  [0,\infty) $ denotes a convex function satisfying \eqref{eq: growth conditions}, and as before we set $e(u) := \frac{1}{2}\big((\nabla u)^T + \nabla u\big)$.
 Notice  that, in contrast to \cite{BonCha02}, we suppose  that the functions  $h$ are  equibounded by a value $M$: this is for technical reasons only and is indeed justified from a mechanical point of view since other effects 
come into play for  very high crystal profiles.  

We     study the relaxation of $G$ with respect to the
$L^0(\Omega;\Rd){\times}L^1(\omega; [0,M])$ topology, i.e.,   its  
lower semicontinuous envelope $\ove G\colon  L^0(\Omega;\R^d) \times L^1(\omega;[0, M ]) \to \R \cup \lbrace+ \infty \rbrace$, 
defined as 
\begin{equation*}
\ove{G}(u,h) = \inf\big\{ \liminf\nolimits_{n\to\infty} G(u_n,h_n)\colon  \,  u_n \to u \text { in } L^0(\Omega;\R^d), \ \  h_n \to h \text { in } L^1(\omega) \big\} \,.
\end{equation*}
We characterize $\ove G$ as follows, further assuming that the Lipschitz set  $\omega \subset \R^{d-1}$ is uniformly star-shaped with respect to the origin, i.e.,
\begin{align}\label{eq: star-shaped}
tx \subset \omega \ \ \ \text{for all} \ \ t \in (0,1), \, x \in \partial \omega.
\end{align}
\begin{theorem}[Characterization of  the lower semicontinuous envelope $\ove{G}$]\label{thm:relG}
Suppose that $f$ is convex satisfying \eqref{eq: growth conditions} and that \eqref{eq: star-shaped} holds. Then we have 
\begin{equation*}
\ove{G}(u,h) = 
\begin{dcases}
\int_{\Omega_h^+}  f(e(u))\,& \hspace{-1.0em}\dx  +\mathcal{H}^{d-1}(\partial^* \Omega_h   \cap \Omega  ) + 2 \mathcal{H}^{d-1}(J_u' \cap \Omega_h^1)  \\
& \hspace{-0.4cm} \text{if } u=  u \chi_{ \Omega_h } \in GSBD^p(\Omega), \, u=u_0 \text{ in }\omega{\times}(-1,0),\, h \in BV (\omega;  [0,M]  )\,,  \\
+\infty  &\hspace{-0.4cm} \hspace{-0.4em}\text{ otherwise,}
\end{dcases}
\end{equation*}
where 
\begin{align}\label{eq: Ju'}
J_u' := \lbrace (x',x_d + t)\colon  \, x \in J_u, \,  t \ge 0 \rbrace.
\end{align}
\end{theorem} 

The assumption \eqref{eq: star-shaped} on $\omega$ is more general than the one considered in \cite{ChaSol07}, where $\omega$ is 
 assumed to be  a torus. We point out, however, that both assumptions  are only of technical nature and could be dropped  at  the expense of more elaborated estimates, see also \cite{ChaSol07}.  The proof of this result will be given in Subsection~\ref{subsec:RelG}. 

We note that the functional $G$ could be considered with an additional volume constraint on the film, i.e., $\mathcal{L}^d(\Omega_h^+) = \int_\omega h(x') \, \d x'$ is fixed. 
An easy  adaptation  of the proof shows that the relaxed functional $\ove G$ is not changed under this constraint,  
 see Remark~\ref{rem: volume constraint} for details.   

 In Subsection~\ref{subsec:compactness},  we further prove the following 
general compactness result, from which we deduce the existence of equilibrium  configurations  for epitaxially strained films. 

\begin{theorem}[Compactness 
for $\ove G$]\label{thm:compG} 
Suppose that $f$ is convex and satisfies   \eqref{eq: growth conditions}.
For any $(u_n,h_n)_n$ with $\sup_{n} G(u_n,h_n) <+\infty$, there exist a subsequence (not relabeled)  and functions $u \in GSBD^p(\Omega)$, $h \in BV(\omega;[0,M])$  with $u = u\chi_{\Omega_h}$ and $u=u_0$ on $\omega \times (-1,0)$  such that 
\begin{equation*}
(u_n, h_n) \to (u, h) \quad \text{ in } \quad
 L^0(\Omega;\R^d){\times} { L^1(\omega) \,. }
 \end{equation*}
 \end{theorem}  
In particular, general properties of relaxation (see e.g.\ \cite[Theorem~3.8]{DMLibro}) imply that, given  $0<m<M\mathcal{H}^{d-1}(\omega)$,  the minimization problem 
\begin{align}\label{eq: minimization problem2}
\inf \Big\{ \ove{G}(u,h)\colon (u,E) \in  L^0(\Omega;\R^d) \times L^1(\omega), \ \mathcal{L}^d(\Omega_h^+) =  m  \Big\}  
\end{align} 
 admits solutions.  Moreover, fixed $m$ and the volume constraint $\mathcal{L}^d(\Omega_h^+) =  m$ for $G$ and $\ove G$, any cluster point for minimizing sequences of $G$  is a minimum point for $\ove G$.

 Our final  issue is a   phase-field approximation of $\ove G$. The idea is to represent  any subgraph $\Omega_h$  by a  (regular) function   $v$ which will be an approximation of the characteristic function $\chi_{\Omega_h}$ at a scale of order $\eps$. Let $W \colon [0,1] \to  [0,\infty) $ be continuous, with $W(1)=W(0)=0$, $W>0$ in $(0,1)$, and let $(\eta_\eps)_\varepsilon$ with $\eta_\varepsilon>0$ and $\eta_\varepsilon \varepsilon^{1-p} \to 0$ as $\varepsilon \to 0$.  Let $c_W:=(\int_0^1 \sqrt{2 W(s)} \,\d s)^{-1}$.  In the reference domain $\Omega=\omega{\times}(-1,M+1)$, we introduce the functionals 
\begin{equation}\label{eq: phase-approx}
G_\varepsilon(u,v):=\int_\Omega \bigg( (v^2+\eta_\varepsilon) f(e(u)) + c_W\Big(\frac{W(v)}{\varepsilon} + \frac{\varepsilon}{2} |\nabla v|^2 \Big) \bigg)\dx\,,
\end{equation}
 if 
\[ {u \in W^{1,p}(\Omega; \Rd)\,,\quad u=u_0 \text{ in }\omega{\times}  (-1,0)   \,,}\]
\[
v \in H^1(\Omega; [0,1])\,,\quad v=1\text{ in }\omega{\times}(-1,0)\,, \, v=0\text{ in }\omega{\times}(M,M+1) \quad \partial_d v \leq 0 \,\, \Ld\text{-a.e.\ in }\Omega\,,
\]
and $G_\varepsilon(u,v):=+\infty$ otherwise.
The following phase-field approximation is the analog of \cite[Theorem~5.1]{ChaSol07}  in the frame of linear elasticity.   We remark that here, differently from \cite{ChaSol07}, we assume only $u_0 \in W^{1,p}( \omega \times (-1,0);  \Rd)$, and not necessarily  $u_0 \in L^\infty(\omega \times (-1,0);\R^d)$.  For the proof we refer to Subsection~\ref{sec:phasefield}.   
\begin{theorem}\label{thm:phasefieldG}
 Let $u_0 \in W^{1,p}(\omega \times (-1,0);\R^d)$.  For any  decreasing sequence  $(\varepsilon_n)_n$  of positive numbers converging to zero, the following  hold: 
\begin{itemize}
\item[(i)] 
For any $(u_n, v_n)_n$ with  $\sup_n G_{\varepsilon_n}(u_n, v_n)< +\infty$,  there exist $u \in L^0(\Omega;\Rd)$ and $h\in BV(\omega;  [0,M]  )$ such that, up to a subsequence,  $u_n \to u$ a.e.\ in $\Omega$,  $v_n \to \chi_{\Omega_h}$ in $L^1(\Omega)$, and
\begin{equation}\label{1805192018}
\ove G(u,h) \leq \liminf_{n\to +\infty}G_{\varepsilon_n}(u_n, v_n)\,.
\end{equation}
\item[(ii)]For any $(u,h)$ with $\ove G(u, h)  < + \infty $, there exists $(u_n, v_n)_n$ such that $u_n \to u$ a.e.\ in $\Omega$, $v_n \to \chi_{\Omega_h}$ in $L^1(\Omega)$, and
$$
\limsup_{n\to \infty} G_{\varepsilon_n}(u_n, v_n)= \ove G(u,h)\,.
$$
\end{itemize}
\end{theorem}

 \section{Preliminaries}\label{sec:prel}
 
  In this section, we recall the definition and main properties of the function space $GSBD^p$. Moreover, we introduce the space $GSBD^p_\infty$ of functions which may attain the value infinity.

 \subsection{Notation} 
For every $x\in \Rd$ and $\varrho>0$, let $B_\varrho(x) \subset \R^d$ be the open ball with center $x$ and radius $\varrho$. For $x$, $y\in \Rd$, we use the notation $x\cdot y$ for the scalar product and $|x|$ for the  Euclidean  norm.   By ${\mathbb{M}^{d\times d}}$ and ${\mathbb{M}^{d\times d}_{\rm sym}}$ we denote the set of matrices and symmetric matrices, respectively.   We write $\chi_E$ for the indicator function of any $E\subset \R^n$, which is 1 on $E$ and 0 otherwise.  If $E$ is a set of finite perimeter, we denote its essential boundary by $\partial^* E$, and by $E^s$ the set of points with density $s$ for $E$,  see \cite[Definition 3.60]{AFP}.   
 We indicate the minimum  and maximum  value between $a, b \in \R$ by  $a \wedge b$  and $a \vee b$, respectively. The symmetric difference of two sets $A,B \subset \R^d$ is indicated by $A \triangle B$.

We denote by $\Ld$ and $\mathcal{H}^k$ the $n$-dimensional Lebesgue measure and the $k$-dimensional Hausdorff measure, respectively. For any locally compact subset $B  \subset \Rd$, (i.e.\ any point in $B$ has a neighborhood contained in a compact subset of $B$),
the space of bounded $\R^m$-valued Radon measures on $B$ [respectively, the space of $\R^m$-valued Radon measures on $B$] is denoted by $\mathcal{M}_b(B;\R^m)$ [resp., by $\mathcal{M}(B;\R^m)$]. If $m=1$, we write $\mathcal{M}_b(B)$ for $\mathcal{M}_b(B;\R)$, $\mathcal{M}(B)$ for $\mathcal{M}(B;\R)$, and $\mathcal{M}^+_b(B)$ for the subspace of positive measures of $\mathcal{M}_b(B)$. For every $\mu \in \mathcal{M}_b(B;\R^m)$, its total variation is denoted by $|\mu|(B)$.  Given $\Omega \subset \R^d$ open, we   use the notation  
$L^0(\Omega;\Rd)$  for  the space of $\Ld$-measurable functions $v \colon \Omega \to \Rd$.

\begin{definition}
Let $E\subset \Rd$,  $v \in L^0(E;\R^m)$,   and  $x\in \Rd$ such that
\begin{equation*}
\limsup_{\varrho\to 0^+}\frac{\Ld(E\cap B_\varrho(x))}{\varrho^{  d  }}>0\,.
\end{equation*}
A vector $a\in \Rd$ is the \emph{approximate limit} of $v$ as $y$ tends to $x$ if for every $\varepsilon>0$ there holds
\begin{equation*}
\lim_{\varrho \to 0^+}\frac{\Ld(E \cap B_\varrho(x)\cap \{|v-a|>\varepsilon\})}{ \varrho^d  }=0\,,
\end{equation*}
and then we write
\begin{equation*}
\aplim \limits_{y\to x} v(y)=a\,.
\end{equation*}
\end{definition}

\begin{definition}
Let $U\subset \Rd$  be  open  and $v \in L^0(U;\R^m)$.    The \emph{approximate jump set} $J_v$ is the set of points $x\in U$ for which there exist $a$, $b\in \R^m$, with $a \neq b$, and $\nu\in \Sd$ such that
\begin{equation*}
\aplim\limits_{(y-x)\cdot \nu>0,\, y \to x} v(y)=a\quad\text{and}\quad \aplim\limits_{(y-x)\cdot \nu<0, \, y \to x} v(y)=b\,.
\end{equation*}
The triplet $(a,b,\nu)$ is uniquely determined up to a permutation of $(a,b)$ and a change of sign of $\nu$, and is denoted by $(v^+(x), v^-(x), \nu_v(x))$. The jump of $v$ is the function 
defined by $[v](x):=v^+(x)-v^-(x)$ for every $x\in J_v$. 
\end{definition}
 We note that  $J_v$ is a Borel set  with $\Ld(J_v)=0$, and that $[v]$ is a Borel function. 

\subsection{$BV$ and $BD$ functions}
 Let $U\subset \Rd$  be open. We say that a function $v\in L^1(U)$ is a \emph{function of bounded variation} on $U$, and we write $v\in BV(U)$, if $\mathrm{D}_i v\in \mathcal{M}_b(U)$ for  $i=1,\dots,d$,  where $\mathrm{D}v=(\mathrm{D}_1 v,\dots, \mathrm{D}_d v)$ is its distributional  derivative.  A vector-valued function $v\colon U\to \R^m$ is in $BV(U;\R^m)$ if $v_j\in BV(U)$ for every $j=1,\dots, m$.
The space $BV_{\mathrm{loc}}(U)$ is the space of $v\in L^1_{\mathrm{loc}}(U)$ such that $\mathrm{D}_i v\in \mathcal{M}(U)$ for $i=1,\dots,d$. 


A function $v\in L^1(U;\Rd)$ belongs to the space of \emph{functions of bounded deformation} if 
 the distribution 
$\mathrm{E}v := \frac{1}{2}((\mathrm{D}v)^T + \mathrm{D}v )$  belongs to $\mathcal{M}_b(U;\Mdd)$.
It is well known (see \cite{AmbCosDM97, Tem}) that for $v\in BD(U)$, $J_v$ is countably $(\hd, d-1)$ rectifiable, and that
\begin{equation*}
\mathrm{E}v=\mathrm{E}^a v+ \mathrm{E}^c v + \mathrm{E}^j v\,,
\end{equation*}
where $\mathrm{E}^a v$ is absolutely continuous with respect to $\Ld$, $\mathrm{E}^c v$ is singular with respect to $\Ld$ and such that $|\mathrm{E}^c v|(B)=0$ if $\hd(B)<\infty$, while $\mathrm{E}^j v$ is concentrated on $J_v$. The density of $\mathrm{E}^a v$ with respect to $\Ld$ is denoted by $e(v)$.

The space $SBD(U)$ is the subspace of all functions $v\in BD(U)$ such that $\mathrm{E}^c v=0$. For $p\in (1,\infty)$, we define
\begin{equation*}
SBD^p(U):=\{v\in SBD(U)\colon e(v)\in L^p(\Omega;\Mdd),\, \hd(J_v)<\infty\}\,.
\end{equation*}
Analogous properties hold for $BV$,  such as  the countable rectifiability of the jump set and the decomposition of $\mathrm{D}v$.  The  spaces $SBV(U;\R^m)$ and $SBV^p(U;\R^m)$ are defined similarly, with $\nabla v$, the density of $\mathrm{D}^a v$, in place of $e(v)$.
For a complete treatment of $BV$, $SBV$ functions and $BD$, $SBD$ functions, we refer to \cite{AFP} and to \cite{AmbCosDM97, BelCosDM98,  Tem}, respectively.

\subsection{$GBD$ functions}
We now recall the definition and the main properties of the space $GBD$ of \emph{generalized functions of bounded deformation}, introduced in \cite{DM13}, referring to that paper for a general treatment and more details. Since the definition of $GBD$ is given by slicing (differently from the definition of $GBV$, cf.~\cite{Amb90GSBV, DeGioAmb88GBV}), we  first  need to introduce some notation. Fixed $\xi \in \Sd:=\{\xi \in \Rd\colon |\xi|=1\}$, we let
\begin{equation}\label{eq: vxiy2}
\Pi^\xi:=\{y\in \Rd\colon y\cdot \xi=0\},\qquad B^\xi_y:=\{t\in \R\colon y+t\xi \in B\} \ \ \ \text{ for any $y\in \Rd$ and $B\subset \Rd$}\,,
\end{equation}
and for every function $v\colon B\to  \R^d  $ and $t\in B^\xi_y$ let
\begin{equation}\label{eq: vxiy}
v^\xi_y(t):=v(y+t\xi),\qquad \widehat{v}^\xi_y(t):=v^\xi_y(t)\cdot \xi\,.
\end{equation}

\begin{definition}[\cite{DM13}]
Let $\Omega\subset \Rd$ be a  bounded open set, and let  $v \in L^0(\Omega;\Rd)$.   Then $v\in GBD(\Omega)$ if there exists $\lambda_v\in \mathcal{M}^+_b(\Omega)$ such that  one of the following equivalent conditions holds true 
 for every 
$\xi \in \Sd$: 
\begin{itemize}
\item[(a)] for every $\tau \in C^1(\R)$ with $-\tfrac{1}{2}\leq \tau \leq \tfrac{1}{2}$ and $0\leq \tau'\leq 1$, the partial derivative $\mathrm{D}_\xi\big(\tau(v\cdot \xi)\big)=\mathrm{D}\big(\tau(v\cdot \xi)\big)\cdot \xi$ belongs to $\mathcal{M}_b(\Omega)$, and for every Borel set $B\subset \Omega$ 
\begin{equation*}
\big|\mathrm{D}_\xi\big(\tau(v\cdot \xi)\big)\big|(B)\leq \lambda_v(B);
\end{equation*}
\item[(b)] $\widehat{v}^\xi_y \in BV_{\mathrm{loc}}(\Omega^\xi_y)$ for $\hd$-a.e.\ $y\in \Pi^\xi$, and for every Borel set $B\subset \Omega$ 
\begin{equation*}
\int_{\Pi^\xi} \Big(\big|\mathrm{D} {\widehat{v}}_y^\xi\big|\big(B^\xi_y\setminus J^1_{{\widehat{v}}^\xi_y}\big)+ \mathcal{H}^0\big(B^\xi_y\cap J^1_{{\widehat{v}}^\xi_y}\big)\Big)\dh(y)\leq \lambda_v(B)\,,
\end{equation*}
where
$J^1_{{\widehat{u}}^\xi_y}:=\left\{t\in J_{{\widehat{u}}^\xi_y} : |[{\widehat{u}}_y^\xi]|(t) \geq 1\right\}$.
\end{itemize} 
The function $v$ belongs to $GSBD(\Omega)$ if $v\in GBD(\Omega)$ and $\widehat{v}^\xi_y \in SBV_{\mathrm{loc}}(\Omega^\xi_y)$ for 
every
$\xi \in \Sd$ and for $\hd$-a.e.\ $y\in \Pi^\xi$.
\end{definition}
$GBD(\Omega)$ and $GSBD(\Omega)$ are vector spaces, as stated in \cite[Remark~4.6]{DM13}, and one has the inclusions $BD(\Omega)\subset GBD(\Omega)$, $SBD(\Omega)\subset GSBD(\Omega)$, which are in general strict (see \cite[Remark~4.5 and Example~12.3]{DM13}).
Every $v\in GBD(\Omega)$ has an \emph{approximate symmetric gradient} $e(v)\in L^1(\Omega;\Mdd)$ such that for every $\xi \in \Sd$ and $\hd$-a.e.\ $y\in\Pi^\xi$ there holds
\begin{equation}\label{3105171927}
 e(v)(y +  t\xi)   \xi   \cdot \xi= (\widehat{v}^\xi_y)'(t)  \quad\text{for } \mathcal{L}^1\text{-a.e.\ }   t \in  \Omega^\xi_y\,.
\end{equation}
We recall also that by the area formula (cf.\ e.g.\ \cite[(12.4)]{Sim84}; see \cite[Theorem~4.10]{AmbCosDM97} and \cite[Theorem~8.1]{DM13}) it follows that for any $\xi \in \Sd$
\begin{subequations}\label{2304191254}
\begin{equation}\label{2304191254-1}
 (J^\xi_v)\xy = J_{\widehat{v}\xy} \ \ \text{for $\mathcal{H}^{d-1}$-a.e.\ $y \in \Pi^\xi$,} \  \ \text{where} \ \  J_v^\xi:= \{ x \in J_v \colon [v](x) \cdot \xi \neq 0\}\,, 
\end{equation}
\begin{equation}
\int_{\Pi^\xi} \mathcal{H}^{0}(J_{\widehat{v}\xy}) \dh(y) = \int_{J_v^ \xi} |\nu_v \cdot \xi| \dh\,.
\end{equation}
\end{subequations}
 Moreover, there holds 
\begin{equation}\label{2304191637}
\hd(J_v \sm J_v^\xi)=0 \qquad\text{for }\hd\text{-a.e.\ }\xi \in \Sd\,.
\end{equation}
 Finally, if $\Omega$ has Lipschitz boundary, for each $v\in GBD(\Omega)$ the traces on $\partial \Omega$ are well defined in the sense that for $\mathcal{H}^{d-1}$-a.e.\ $x \in \partial\Omega$ there exists ${\rm tr}(v)(x) \in \R^d$ such that
$$\aplim\limits_{y \to x, \ y \in \Omega} v(y) = {\rm tr}(v)(x). $$
 For $1 < p < \infty$, the   space $GSBD^p(\Omega)$ is  defined by  
\begin{equation*}
GSBD^p(\Omega):=\{u\in GSBD(\Omega)\colon e(u)\in L^p(\Omega;\Mdd),\, \hd(J_u)<\infty\}\,.
\end{equation*}
We recall below two general density and compactness results in $GSBD^p$, from \cite{CC17} and \cite{CC18}.

\begin{theorem}[Density in $GSBD^p$]\label{thm:densityGSBD}
Let $\Omega\subset \Rd$ be an open,  bounded set with finite perimeter  and let $\partial \Omega$ be a $(d{-}1)$-rectifiable, 
$ p > 1$,   $\psi(t)= t \wedge 1$,
and $u\in GSBD^p(\Omega)$.   
Then there exist $u_n\in SBV^p(\Omega;\Rd)\cap   L^\infty(\Omega; \Rd)$ such that each
$J_{u_n}$ is closed in $\Omega$ and included in a finite union of closed connected pieces of $C^1$ hypersurfaces, $u_n\in   W^{1,\infty}(\Omega\setminus J_{u_n}; \Rd)$, and:
\begin{subequations}\label{eqs:main'}
\begin{align}
\int_\Omega \psi(|u_n - u|) \dx & \to 0 \,,\label{1main'}\\
 \|e(u_n) - e(u) \|_{L^p(\Omega) } & \to 0 \,, \label{2main'}\\
 \hd(J_{u_n}\triangle J_u)&\to 0 \,.\label{3main'}
 \end{align}
\end{subequations}
\end{theorem}

  We refer to \cite[Theorem 1.1]{CC17}.   In contrast to   \cite{CC17}, we use here the function $\psi(t):= t \wedge 1$ for simplicity. It is indeed easy to check that \cite[(1.1e)]{CC17} implies \eqref{1main'}.

\begin{theorem}[$GSBD^p$ compactness]\label{th: GSDBcompactness}
 Let $\Omega \subset \R$ be an open, bounded set,  and let $(u_n)_n \subset  GSBD^p(\Omega)$ be a sequence satisfying
$$ \sup\nolimits_{n\in \N} \big( \Vert e(u_n) \Vert_{L^p(\Omega)} + \mathcal{H}^{d-1}(J_{u_n})\big) < + \infty.$$
Then, there exists a subsequence, still denoted by $(u_n)_n$, such that the set  $A := \lbrace x\in \Omega\colon \, |u_n(x)| \to \infty \rbrace$ has finite perimeter, and  there exists  $u \in GSBD^p(\Omega)$ such that 
\begin{align}\label{eq: GSBD comp}
{\rm (i)} & \ \ u_n \to u \  \ \ \ \text{ in } L^0(\Omega \setminus A; \R^d), \notag \\ 
{\rm (ii)} & \ \ e(u_n) \rightharpoonup e(u) \ \ \ \text{ weakly  in } L^p(\Omega \setminus A; \Mdd),\notag \\
{\rm (iii)} & \ \ \liminf_{n \to \infty} \mathcal{H}^{d-1}(J_{u_n}) \ge \mathcal{H}^{d-1}(J_u \cup  (\partial^*A \cap\Omega)  ).
\end{align}
Moreover, for each $\Gamma\subset \Omega$ with $\mathcal{H}^{d-1}(\Gamma) < + \infty$, there holds
\begin{align}\label{eq: with Gamma}
\liminf_{n \to \infty} \mathcal{H}^{d-1}(J_{u_n}\setminus \Gamma) \ge \mathcal{H}^{d-1} \big( (J_u \cup  (\partial^*A \cap\Omega)  ) \setminus \Gamma \big)\,.
\end{align}
\end{theorem}

\begin{proof}
We refer to \cite{CC18}. The additional statement \eqref{eq: with Gamma} is proved, e.g., in  \cite[Theorem 2.5]{FriSol16}. 
\end{proof}

 Later, as a byproduct of our analysis, we will generalize the lower semicontinuity property \eqref{eq: GSBD comp}(iii) to anisotropic surface energies, see Corollary \ref{cor: GSDB-lsc}.

\subsection{$GSBD^p_\infty$ functions}\label{sec:prel4}
 Inspired by the previous compactness result, we now introduce a space of $GSBD^p$ functions which may also attain  a limit  value $\infty$. Define  $\bar{\R}^d := \R^d \cup \lbrace \infty \rbrace$. The sum on $\bar{\R}^d$ is given by $a + \infty = \infty$ for any $a \in \bar{\R}^d$.  There is a natural bijection between $\bar{\R}^d$ and $ \mathbb{S}^d  =\lbrace \xi \in\R^{d+1}:\,|\xi| =1 \rbrace$ given by the stereographic projection of $\mathbb{S}^{d}$ to $\bar{\R}^d$: for $\xi \neq e_{d+1}$, we define
$$\phi(\xi) = \frac{1}{1-\xi_{d+1}}(\xi_1,\ldots,\xi_d),$$
and let  $\phi(e_{d+1}) = \infty$. By $\psi: \bar{\R}^d\to \mathbb{S}^{d}$ we denote the inverse. Note that 
\begin{equation}\label{3005191230}
d_{\bar{\R}^d}(x,y):= |\psi(x) - \psi(y)|\quad \text{for }x,y \in \bar{\R}^d\end{equation} 
induces a bounded metric on $\bar{\R}^d$. We define
\begin{align}\label{eq: compact extension}
GSBD^p_\infty(\Omega) := \Big\{ &u \in L^0(\Omega;\bar{\R}^d)\colon \,   A^\infty_u  := \lbrace u = \infty \rbrace \text{ satisfies } \mathcal{H}^{d-1}(\partial^* A^\infty_u)< +\infty, \notag \\
&  \ \ \ \ \ \ \ \ \ \  \ \ \ \ \ \tilde{u}_t := u \chi_{\Omega \setminus A^\infty_u} + t \chi_{A^\infty_u} \in GSBD^p(\Omega) \ \text{ for all $t \in \R^d$} \Big\}. 
\end{align}
Symbolically, we will also write
$$u = u \chi_{\Omega \setminus A^\infty_u} + \infty \chi_{A^\infty_u}.$$
 Moreover, for any $u \in GSBD^p_\infty(\Omega)$, we  set  $e(u) = 0$ in $A^\infty_u$, and 
\begin{align}\label{eq: general jump}
J_u = J_{u \chi_{\Omega \setminus A^\infty_u}} \cup (\partial^*A^\infty_u \cap \Omega).
\end{align}
 In particular, we have
\begin{align}\label{eq:same}
e(u) = e(\tilde{u}_t) \ \ \text{$\mathcal{L}^d$-a.e.\ on  $\Omega$} \ \ \ \text{ and  } \ \ \ J_u = J_{\tilde{u}_t} \ \ \text{$\mathcal{H}^{d-1}$-a.e.} \ \ \ \text{ for almost all $t \in \R$}\,,
\end{align} where $\tilde{u}_t$ is the function from \eqref{eq: compact extension}.  Hereby, we also get a natural definition of a normal $\nu_u$ to the jump set $J_u$, and the slicing properties described in \eqref{3105171927}--\eqref{2304191637} still hold.   Finally, we point out that all definitions are consistent with the usual ones if $u \in GSBD^p(\Omega)$, i.e., if $A^\infty_u= \emptyset$.  Since $GSBD^p(\Omega)$ is a vector space, we observe that the sum of two functions in $GSBD^p_\infty(\Omega)$ lies again in this space.

 A  metric on $GSBD^p_\infty(\Omega)$  is given by
\begin{equation}\label{eq:metricd}
\GGG\bar{d}(u,v) \EEE { := } \int_\Omega  d_{\bar{\R}^d}(u(x),v(x)) \, \dx\,,
\end{equation} 
 where $d_{\bar{\R}^d}$ is the distance in \eqref{3005191230}.  We now state  compactness properties in $GSBD^p_\infty(\Omega)$.
\begin{lemma}[Compactness in $GSBD^p_\infty$]\label{eq: corollary-comp}
For $L>0$ and $\Gamma \subset \Omega$ with $\mathcal{H}^{d-1}(\Gamma)<+\infty$,   we introduce the sets
\begin{align}\label{eq: XL}
X_L(\Omega) & = \big\{         v \in GSBD^p_\infty(\Omega)\colon   \mathcal{H}^{d-1}(J_v ) \le L, \ \ \Vert e(v) \Vert_{L^p(\Omega)}  \le 1    \big\}\,,\notag\\
X_\Gamma(\Omega) & = \big\{         v \in GSBD^p_\infty(\Omega)\colon   \mathcal{H}^{d-1}(J_v \setminus \Gamma ) = 0, \ \ \Vert e(v) \Vert_{L^p(\Omega)}  \le 1    \big\}\,.
\end{align}
Then the sets $X_L(\Omega), X_\Gamma(\Omega)  \subset GSBD^p_\infty(\Omega)$ are compact with respect to the metric $\GGG\bar{d} \EEE$.
\end{lemma}

\begin{proof}
For $X_L(\Omega)$, the statement follows from  Theorem \ref{th: GSDBcompactness} and the definitions \eqref{eq: compact extension}--\eqref{eq: general jump}:  in  fact, given a sequence $(u^n)_n \subset X_L(\Omega)$, we consider a sequence $(\tilde{u}_{t_n}^n)_n \subset GSBD^p(\Omega)$ as in \eqref{eq: compact extension}, for  suitable   $(t_n)_n \subset \R^d$ with $|t_n| \to \infty$. This implies
\begin{align}\label{eq: t-def}
\GGG\bar{d} \EEE(u^n,\tilde{u}_{t_n}^n) \to 0 \text{ as } n \to \infty.
\end{align}
 Then, by Theorem \ref{th: GSDBcompactness} there exists $v \in  GSBD^p(\Omega)$ and $A = \lbrace x \in \Omega\colon \, |\tilde{u}_{t_n}^n(x)| \to \infty \rbrace$ such that $\tilde{u}_{t_n}^n \to v$ in $L^0(\Omega \setminus A;\R^d)$. We define $u = v\chi_{\Omega \setminus A} + \infty \chi_A \in GSBD^p_\infty(\Omega)$. By \eqref{eq: GSBD comp}(ii),(iii) and \eqref{eq: general jump} we get that $u \in X_L(\Omega)$.  We observe that $\GGG\bar{d} \EEE(\tilde{u}_{t_n}^n,u) \to 0$ and then by \eqref{eq: t-def} also $\GGG\bar{d} \EEE(u^n,u) \to 0$.
 
 The proof for the set $X_\Gamma(\Omega)$ is similar, where we additionally use \eqref{eq: with Gamma} to ensure that $\mathcal{H}^{d-1}(J_u \setminus \Gamma) = 0 $.         
\end{proof}
 In the next  sections,  we will use the following  notation. We say that a sequence $(u_n)_n \subset GSBD^p_\infty(\Omega)$ \emph{converges weakly} to $u \in GSBD^p_\infty(\Omega)$ if 
\begin{align}\label{eq: weak gsbd convergence}
 \sup\nolimits_{n\in \N} \big( \Vert e(u_n) \Vert_{L^p(\Omega)} + \mathcal{H}^{d-1}(J_{u_n})\big) < + \infty \ \ \ \text{and} \ \ \    \GGG\bar{d} \EEE(u_n,u) \to 0 \text{ for } n \to \infty\,.
 \end{align}

\GGG We close this subsection by pointing out that a similar space has been introduced in \cite{CagColDePMag17}, in the case of scalar valued functions attaining extended real values: the space $GBV_*(\R^{d})$ was defined by $f \colon \R^{d} \to \R \cup \{ \pm \infty\} \in GBV_*(\R^{d})$ if and only if $(-M \vee  f)\wedge M \in BV_{\mathrm{loc}}(\R^d)$ for every $M>0$. In \cite[Proposition~3.1]{CagColDePMag17} it is shown that $f \in GBV_*(\R^{d})$ if and only if its epigraph is of locally finite perimeter in $\R^{d+1}$.
Our definition is based on the structure of the set where functions attain infinite values, rather than employing (the analog of) truncations. In fact, the latter is not meaningful if one controls only symmetric gradients. \EEE


\section{The $\sps$-convergence of sets}\label{sec:sigmap}

 This section is devoted to the introduction of a convergence of sets in the framework of $GSBD^p$ functions 
  analogous  to  $\sigma^p$-convergence defined in \cite{DMFraToa02} for the space  
 \GGG $SBV^p$. \EEE   
  This type of convergence of sets will be useful to study the lower limits in the relaxation results in  Subsection~\ref{subsec:RelG}  and the compactness properties in  Subsection~\ref{subsec:compactness}.   We believe that this notion may be of independent interest and is potentially  helpful to study also other  problems such as quasistatic crack evolution. 
 
We start by recalling briefly the definition of $\sigma^p$-convergence in \cite{DMFraToa02}: a sequence of sets $(\Gamma_n)_n$ $\sigma^p$-converges to $\Gamma$ if  (i) for any sequence $(u_n)_n$ converging to $u$ \GGG weakly \EEE in \GGG $SBV^p$  \EEE   with $J_{u_n} \subset \Gamma_n$, it holds $J_u \subset \Gamma$ and (ii)  there  exists a \GGG $SBV^p$ \EEE function whose jump is $\Gamma$, which is approximated \GGG (in the sense of weak convergence in $SBV^p$) \EEE by \GGG $SBV^p$ \EEE functions with jump included in $\Gamma_n$. \GGG (Here, weak convergence in $SBV^p$ means that $\sup_n \big(\|u_n\|_{L^\infty} + \hd(J_{u_n})\big) < +\infty$,   $\nabla u_n \weak \nabla u$ in $L^p$, and $u_n \to u$ almost everywhere.) \EEE For sequences of sets $(\Gamma_n)_n$ with $\sup_n \hd(\Gamma_n) < +\infty$, a compactness result
 with respect to  $\sigma^p$-convergence is obtained by means of Ambrosio's compactness theorem \cite{Amb90GSBV}, see \cite[Theorem~4.7]{DMFraToa02} and \cite[Theorem~3.3]{ChaSol07}.    We refer to \cite[Section~4.1]{DMFraToa02} for a general motivation to consider such a kind of convergence.

We now introduce the notion of $\sps$-convergence.   In the following, we use the notation $A \tilde{\subset} B$ if $\mathcal{H}^{d-1}(A \setminus B) = 0$ and $A \tilde{=} B$ if  $A \tilde{\subset} B$ and  $B \tilde{\subset} A$.   As before, by  $(G)^1$ we denote the set of points with density $1$ for $G \subset \R^d$. Recall also the definition and properties of $GSBD^p_\infty$ in Subsection~\ref{sec:prel4}, in particular \eqref{eq: weak gsbd convergence}.

\begin{definition}[$\sps$-convergence]\label{def:spsconv}
Let $ U  \subset \R^d$  be  open, let  $U' \supset U$  be  open  with $\mathcal{L}^d(U' \setminus U)>0$, and let   $p \in (1,\infty)$.   We say that a sequence $(\Gamma_n)_n  \subset  \overline{U}\cap U'$  with $\sup_{n\in \N} \mathcal{H}^{d-1}(\Gamma_n) <+\infty$  $\sps$-converges to a pair $(\Gamma, G_\infty)$ satisfying  $\Gamma \subset \overline{U} \cap U'$   together with 
\begin{align}\label{eq: limit-prop}
\mathcal{H}^{d-1}(\Gamma) < +\infty,  \ \ G_\infty \subset U, \ \ \partial^*G_\infty \cap U' \, \tilde{\subset} \, \Gamma, \ \ \text{ and } \ \ \Gamma \cap (G_\infty)^1 = \emptyset
\end{align} 
if there holds:  

(i)  for any sequence $(v_n)_n \subset GSBD^p_\infty(U')$ with $J_{v_n} \tilde{\subset} \Gamma_n$ and $v_n = 0$ in $U' \setminus U$, if a subsequence $(v_{n_k})_k$ converges weakly in $GSBD^p_\infty(U')$ to $v \in GSBD^p_\infty(U')$, then 
\CCC $\mathcal{L}^d(\lbrace v = \infty \rbrace \setminus G_\infty) = 0$ and \EEE 
$J_v \setminus \Gamma \tilde{\subset} ( G_\infty)^1$.  

(ii) there exists a function $v \in GSBD^p_\infty(U')$  and a sequence  $(v_n)_n \subset GSBD^p_\infty(U')$ converging weakly in $GSBD^p_\infty(U')$ to $v$ such that $ J_{v_n}  \tilde{\subset} \Gamma_n$, $v_n = 0$ on $U' \setminus U$ for all $n \in \N$,  $J_v  \tilde{=} \Gamma$,  and  $\lbrace v = \infty \rbrace = G_\infty$.

\end{definition}

 Our definition deviates from  $\sigma^p$-convergence  in the sense that, besides a limiting $(d{-}1)$-rectifiable set $\Gamma$, there exists also a set of finite perimeter $G_\infty$.  Roughly speaking, in view of  $\partial^* G_\infty \subset \Gamma \cup \partial U$,   this set  represents the parts which are completely disconnected by $\Gamma$ from the rest of the domain. The  behavior  of functions cannot be controlled there,  i.e., a sequence $(v_n)_n$ as in (i) may converge to infinity on this set or exhibit further cracks.  \GGG
 In the  framework of $GSBV^p$ functions in \cite{DMFraToa02}, it was possible to avoid such a phenomenon  by  working with truncations which allows to resort to $SBV^p$ functions with uniform $L^\infty$-bounds. 
  \EEE  In $GSBD$, however, this truncation technique is not available and we therefore need a more general definition involving  the space $GSBD^p_\infty$ and   a set of finite perimeter $G_\infty$.


 Moreover, due to the presence of the set $G_\infty$, in contrast to the definition of $\sigma^p$-convergence, it is essential to control the functions in a set $U' \setminus U$:  the assumptions $\mathcal{L}^d(U' \setminus U)>0$ and  $ G_\infty  \subset U$ are crucial since otherwise, if $U' = U$, conditions (i) and (ii) would always be trivially satisfied with $G_\infty  = U$ and $\Gamma = \emptyset$. 
 
 \CCC We briefly note that the pair $(\Gamma,G_\infty)$ is unique. In fact, if there were two different limits $(\Gamma^1,G^1_\infty)$ and $(\Gamma^2,G^2_\infty)$, we could choose functions $v^1$ and $v^2$ with  $J_{v^1}  \tilde{=} \Gamma^1$, $J_{v^2}  \tilde{=} \Gamma^2$, $\lbrace v^1 = \infty \rbrace = G^1_\infty$, and $\lbrace v^2 = \infty \rbrace = G^2_\infty$, as well as corresponding sequences  $(v^1_n)_n$ and  $(v^2_n)_n$ as in (ii). But then (i) implies   $\Gamma^1 \setminus \Gamma^2 \tilde{\subset} ( G^2_\infty)^1$, $\Gamma^2 \setminus \Gamma^1 \tilde{\subset} ( G^1_\infty)^1$, as well as $G_\infty^1 \subset G^2_\infty$ and $G_\infty^2 \subset G^1_\infty$. As $\Gamma^i\cap (G_\infty^i)^1 = \emptyset$ for $i=1,2$, this shows $(\Gamma^1,G^1_\infty) = (\Gamma^2,G^2_\infty)$.  In a similar way,  if a sequence $(\Gamma_n)_n$ $\sps$-converges to $(\Gamma, G_\infty)$, then every subsequence $\sps$-converges to the same limit. \EEE

\CCC 
Let us mention that, in our application in Section \ref{sec:GGG}, the sets $\Gamma_n$ will be graphs of functions. In this setting, we will be able to ensure that $G_\infty = \emptyset$, see \eqref{eq: G is empty} below, and thus a simplification of Definition \ref{def:spsconv} only in terms of $\Gamma$ without $G_\infty$ is in principle possible. We believe, however, that the notion of $\sps$-convergence  may be of independent interest and is potentially  helpful to study also other  problems such as quasistatic crack evolution in linear elasticity \cite{FriSol16}, where $G_\infty = \emptyset$ cannot be expected. Therefore, we prefer to treat this more general definition here. 
\EEE

   The main goal of this section is to prove the following compactness result for $\sps$-convergence.

\begin{theorem}[Compactness of $\sps$-convergence]\label{thm:compSps}
Let $U \subset \R^d$  be  open, let $U' \supset U$  be open  with $\mathcal{L}^d(U' \setminus U)>0$, and let   $p \in (1,\infty)$.  Then, every sequence $(\Gamma_n)_n \subset U$ with $\sup_n \mathcal{H}^{d-1}(\Gamma_n) < + \infty$ has a $\sps$-convergent subsequence  with limit  $(\Gamma,G_\infty)$  satisfying $\hd(\Gamma) \leq \liminf_{n\to\infty} \hd(\Gamma_n)$.  
\end{theorem}

 For the proof, we need the following two auxiliary results.

\begin{lemma}\label{lemma: good function choice}
Let $(v_i)_i \subset GSBD^p(\Omega)$ such that  $\|e(v_i)\|_{L^p(\Omega)} \leq 1$ for all $i$ and  $\Gamma:= \bigcup_{i=1}^\infty J_{v_i}$ satisfies $\mathcal{H}^{d-1}(\Gamma) < + \infty$.  Then there exist constants $c_i>0$, $i \in \N$, such that  $\sum_{i=1}^\infty c_i \le 1$ and $v:= \sum\nolimits_{i=1}^\infty c_i v_i \in GSBD^p(\Omega)$ satisfies  $J_v \tilde{=} \GGG \Gamma \EEE$.
\end{lemma}

\begin{lemma}\label{lemma: theta}
Let $V \subset  \R^d  $ and suppose that two sequences $(u_n)_n,(v_n)_n \in L^0(V;\bar{\R}^d)$ satisfy $|u_n|, |v_n| \to \infty$ on $V$.   Then for $\mathcal{L}^1$-a.e. $\theta \subset(0,1)$ there holds 
$$|(1-\theta) u_n(x) + \theta v_n(x)| \to \infty\ \ \  \text{for a.e.\ $x \in V$}. $$  
\end{lemma}

We postpone the proof of the lemmas  and proceed with the proof of Theorem \ref{thm:compSps}.

\begin{proof}[Proof of Theorem \ref{thm:compSps}]
For  $\Gamma \subset U$ with $\mathcal{H}^{d-1}(\Gamma) < +\infty$ we define
$$X(\Gamma) = \big\{         v \in GSBD^p_\infty(U')\colon J_v \tilde{\subset}  \Gamma, \ \ \Vert e(v) \Vert_{L^p(U')} \le 1, \ \ \       v = 0 \text{ on } U' \setminus U  \big\}. $$
The set $X(\Gamma)$ is compact with respect to the metric \GGG $\bar{d}$ \EEE  introduced in \eqref{eq:metricd}.   This follows from Lemma \ref{eq: corollary-comp}  and  the fact that $\lbrace v \in L^0( U'  ;\bar{\R}^d)\colon \, v = 0 \text{ on } U' \setminus U \rbrace$ is closed with respect to \GGG $\bar{d}$ \EEE. 

Since we treat any $v \in GSBD^p_\infty(U')$ as a constant function in the exceptional set  $A^\infty_v$  (namely we have no jump and $e(v)=0$ therein,  see \eqref{eq:same}),  we get that the convex combination of two $v, v'\in X(\Gamma)$ is still in $X(\Gamma)$. 
(Recall that  the sum on $\bar{\R}^d$ is given by $a + \infty = \infty$ for any $a \in \bar{\R}^d$.) 
\vspace{0.1em}
 
\noindent  \emph{Step 1: Identification of a compact  and convex  subset.} 
Consider  $(\Gamma_n)_n \subset U$ with $\sup_{n }\mathcal{H}^{d-1}(\Gamma_n) < + \infty$. Fix $\delta>0$ small  and define
\begin{align}\label{eq: delta error}
L:= \liminf_{n \to \infty} \mathcal{H}^{d-1}(\Gamma_n) +\delta\,.
\end{align}
\GGG By \eqref{eq: delta error} we have that, up to a subsequence (not relabeled), each $X(\Gamma_n)$ is  contained in $X_L(U')$ defined in \eqref{eq: XL}. Moreover, as noticed above, $X_L(U')$ and each $X(\Gamma_n)$ are compact with respect to $\bar{d}$. \EEE
\GGG Since the class of non-empty compact subsets of a compact metric space $(M,d_M)$ is itself compact with respect to the Hausdorff distance induced by $d_M$,  \EEE a subsequence (not relabeled) of $(X(\Gamma_n))_n$ converges in the Hausdorff sense (with the Hausdorff distance induced by \GGG $\bar{d}$) \EEE to a compact set $K \subset X_L(U')$. 


 We first observe that the function identical to zero lies in $K$.  We now show that $K$ is convex. Choose $u,v \in K$ and $\theta \in (0,1)$. We need to check that $w := (1-\theta)u + \theta v \in K$. Observe that $A^\infty_{w} = A^\infty_u \cup A^\infty_v$, where $A^\infty_u$, $A^\infty_v$,  and $A^\infty_w$  are the exceptional sets given in \eqref{eq: compact extension}.  There exist sequences $(u_n)_n$ and $(v_n)_n$  with $u_n,v_n \in X(\Gamma_n)$ such that $\GGG\bar{d} \EEE(u_n,u) \to 0$ and $\GGG\bar{d} \EEE(v_n,v) \to 0$. In particular, note that $|u_n| \to \infty$ on $A^\infty_u$ and $|v_n| \to \infty$ on $A^\infty_v$. By Lemma \ref{lemma: theta} and a diagonal argument we can choose  $(\theta_n)_n\subset (0,1)$ with $\theta_n \to \theta$ such that
 $w_n := (1-\theta_n)u_n + \theta_n v_n$ satisfies $|w_n| \to \infty$ on $A^\infty_u \cap A^\infty_v$.  As clearly  $|w_n|\to \infty$ on $A^\infty_u \triangle A^\infty_v$ and   $(1-\theta_n)u_n + \theta_n v_n \to (1-\theta)u + \theta v $ in measure on $U' \setminus (A^\infty_u \cup A^\infty_v)$, we get $\GGG\bar{d} \EEE(w_n,w) \to 0$. Since $X(\Gamma_n)$ is convex, there holds $w_n \in X(\Gamma_n)$. Then $\GGG\bar{d} \EEE(w_n,w) \to 0$ implies $w \in K$,  as desired. 
\vspace{0.11em}

\CCC
\noindent \emph{Step 2: Choice of dense subset.} \EEE Since $K$ is compact with respect to the metric \GGG $\bar{d}$ \CCC (so, in particular, $K$ is separable), \EEE we can choose a countable set $(y_i)_i \subset GSBD^p_\infty(U')$ with $y_i = 0$ on $U' \setminus U$  which is \GGG $\bar{d}$\EEE-dense in $K$. \CCC We now show that this countable set can be chosen with the additional property
\begin{align}\label{eq: additional property}
\mathcal{L}^d\Big( A^\infty_v \setminus \bigcup\nolimits_i  A^\infty_{y_i}  \Big)  = 0 \quad \quad \text{for all $v \in K$,}
\end{align}
 where we again denote by $A^\infty_{y_i}$ and $A^\infty_v$ the sets where the functions attain the value $\infty$. In fact, fix an arbitrary countable and $\bar{d}$-dense set $(y_i)_i$ in $K$, and let $\eta >0$. After adding a finite number (smaller than $\mathcal{L}^d(U)/\eta$) of functions of $K$ to this collection, we obtain a countable $\bar{d}$-dense family $(y^\eta_i)_i$ such that 
$$
\mathcal{L}^d\Big( A^\infty_v \setminus \bigcup\nolimits_i  A^\infty_{y^\eta_i}  \Big)  \le \eta \quad \quad \text{for all $v \in K$.}
$$
Then, we obtain the desired countable set by taking the union of $(y^{1/k}_i)_i$ for $k \in \N$.   \EEE

 \noindent \emph{Step 3: Definition of $\Gamma$ and $G_\infty$.}   Fix $v,v' \in K$. Since $\lbrace x \in J_v \setminus \partial^* A^\infty_v \colon [v](x) =t \rbrace$ has negligible $\mathcal{H}^{d-1}$-measure up to a countable set of points $t$, we find  some $\theta \in (0,1)$ such that $w:= \theta v + (1-\theta) v'$ satisfies 
\begin{align}\label{eq: jump-w}
J_w \tilde{\subset} J_v \cup J_{v'}, \ \ \ \ \ \ (J_v \cup J_{v'}) \setminus J_w \tilde{\subset} (A^\infty_v\cup A^\infty_{v'})^1.
\end{align}
Here, we particularly point out that $\lbrace w = \infty\rbrace = A^\infty_v\cup A^\infty_{v'}$ and that $\partial^* (A^\infty_v\cup A^\infty_{v'})  \cap U'  \tilde{\subset} J_w$ by \eqref{eq: general jump}. Note that $w \in K$ since $K$ is convex.  Since $w \in K \subset X_L(U')$, \eqref{eq: jump-w} implies 
$$\mathcal{H}^{d-1}( (J_v \cup J_{v'}) \setminus (A^\infty_v\cup A^\infty_{v'})^1)   \le \mathcal{H}^{d-1}(J_w)   \le L, \ \ \ \ \ \ \mathcal{H}^{d-1}(\partial^* (A^\infty_v\cup A^\infty_{v'}) \cap U') \le \mathcal{H}^{d-1}(J_w) \le L\,.$$
\CCC Let $(y_i)_i \subset GSBD^p_\infty(U')$ with $y_i = 0$ on $U' \setminus U$  be the countable and $\bar{d}$-dense subset of $K$ satisfying \eqref{eq: additional property} that we defined in Step 2. \EEE By the above convexity argument, we find 
\begin{align}\label{eq: LLL}
\mathcal{H}^{d-1}\Big(\bigcup\nolimits_{i=1}^k J_{y_i} \setminus \big(\bigcup\nolimits_{i=1}^k A_i \big)^1\Big)\le L, \ \ \ \ \ \ \ \ \mathcal{H}^{d-1}\Big(\partial^*\big(\bigcup\nolimits_{i=1}^k A_i \big) \cap U'\Big)\le L
\end{align}
for all $k \in \N$, where 
\[
A_i :=  A^\infty_{y_i} =  \lbrace y_i = \infty \rbrace\,.
\] 
We define 
\begin{align}\label{eq: G-deffi}
G_\infty := \bigcup\nolimits_{i=1}^\infty A_i\,.
\end{align}
By passing to the limit $k \to \infty$  in \eqref{eq: LLL}, we  get $\mathcal{H}^{d-1}(\partial^* G_\infty \cap U') \le L$  and  $\mathcal{H}^{d-1}(\bigcup\nolimits_{i=1}^k J_{y_i} \setminus (G_\infty )^1 )\le L $ for all $k \in \N$. Passing again to the limit $k \to \infty$, and setting  
\begin{equation}\label{eq: Gamma-def}
\Gamma := \bigcup\nolimits_{i=1}^\infty J_{y_i} \setminus ( G_\infty )^1
\end{equation} 
we get     $\mathcal{H}^{d-1}(\Gamma) \le L$.  Notice that $\Gamma\cap(G_\infty )^1 = \emptyset$ by definition. Moreover, the fact that $y_i = 0$ on $U' \setminus U$ for all $i\in \N$ implies both that $G_\infty \subset U$ and that $\Gamma \subset \ove{U} \cap U'$.
By \eqref{eq: delta error} and the arbitrariness of $\delta$ we get $\hd(\Gamma) \leq \liminf_{n\to \infty} \hd(\Gamma_n)$.   Since $\partial^* A_i \cap U' \tilde{\subset} J_{y_i}$ for all $i \in \N$ by \eqref{eq: general jump}, we also get $\Gamma \, \tilde{\supset} \, \partial^* G_\infty  \cap U'$.    Thus, \eqref{eq: limit-prop} is satisfied.   

We now claim that  for each $v \in K$ there holds
\begin{align}\label{eq: subset outside G}
{\CCC \mathcal{L}^d(\lbrace v = \infty \rbrace \setminus G_\infty) = 0 \quad \quad \EEE \text{and} \quad \quad    J_v \setminus \Gamma \tilde{\subset} (G_\infty )^1.}
\end{align}
 Indeed, \CCC the first property follows from \eqref{eq: additional property} and \eqref{eq: G-deffi}. To see the second, we note that, \EEE for any fixed $v\in K$, there is a sequence $(y_k)_k=(y_{i_k})_k$ with $\GGG\bar{d} \EEE(y_k, v) \to 0 $, by the density of $(y_i)_i$. Consider the functions $ \tilde{v}_k:=  y_k (1-\chi_{G_\infty})$ that \GGG $\bar{d}$\EEE-converge to $\tilde{v}:=v (1-\chi_{G_\infty})$: since $J_{ \tilde{v}_k  } \tilde{\subset} \Gamma$ for any $k$ (we employ  \eqref{eq: Gamma-def} and that  $\partial^* G_\infty \cap U' \tilde{\subset} \Gamma$), 
 the fact that $X(\Gamma)$ is closed gives that  $J_{\tilde{v}} \tilde{\subset} \Gamma$. This implies  \eqref{eq: subset outside G}. 
\vspace{0.1em}

\noindent \emph{Step 4: Proof of properties (i) and (ii).}    We first show (i). Given a sequence $(v_n)_n \subset GSBD^p_\infty(U')$ with $J_{v_n} \tilde{\subset} \Gamma_n$ and $v_n = 0$ on $U' \setminus U$, and  a subsequence $(v_{n_k})_k$ that converges weakly in $GSBD^p_\infty(U')$ to $v$,  we clearly get $v \in K$ by Hausdorff convergence of $X(\Gamma_n) \to K$.  (More precisely, consider $\lambda  v_{n_k}$ and $\lambda v$ for $\lambda >0$ such that $\Vert e(\lambda v_{n_k}) \Vert_{L^p(U')} \le 1$ for all $k$.)  By \eqref{eq: subset outside G}, this implies  \CCC $\mathcal{L}^d(\lbrace v = \infty \rbrace \setminus G_\infty) = 0$ and \EEE  $J_v \setminus \Gamma \tilde{\subset} (G_\infty )^1$. This shows (i).

We now address (ii). Recalling the  choice of the sequence $(y_i)_i \subset K$,  for each $i \in \N$, we choose $\tilde{y}_i = y_i \chi_{U' \setminus  G_\infty } + t_i \chi_{ G_\infty } \in GSBD^p(U')$   for some $t_i \in \R^d$ such that $J_{\tilde{y}_i} \tilde{=} J_{y_i}  \setminus ( G_\infty )^1$.   (Almost every $t_i$ works. Note that the function indeed lies in $GSBD^p( U  )$,  see \eqref{eq: compact extension} and \eqref{eq: G-deffi}.) In view of \eqref{eq: Gamma-def},  we also observe that $\bigcup_i J_{\tilde{y}_i} = \Gamma$.

By Lemma \ref{lemma: good function choice} (recall $(y_i)_i \subset K \subset X_L(\Omega)$)  we get a function $\tilde{v}= \sum_{i=1}^\infty c_i \tilde{y}_i \in GSBD^p( U'  )$ such that $J_{\tilde{v}}  \tilde{=}  \Gamma$, where  $\sum_{i=1}^\infty c_i \le 1$.  We also define $v = \tilde{v} \chi_{ U'  \setminus  G_\infty } + \infty \chi_{ G_\infty}  \in  GSBD^p_\infty(U')$.  Note that  $\lbrace v = \infty \rbrace = G_\infty$ and $J_v \tilde{=} \Gamma$  since $\Gamma \cap (G_\infty)^1 = \emptyset$ and $\partial^* G_\infty \cap U' \subset \Gamma$.    Then by the convexity of $K$, we find $z_k: = \sum_{i=1}^k c_i y_i \in K$. (Here we also use that the function identical to zero lies in $K$.)  As $G_\infty= \bigcup_{i=1}^\infty A_i$, we obtain   $\GGG\bar{d} \EEE(z_k,v) \to 0$ for $k \to \infty$. Thus, also $v \in K$ since $K$ is compact. As $X(\Gamma_n)$ converges to $K$ in Hausdorff convergence, we find a sequence $(v_n)_n \subset GSBD^p_\infty( U'  )$ with $J_{v_n} \tilde{\subset} \Gamma_n$, $v_n = 0$ on $U' \setminus U$,  and $\GGG\bar{d} \EEE(v_n,v) \to 0$. This shows (ii). 
\end{proof}

 Next, we  prove Lemma \ref{lemma: good function choice}. To this end,  we will need the following measure-theoretical result. (See \cite[Lemma 4.1, 4.2]{Fri17ARMA} and note that the statement in fact holds in arbitrary space dimensions for measurable functions.) 
\begin{lemma}
Let $\Omega\subset \R^d$ with $\mathcal L^d(\Omega)<\infty$,  and $N \in \mathbb{N}$.  Then for every sequence $(u_n)_n \subset L^0(\Omega;\R^N )$ with
\begin{align}\label{eq: inclusion condition}
\mathcal L^d\left(\bigcap\nolimits_{n \in \N} \bigcup\nolimits_{m \ge n} \lbrace |u_m - u_n| > 1 \rbrace\right)=0
\end{align}
there exist a  subsequence (not relabeled) and an increasing concave function  $\psi:  [0,\infty)  \to  [0,\infty) $  with 
$
\lim_{t\to \infty}\psi(t)=+\infty
$
such that   $$\sup_{n \ge 1} \int_{\Omega}\psi(|u_n|)\, \dx < + \infty.$$ 
\end{lemma}

\begin{proof}[Proof of Lemma \ref{lemma: good function choice}]
Let $(v_i)_i \subset  GSBD^p(\Omega)$ be given  satisfying the assumptions of the lemma. 
 First, choose $0 < d_i < 2^{-i}$ such that
\begin{align}\label{eq: small volume}
\mathcal{L}^d\Big( \Big\{ |{v}_i| \ge \frac{1}{2^i d_i} \Big\} \Big) \le 2^{-i},\ \ \ \ \ \ \ \ \mathcal{H}^{d-1}\Big( \Big\{x \in J_{v_i}\colon |[v_i](x)| \ge \frac{1}{d_i}      \Big\} \Big) \le 2^{-i}. 
\end{align}
Our goal is to select constants $c_i \in (0,d_i)$ such that the function ${v}:= \sum\nolimits_{i=1}^\infty c_i {v}_i$ lies in $GSBD^p(\Omega)$ and satisfies $J_v \, \tilde{=} \, \Gamma := \bigcup_{i=1}^\infty J_{v_i}$.  We proceed in two steps: we first show that for each choice $c_i \in (0,d_i)$ the function  ${v}= \sum\nolimits_{i=1}^\infty c_i {v}_i$ lies indeed in $GSBD^p(\Omega)$ (Step 1). Afterwards, we prove that for a specific choice there holds  $J_{{v}} \tilde{=} \Gamma$.

\emph{Step 1.} Given  $c_i \in (0,d_i)$,  we define $u_k = \sum_{i=1}^k c_iv_i$. Fix \GGG $m \ge n+1$. \EEE We observe that 
\begin{align*}
\lbrace |u_m-u_n|>1 \rbrace =  \Big\{  \big|\sum\nolimits_{i=n+1}^m  c_i v_i \big|>1 \Big\} \subset \bigcup\nolimits_{i=n+1}^m  \Big\{ |c_i{v}_i| \ge 2^{-i} \Big\} \subset \bigcup\nolimits_{i=n+1}^m  \Big\{ |{v}_i| \ge \frac{1}{2^i d_i} \Big\}.
\end{align*}
By passing to the limit $m \to \infty$ and by using  \eqref{eq: small volume} we get    
\begin{align*}
\mathcal{L}^d\Big(\bigcup\nolimits_{m \ge n}\lbrace |u_m-u_n|>1 \rbrace \Big) \le  \sum\nolimits_{i=n+1}^\infty \mathcal{L}^d \Big(\Big\{ |{v}_i| \ge \frac{1}{2^i d_i} \Big\} \Big) \le  \sum\nolimits_{i=n+1}^\infty 2^{-i} = 2^{-n}.
\end{align*}
This shows that the sequence $(u_k)_k$ satisfies \eqref{eq: inclusion condition}, and therefore there exist a  subsequence (not relabeled) and an increasing continuous function  $\psi:  [0,\infty)  \to  [0,\infty) $ with 
$
\lim_{t\to \infty}\psi(t)=+\infty
$
such that   $\sup_{k \ge 1} \int_{\Omega}\psi(|u_k|)\, \dx < + \infty$.  Recalling also that $\|e(v_i)\|_{L^p(\Omega)}\leq 1$ for all $i$ 
and $\mathcal{H}^{d-1}(\Gamma) <+\infty$, we are now in the position to apply the $GSBD^p$-compactness result \cite[Theorem~11.3]{DM13} (alternatively, one could apply Theorem~\ref{th: GSDBcompactness} and observe that the limit $v$ satisfies $\mathcal{L}^d(\lbrace v = \infty \rbrace)=0$), to get  that the function  
$v = \sum_{i=1}^\infty c_iv_i$  
lies in $GSBD^p(\Omega)$.  For later purposes, we note that by \eqref{eq: with Gamma}  (which holds also in addition to \cite[Theorem~11.3]{DM13})  we obtain 
\begin{align}\label{eq:  the first inclusion}
J_v \tilde{\subset} \bigcup\nolimits_{i=1}^\infty J_{v_i} = \Gamma.
\end{align}
 This concludes Step 1 of the proof.

\emph{Step 2.} We define the constants $c_i \in (0,d_i)$ iteratively by following the arguments in \cite[Lemma 4.5]{DMFraToa02}. Suppose that $(c_i)_{i=1}^k$, and a decreasing sequence $(\eps_i)^k_{i=1} \subset (0,1)$  have been chosen such that the functions $u_j = \sum_{i=1}^j c_iv_i$, $ 1 \le  j \le k$, satisfy 
\begin{align}\label{eq: iterative conditions}
{\rm (i)} & \ \ J_{u_j} \tilde{=} \bigcup\nolimits_{i=1}^{j} J_{v_i},\notag\\
{\rm (ii)} &  \ \  \mathcal{H}^{d-1}(\lbrace x\in J_{u_j}\colon | [u_j](x) |  \le \eps_j \rbrace) \le 2^{-j},
\end{align}
and, for $2\le j\le k$,  there holds
\begin{align}\label{eq: iterative conditions2}
c_{j} \le \eps_{j-1} d_j2^{-j-1}.
\end{align}
(Note that in the first step we can simply set $c_1 = 1/4$ and $0 < \eps_1< 1$ such that \eqref{eq: iterative conditions}(ii) holds.) 

We pass to the step $k + 1$ as follows. Note that there is a set $N_0  \subset \R  $ of negligible measure such that for all  $t \in \R \setminus N_0$ there holds $J_{u_k + t v_{k+1}} \tilde{=} J_{u_k} \cup J_{v_{k+1}}$. We choose $c_{k+1}  \in \R \setminus N_0$ such that additionally  $c_{k+1} \le \eps_k  d_{k+1} 2^{-k-2}$. Then \eqref{eq: iterative conditions}(i) and \eqref{eq: iterative conditions2} hold. We can then choose $\eps_{k+1} \le \eps_k$ such that also    \eqref{eq: iterative conditions}(ii) is satisfied. 

We proceed in this way for all $k \in \N$. Let us now introduce the sets 
\begin{align}\label{eq: Ek-def}
E_k = \bigcup\nolimits_{m \ge k} \lbrace x \in J_{u_m}\colon \  | [u_m](x) |  \le \eps_m \rbrace, \ \ \  F_k = \bigcup\nolimits_{m \ge k} \lbrace x \in J_{v_m}\colon \ |[v_m](x)| > 1/d_m\rbrace\,.
\end{align}
Note  by \eqref{eq: small volume} and \eqref{eq: iterative conditions}(ii) that  
\begin{align}\label{eq: EkFk}
 \mathcal{H}^{d-1}  (E_k \cup F_k) \le 2 \sum\nolimits_{m \ge k} 2^{-m} = 2^{2-k}\,. 
\end{align}
We now show that for all $k \in \N$ there holds  
\begin{align}\label{eq: important inclusion}
J_{u_k} \tilde{\subset} J_v \cup E_k \cup F_k\,.
\end{align}
To see this, we first observe that for $\mathcal{H}^{d-1}$-a.e.\ $x \in \Gamma = \bigcup\nolimits_{i=1}^\infty J_{v_i}  $
there holds
\begin{align}\label{eq: jumpsum}
[v](x) = [u_k](x) + \sum\nolimits_{i=k+1}^\infty c_i[v_i](x)\,.
\end{align}
Moreover, we get that  $c_i \le \eps_k d_i   2^{-i-1}$ for all $i \ge k+1$ by \eqref{eq: iterative conditions2}    and the fact that $(\eps_i)_i$ is decreasing. Fix $x \in J_{u_k} \setminus (E_k \cup F_k)$. Then  by  \eqref{eq: Ek-def} and \eqref{eq: jumpsum} we get 
\begin{align*}
|[v](x)| \ge |[u_k](x)| - \sum\nolimits_{i=k+1}^\infty c_i|[v_i](x)| \ge \eps_k - \sum\nolimits_{i=k+1}^\infty \frac{c_i}{d_i} \ge \eps_k\Big(1- \sum\nolimits_{i=k+1}^\infty 2^{-i-1}  \Big) \ge \eps_k/2\,,
\end{align*}   
where we have used that $|[u_k](x)| \ge \eps_k$ and $[v_i](x) \le 1 /d_i$,  for  $i \ge k+1$. Thus, $[v](x) \neq 0$ and therefore $x \in J_v$. Consequently, we have shown that $\mathcal{H}^{d-1}$-a.e.\ $x \in J_{u_k} \setminus (E_k \cup F_k)$ lies in $J_v$. This  shows \eqref{eq: important inclusion}. 

We now conclude the proof as follows: by \eqref{eq: iterative conditions}(i) and \eqref{eq: important inclusion} we get that 
 $$\bigcup\nolimits_{i=1}^l J_{v_i} \ \tilde{=} \ J_{u_l} \  \tilde{\subset} \  J_v \cup E_l \cup F_l   \tilde{\subset} \  J_v \cup E_k \cup F_k $$  
for all $l \ge k$, where we used that the sets $(E_k)_k$ and $(F_k)_k$ are decreasing. Taking the union with respect to $l$, we get that $\Gamma \tilde{\subset} J_v \cup E_k \cup F_k$ for all $k \in \N$. By \eqref{eq: EkFk} this implies $\mathcal{H}^{d-1}(\Gamma \setminus J_v) \le 2^{2-k}$. Since $k \in \N$ was arbitrary, we get $\Gamma \tilde{\subset} J_v$.  This along with \eqref{eq:  the first inclusion} shows $J_v \tilde{=} \Gamma$ and concludes the proof.    
\end{proof}

We close this section with the proof of  Lemma \ref{lemma: theta}.

\begin{proof}[Proof of Lemma \ref{lemma: theta}]
Let $B = \lbrace x \in V\colon \,  \limsup_{n \to \infty}|u_n(x) - v_n(x)| < +  \infty\rbrace$.  For $\theta \in (0,1)$, define $w^\theta_n = (1-\theta) u_n + \theta v_n$ and observe that $|w^\theta_n| \to \infty$ on $B$ for all $\theta$ since $|u_n| \to \infty$ on $V$. Let $D_\theta = \lbrace x \in V\setminus B\colon \, \limsup_{n \to \infty} |w^\theta_n(x)| < + \infty \rbrace$. As $|u_n -v_n| \to \infty$ on $V\setminus B$ and thus $|w_n^{\theta_1}- w_n^{\theta_2}| = |(\theta_1 - \theta_2)(v_n - u_n)|\to \infty$ on  $V \setminus B$ for all $\theta_1\neq \theta_2$, we obtain $D_{\theta_1} \cap D_{\theta_2} = \emptyset$. This implies that $\mathcal{L}^d(D_\theta)>0$ for an at most countable number of different $\theta$. We note that for all $\theta$ with $\mathcal{L}^d(D_\theta) = 0$ there holds $|w^\theta_n| \to \infty$ a.e.\ on $V$.   This yields the claim.   
\end{proof}

\section{Functionals defined on pairs of  function-set}\label{sec:FFF}

This section is devoted to the proofs of the results announced in Subsection \ref{sec: results1}. Before proving the relaxation and existence results, we address the lower bound separately since this will be instrumental also for Section \ref{sec:GGG}.

\subsection{The lower bound}

In this  subsection  we prove a lower bound for functionals defined on pairs of  function-set  which will be  needed  
for  the proof of   Theorem \ref{thm:relFDir}--Theorem~\ref{thm:relG}.   We will make use of the definition of $GSBD^p_\infty(\Omega)$ in Subsection \ref{sec:prel4}. In particular, we refer to  the definition of  $e(u)$ and of the  jump set $J_u$ with its normal $\nu_u$, see  \eqref{eq: general jump}--\eqref{eq:same},  as well as to  the notion of weak convergence in  $GSBD^p_\infty(\Omega)$, see \eqref{eq: weak gsbd convergence}.  We recall also that for any $s \in [0,1]$ and any 
$E \in \M(\Omega)$, $E^s$ denotes the set of points with density $s$ for $E$,  see \cite[Definition 3.60]{AFP}.

\begin{theorem}[Lower bound]\label{thm:lower-semi}
 Let $\Omega \subset \R^d$ be open and bounded, let $1 < p < \infty$.   Consider a sequence of Lipschitz sets $(E_n)_n \subset \Omega$ with $\sup_{n \in \N} \mathcal{H}^{d-1}(\partial E_n) < +\infty$  and a sequence of functions $(u_n)_n \subset GSBD^p(\Omega)$ such that $u_n |_{\Omega\sm \overline{E_n}} \in W^{1,p}(\Omega\sm \overline{E_n}; \Rd)$ and $u_n = 0$ in $ E_n  $.  Let $u \in GSBD^p_\infty(\Omega)$ and  $E \in \M(\Omega)$  such that  $u_n$  converges  weakly in $GSBD^p_\infty(\Omega)$ to $u$  and 
\begin{equation}\label{1405192012}
\chi_{E_n} \to \chi_E \text{ in }L^1(\Omega)\,.
\end{equation}
Then,   for any  norm   $\varphi$   on $\R^d$ there holds  
\begin{subequations}\label{eqs:1405191304}
\begin{equation}\label{1405191303}
e(u_n)\chi_{\Omega \sm (E_n \cup A^\infty_u)} \weak e(u) \chi_{\Omega \sm (E \cup A^\infty_u)} \quad\text{  weakly  in }L^p(\Omega;\Mdd)\,,
\end{equation}
\begin{equation}\label{1405191304}
\int_{J_u \cap E^0} 2\varphi(\nu_u)\dh + \int_{\Omega \cap \partial^* E} \varphi(\nu_E) \dh \leq \liminf_{n \to +\infty} \int_{ \Omega \cap  \partial E_n } \varphi(\nu_{E_n}) \dh\,,
\end{equation}
\end{subequations}
where $A^\infty_u=\{u=\infty\}$. 
\end{theorem}

In the proof, we need the following two auxiliary results, see \cite[Proposition 4, Lemma 5]{BraChaSol07}.
\begin{proposition}\label{prop:prop4BCS}
Let $\Omega$ be an open subset of $\Rd$ and $\mu$  be  a finite, positive set function defined on the family of open subsets of $\Omega$.  Let $\lambda \in \mathcal{M}_b^+(\Omega)$, and $(g_i)_{i \in \N}$ be a family of positive Borel functions on $\Omega$. Assume that $\mu(U) \geq \int_U g_i \,\mathrm{d}\lambda$ for every $U$ and $i$, and that $\mu(U \cup V) \geq \mu(U) + \mu(V)$ whenever $U$, $V \subset \subset \Omega$ and $\ove U \cap \ove V = \emptyset$ (superadditivity). Then $\mu(U) \geq \int_U (\sup_{i \in \N} g_i) \,\mathrm{d} \lambda$ for every open $U \subset \Omega$.
\end{proposition}

\begin{lemma}\label{le:lemma5BCS}
Let $\Gamma \subset E^0$ be a $(d{-}1)$-rectifiable subset, $\xi \in \Sd$ such that $\xi$ is not orthogonal to the normal $\nu_\Gamma$ to $\Gamma$ at any point of $\Gamma$. Then, for $\hd$-a.e.\ $y\in \Pi^\xi$,  the set $E\xy$ (see \eqref{eq: vxiy2})  has density 0 in $t$ for every $t \in \Gamma\xy$.
\end{lemma}

\begin{proof}[Proof of Theorem \ref{thm:lower-semi}]
Since  $u_n$  converges  weakly in $GSBD^p_\infty(\Omega)$ to $u$,  \eqref{eq: weak gsbd convergence} implies
\begin{align}\label {eq: weak gsbd convergence2}
\sup_{n \in \N}\nolimits \big( \Vert e(u_n) \Vert^p_{L^p(\Omega)} + \mathcal{H}^{d-1}(J_{u_n}) + \mathcal{H}^{d-1}(\partial E_n) \big) =:M < + \infty.
\end{align}
Consequently, Theorem~\ref{th: GSDBcompactness} and the fact that $\GGG\bar{d} \EEE(u_n,u) \to 0$, see \eqref{eq:metricd}  and   \eqref{eq: weak gsbd convergence},   imply that  $A^\infty_u=\lbrace u = \infty \rbrace = \lbrace x \in \Omega\colon \, |u_n(x)| \to \infty \rbrace$ and 
\begin{subequations}
\begin{align}
u_n &\to u \quad  \ \ \ \ \text{$\mathcal{L}^d$-a.e.\ in }\Omega \sm A^\infty_u\,,\label{1405192139}\\
e(u_n) &\weak e(u) \quad \text{ weakly  in } L^p(\Omega\sm A^\infty_u; \Mdd)\,.\label{1405192141}
\end{align}
\end{subequations}
By \eqref{1405192012}, \eqref{1405192139}, $u_n = 0$ on $E_n$, and the definition of $A^\infty_u$, we have 
\begin{align}\label{eq: EcapA}
E \cap A^\infty_u = \emptyset \ \ \ \text{ and  } \ \ \  \text{$u=0$ \ $\mathcal{L}^d$-a.e.\ in $E$}.
\end{align}
Then \eqref{1405192141} gives \eqref{1405191303}.

We now show \eqref{1405191304}  which is the core of the proof.   Let $\varphi^*$ be the dual norm of $\varphi$ and observe that (see, e.g.\ \cite[Section~4.1.2]{Bra98})
\begin{equation}\label{1904191706}
\varphi(\nu)= \max_{\xi \in \Sd} \frac{\nu \cdot \xi}{\varphi^*(\xi)} =  \max_{\xi \in \Sd} \frac{|\nu \cdot \xi|}{\varphi^*(\xi)} \,,
\end{equation}
where the second equality holds since   $\varphi(\nu) =  \varphi(-\nu)$.   

As a preparatory step, we consider a set $B \subset \Omega$ with Lipschitz boundary and a function $v$ with $v |_{\Omega\sm \ove B} \in W^{1,p}(\Omega\sm \ove B;\R^d)$  
and  $v=0$ in $B$  (observe  that  $v \in GSBD^p(\Omega)$). Recall the notation in \eqref{eq: vxiy2}--\eqref{eq: vxiy}. Let $\eps\in (0,1)$ and $U \subset \Omega$ be open. For each $\xi \in \Sd$ and $y \in \Pi^\xi$, we define 
\begin{align}\label{eq: fxsieps}
F^\xi_\eps (\widehat{v}\xy, B\xy; U\xy) =\eps \int_{U\xy \sm B\xy}  |(\widehat{v}\xy)'|^p  \,\mathrm{d}t +  \mathcal{H}^0(\partial B\xy \cap U\xy )  \frac{1}{\varphi^*(\xi)}\,. 
\end{align}
 By Fubini-Tonelli Theorem, with the slicing properties \eqref{3105171927}, \eqref{2304191254}, \eqref{2304191637},  for a.e.\ $\xi \in \mathbb{S}^{d-1}$  there holds
$$\int_{\Pi^\xi}  F_\eps^\xi(\widehat{v}\xy, B\xy; U\xy) \dh(y) = \varepsilon\int_{U \setminus B}|e(v) \xi \cdot \xi|^p \dx + \int_{U \cap \partial B} \frac{|\nu_B \cdot \xi | }{\varphi^*(\xi)}\dh \,.  $$
Since $|e(v)|\geq |e(v) \xi \cdot \xi|$, the previous estimate along with \eqref{1904191706} implies
$$\int_{\Pi^\xi}  F_\eps^\xi(\widehat{v}\xy, B\xy; U\xy) \dh(y) \le \eps \Vert e(v) \Vert_{L^p(U \setminus B)}^p + \int_{U \cap \partial B} \varphi(\nu_B) \, \d\mathcal{H}^{d-1}\,.  $$

By applying this estimate for the sequence of pairs $(u_n,E_n)$, we get by \eqref{eq: weak gsbd convergence2}  
\begin{align}\label{eq: Meps}
 \int_{\Pi^\xi} F_\eps^\xi((\widehat{u}_n)\xy, (E_n)\xy; U\xy) \dh(y) \le M\varepsilon + \int _{U \cap \partial E_n} \varphi(\nu_{E_n}) \dh \le M(\Vert \varphi \Vert_{L^\infty(\mathbb{S}^{d-1})} + \eps)
\end{align}
for all  open $U \subset \Omega$. Since $\GGG\bar{d} \EEE(u_n, u)\to 0$, we have that $\GGG\bar{d} \EEE((\widehat{u}_n)\xy, \widehat{u}\xy) =   \int_{\Omega\xy}  d_{\bar{\R}^d}((\widehat{u}_n)\xy, \widehat{u}\xy) \, \dx  \to 0$  for $\hd$-a.e.\ $y \in \Pi^\xi$ and  
$\hd$-a.e.\ 
$\xi \in \Sd$. (Notice that we have to restrict our choice to the $\xi$ satisfying $|u_n\cdot \xi| \to +\infty$ $\Ld$-a.e.\ in $A^\infty_u$, which are a set of full measure in $\Sd$, cf.\ \cite[Lemma~2.7]{CC18}.) In particular, this implies
\begin{subequations}
\begin{align}
(\widehat{u}_n)\xy \to \widehat{u}\xy \quad &\mathcal{L}^1\text{-a.e.\ in }(\Omega \sm A_u^\infty)\xy\,, \label{1505190913}\\
|(\widehat{u}_n)\xy| \to +\infty \quad &\mathcal{L}^1\text{-a.e.\ in }(A_u^\infty)\xy\,. \label{1505190915}
\end{align}
\end{subequations}
By using \eqref{eq: Meps} and Fatou's lemma we obtain 
\begin{equation}\label{1505190927}
\liminf_{n \to \infty} F_\eps^\xi((\widehat{u}_n)\xy, (E_n)\xy; U\xy) < +\infty
\end{equation}
for $\hd$-a.e.\ $y \in \Pi^\xi$ and any open $U \subset \Omega$. Then, we may find a subsequence $(u_m)_m=(u_{n_m})_m$, depending on $\varepsilon$, $\xi$, and $y$, such that
\begin{equation}\label{1505190928}
\lim_{m \to \infty} F_\eps^\xi((\widehat{u}_m)\xy, (E_m)\xy; U\xy) =\liminf_{n \to \infty} F_\eps^\xi((\widehat{u}_n)\xy, (E_n)\xy; U\xy)  
\end{equation}
 for any open $U \subset \Omega$. At this stage, up to passing to a further subsequence, we have
\begin{equation*}
\mathcal{H}^0\big(\partial (E_m)\xy \big) = N\xy \in \mathbb{N}\,,
\end{equation*}
independently of $m$,  so that the points in $\partial (E_m)\xy$ converge, as $m\to \infty$, to $M\xy \leq N\xy$ points
\begin{equation*}
t_1, \dots, t_{M\xy}\,,
\end{equation*}
which are either in $\partial E\xy$ or in a finite set $S\xy:=\{t_1, \dots, t_{M\xy}\} \setminus\partial E\xy  \subset  (E\xy)^0 \cup (E\xy)^1$,  where $(\cdot)^0$ and $(\cdot)^1$  denote the sets with one-dimensional density $0$  or $1$, respectively.  Notice that $E\xy$ is thus  the union 
of  $M\xy/2 -  \# S\xy$  intervals  (up to a finite set of points) on which there holds  $\widehat{u}\xy=0$,  see \eqref{eq: EcapA} and \eqref{1505190913}. 
In view of \eqref{eq: fxsieps} and \eqref{1505190927},  $( (\widehat{u}_m)\xy)'$ are equibounded (with respect to $m$) in $L^p_{\mathrm{loc}}(t_j, t_{j+1})$, for any interval 
\begin{equation*}
(t_j, t_{j+1}) \subset \Omega\xy \sm  (E\xy \cup S\xy)\,. 
\end{equation*} 
Then, as   in the proof of    \cite[Theorem~1.1]{CC18}, we have two alternative possibilities on $(t_j, t_{j+1})$: either $(\widehat{u}_m)\xy$ converge locally uniformly in $(t_j, t_{j+1})$ to $\widehat{u}\xy$, or $|(\widehat{u}_m)\xy|\to +\infty$ $\mathcal{L}^1$-a.e.\ in $(t_j, t_{j+1})$.
Recalling that $ J_{\widehat{u}\xy}  = \partial (A^\infty_u)\xy \cup \big(  (J_u^\xi)\xy  \sm (A^\infty_u)\xy\big)$, see \eqref{2304191254-1} and \eqref{eq: general jump},  we find 
\begin{equation}\label{1505191219}
J_{\xi,y}:=  J_{\widehat{u}\xy}  \cap (E\xy)^0 \subset  S\xy    \cap (E\xy)^0 \,. 
\end{equation}
 We notice that any point in $S\xy$ is the limit of two distinct sequences of points $(p^1_m)_m$, $(p^2_m)_m$ with $p^1_m$, $p^2_m \in \partial(E_m)\xy$. 
Thus, in view of \eqref{eq: fxsieps} and \eqref{1505190928}, for any open $U \subset \Omega$ we derive 
\begin{align}\label{eq: slicing formula}
\varepsilon\int_{U\xy \sm (E\cup A^\infty_u)\xy} \hspace{-0.2em} |(\widehat{u}\xy)'|^p \,\d t &+ \mathcal{H}^0(U\xy \cap  \partial E\xy  ) \frac{1}{\varphi^*(\xi)} + \mathcal{H}^0(U\xy \cap J_{\xi,y}) \frac{2}{\varphi^*(\xi)} \notag \\&\leq \liminf_{m\to \infty} F_\eps^\xi((\widehat{u}_m)\xy, (E_m)\xy; U\xy) =\liminf_{n \to \infty} F_\eps^\xi((\widehat{u}_n)\xy, (E_n)\xy; U\xy) \,.
\end{align} 
 We apply Lemma~\ref{le:lemma5BCS} to the rectifiable set 
$J_u \cap E^0 \cap \{ \xi \cdot \nu_u \neq 0\}$ and get that for $\hd$-a.e.\ $y\in \Pi^\xi$
\begin{equation*}
y + t \xi \in   J_u \cap E^0 \cap \{ \xi \cdot \nu_u \neq 0\} \ \ \ \Rightarrow \ \ \ t \in (E\xy)^0\,. 
\end{equation*}
This along with \eqref{1505191219}--\eqref{eq: slicing formula}, the slicing properties \eqref{3105171927}--\eqref{2304191637} (which also hold for $GSBD^p_\infty(\Omega)$ functions), and Fatou's lemma yields that for all $\xi \in \Sd  \setminus N_0$, for some $N_0$ with $\mathcal{H}^{d-1}(N_0) = 0$,  there holds
\begin{align*}
\varepsilon &\int_{U \sm (E\cup A^\infty_u)} \hspace{-0.2em} |e(u) \xi \cdot \xi|^p \dx + \int_{U \cap \partial^* E} \frac{|\nu_E \cdot \xi| }{\varphi^*(\xi )}\dh +  \int_{J_u \cap E^0\cap U}  \frac{2| \nu_u \cdot \xi |}{\varphi^*(\xi)}\dh 
\notag\\& 
 \le \int_{\Pi^\xi} \liminf_{n \to \infty}  F_\eps^\xi((\widehat{u}_n)\xy, (E_n)\xy; U\xy) \, \dh 
  \le  \liminf_{n \to \infty} \int_{\Pi^\xi}   F_\eps^\xi((\widehat{u}_n)\xy, (E_n)\xy; U\xy) \, \dh.
\end{align*}
Introducing the set function  $\mu$ defined on the open subsets of $\Omega$ by
 \begin{equation}\label{eq: mu-def}
\mu(U):=\liminf_{n \to +\infty} \int _{U\cap \partial E_n} \varphi(\nu_{E_n}) \dh\,,
\end{equation}
and letting $\eps \to 0$ we find by \eqref{eq: Meps} for all $\xi \in \Sd  \setminus N_0$ that 
\begin{align}\label{eq: lastequ.}
\int_{U \cap \partial^* E} \frac{|\nu_E \cdot \xi| }{\varphi^*(\xi  )}\dh +  \int_{J_u\cap E^0\cap U}  \frac{2| \nu_u \cdot \xi |}{\varphi^*(\xi)}\dh   \leq   \mu(U) \,.  
\end{align}
The set function $\mu $ is clearly superadditive. Let $\lambda= \hd \mres \big( J_u\cap E^0 \big)  + \hd \mres \partial^* E$ and define 
\begin{equation*}
g_i= \begin{dcases}
  \frac{2| \nu_u \cdot \xi_i |}{\varphi^*(\xi_i)}  \qquad & \text{on } J_u \cap E^0\,,\\
\frac{|\nu_E \cdot \xi_i| }{\varphi^*(\xi_i  )}\qquad & \text{on }\partial^*E\,,
\end{dcases}
\end{equation*}
where $(\xi_i)_i \subset \Sd \setminus N_0$ is a dense sequence in $\Sd$. By \eqref{eq: lastequ.} we have $\mu(U) \ge \int_U g_i \,\d \lambda$ for all $i \in \N$ and all open $U \subset \Omega$. Then,     Proposition~\ref{prop:prop4BCS} yields $\mu(\Omega) \ge \int_\Omega \sup_i g_i \,\d \lambda$. In view of \eqref{1904191706} and   \eqref{eq: mu-def}, this implies \eqref{1405191304} and concludes the proof. 
\end{proof}

\subsection{Relaxation for functionals defined on pairs of  function-set}\label{sec: sub-voids}

In this  subsection  we give the proof of  Proposition \ref{prop:relF} and  Theorem \ref{thm:relFDir}. \CCC We also provide corresponding generalizations to the space $GSBD^p_\infty$, see Proposition \ref{prop:relFinfty}  and Theorem \ref{thm:relFDirinfty}. \EEE    For the upper bound, we recall the following result proved in \cite[Proposition 9, Remark 14]{BraChaSol07}.

\begin{proposition}\label{th: Braides-upper bound}
Let $u \in L^1(\Omega;\R^d)$ and  $E \in \M(\Omega)$ such that $\mathcal{H}^{d-1}(\partial^* E)< +\infty$ and  $u\chi_{E^0} \in GSBV^p(\Omega;\R^d)$.  Then, there exists a sequence $(u_n)_n \subset W^{1,p}(\Omega;\R^d)$ and $(E_n)_n \subset \M(\Omega)$  with $E_n$ of class $C^\infty$ such that  $u_n \to u$ in $L^1(\Omega;\R^d)$, $\chi_{E_n} \to \chi_E$ in $L^1(\Omega)$,  and 
\begin{align*}
&\nabla u_n \chi_{\Omega \setminus E_n} \to  \nabla u \chi_{\Omega \setminus E} \ \ \text{ in } L^p(\Omega;\mathbb{M}^{d\times d}), \\ 
&\limsup_{n \to \infty} \int_{\partial E_n \cap \Omega   } \varphi(\nu_{E_n}) \dh \le  \int_{J_u \cap E^0} 2\varphi(\nu_u) \dh + \int_{\partial^* E\cap   \Omega  } \varphi(\nu_E) \dh\,.
\end{align*}
 Moreover, if $\mathcal{L}^d(E)>0$, one can guarantee in addition the condition
$\mathcal{L}^d(E_n) = \mathcal{L}^d(E)$  for $n \in \N$.

\end{proposition}

%


\begin{proof}[Proof of Proposition \ref{prop:relF}] We first   prove the lower inequality, and then the upper inequality. The lower inequality relies on Theorem \ref{thm:lower-semi} and the upper inequality on  a density argument along with   Proposition \ref{th: Braides-upper bound}. 

\noindent \emph{The lower inequality.} Suppose  that  $u_n \to u \text{ in }L^0(\Omega;\R^d) \text{ and } \chi_{E_n} \to \chi_E \text{ in }L^1(\Omega)$.  Without restriction, we can assume that $\sup_n F(u_n,E_n)< +\infty$. In view of \eqref{eq: F functional} and $\min_{\mathbb{S}^{d-1}} \varphi>0$, this implies $\mathcal{H}^{d-1}(\partial E_n) < + \infty$. Moreover, by \eqref{eq: growth conditions} the functions $v_n:= u_n \chi_{\Omega \setminus E_n}$ lie in $GSBD^p(\Omega)$  with $J_{v_n} \subset \partial E_n \cap \Omega$   and satisfy $\sup_{n}\Vert e(v_n)\Vert_{L^p(\Omega)}<+\infty$. This along with the fact that $u_n \to u$ in measure shows that $v_n$ converges  weakly  in $GSBD^p_\infty(\Omega)$ to $u\chi_{E^0}$, see \eqref{eq: weak gsbd convergence}, where we point out that $A^\infty_u = \lbrace u = \infty \rbrace = \emptyset$. In particular,   $u \chi_{E^0} \in GSBD^p_\infty(\Omega)$  and, since  $A^\infty_u = \emptyset$, even  $u \chi_{E^0} \in GSBD^p(\Omega)$, cf.\ \eqref{eq: compact extension}. 
As also \eqref{1405192012} holds, we can apply Theorem \ref{thm:lower-semi}.   The lower inequality now follows from \eqref{eqs:1405191304} and the fact that $f$ is convex.

\emph{The upper inequality.}  We first observe the following:  given $u \in L^0(\Omega;\R^d)$ and $E \in \M(\Omega)$ with $\mathcal{H}^{d-1}(\partial^* E)<\infty$ and $u\chi_{E^0} \in GSBD^p(\Omega)$, we find an approximating sequence $(v_n)_n \subset L^1(\Omega;\R^d)$ with $v_n\chi_{E^0} \in GSBV^p(\Omega;\R^d)$ such that
\begin{align*}
{\rm (i)} & \ \ v_n \to u  \chi_{E^0}  \ \ \text{ in } L^0(\Omega;\R^d)\,,\\ 
{\rm (ii)} & \ \  e(v_n) \chi_{\Omega \setminus E} \to  e(u) \chi_{\Omega \setminus E} \ \ \text{ in } L^p(\Omega;\Mdd)\,, \\ 
{\rm (iii)} & \ \  \mathcal{H}^{d-1}\big((J_{v_n}\triangle J_u) \cap E^0\big) \to 0. 
\end{align*}
This can be seen by approximating first $u \chi_{E^0}$ by a sequence $(\wu_n)_n$ by means of Theorem \ref{thm:densityGSBD}, and by setting $v_n:= \wu_n \chi_{E^0}$ for every $n$. It is then immediate to verify that the conditions in \eqref{eqs:main'} for $(\wu_n)_n$ imply the three conditions above.

By this approximation,  \eqref{eq: growth conditions}, and a diagonal argument, it thus suffices to construct a recovery sequence for  
$u \in L^1(\Omega;\R^d)$ with $u\chi_{E^0} \in GSBV^p(\Omega;\R^d)$. To this end, we apply Proposition~\ref{th: Braides-upper bound} to obtain $(u_n,E_n)_n$ and we consider the sequence $u_n\chi_{\Omega \setminus E_n}$. We further observe that, if $\mathcal{L}^d(E)>0$, this recovery sequences $(u_n,E_n)_n$ can be constructed ensuring $\mathcal{L}^d(E_n) = \mathcal{L}^d(E)$ for  $n \in \N$.  
\end{proof}

\CCC We briefly discuss that by a small adaption we  get a relaxation result for $F$ with respect to the topology induced by $\bar{d}$ on $L^0(\Omega; \bar{\R}^d)$. We  introduce $\ove{F}_\infty\colon L^0(\Omega; \bar{\R}^d) \times  \M(\Omega) \to \R \cup \lbrace+ \infty\rbrace$ by  
\begin{equation*}
\ove{F}_\infty(u,E) = \inf\Big\{ \liminf_{n\to\infty} F(u_n,E_n)\colon  \,  \bar{d}(u_n, u) \to 0 \text{ and } \chi_{E_n} \to \chi_E \text{ in }L^1(\Omega) \Big\}\,.
\end{equation*}


\begin{proposition}[Characterization of  the lower semicontinuous envelope $\ove{F}_\infty$]\label{prop:relFinfty} 
Under the assumptions of Proposition~\ref{prop:relF}, it holds that
\begin{equation*}
\ove{F}_\infty(u,E) = \begin{dcases} \int_{\Omega \setminus E} f(e(u))\, \dx + \int_{\Omega \cap \partial^* E} & \varphi  (\nu_E)  \, \dh + \int_{J_u \cap (\Omega \setminus E)^1} 2\, \varphi(\nu_u)  \, \dh \\
&\hspace{-1em}\text{if } u=  u \,  \chi_{E^0}  \in GSBD^p_\infty(\Omega) \text{ and }\mathcal{H}^{d-1}(\partial^* E) < +\infty\,,\\
+\infty &\hspace{-1em}\text{otherwise.}
\end{dcases}
\end{equation*}
Moreover, if $\mathcal{L}^d(E)>0$, then for any $(u,E) \in L^0(\Omega;\bar{\R}^d){\times}\M(\Omega)$ there exists a recovery sequence $(u_n,E_n)_n\subset L^0(\Omega;\Rd){\times}\M(\Omega)$ such that $\mathcal{L}^d(E_n) = \mathcal{L}^d(E)$ for all $n \in \N$.
\end{proposition}

\begin{proof}
It is easy to check that the  lower inequality still works for  $u = u\chi_{E^0} \in GSBD^p_\infty(\Omega)$ by Theorem \ref{thm:lower-semi}, where we use \eqref{eq: growth conditions}, $f(0)=0$, and the fact that $e(u) = 0$ on $\lbrace u = \infty \rbrace$, see \eqref{eq:same}.   Moreover,  we are able to extend the upper inequality to any $u \in L^0(\Omega; \bar{\R}^d)$ such that $u=u \,\chi_{E^0} \in GSBD^p_\infty(\Omega)$. In fact, it is enough to notice that  for any $u=u \,\chi_{E^0} \in GSBD^p_\infty(\Omega)$ and any sequence $(t_n)_n\subset \Rd$ with $|t_n| \to \infty$ such that for  the functions $\tilde{u}_{t_n} \in GSBD^p(\Omega)$ defined in \eqref{eq: compact extension}  property \eqref{eq:same} holds,   we obtain  $\tilde{u}_{t_n} =\tilde{u}_{t_n} \, \chi_{E^0}$,  $\bar{d}(\tilde{u}_{t_n}, u) \to 0$ as $n \to \infty$ ,  and
\begin{equation*}
\int_{\Omega \setminus E} f(e(u))\, \dx + \int_{\Omega \cap \partial^* E}  \varphi  (\nu_E)  \, \dh + \int_{J_u \cap (\Omega \setminus E)^1} 2\, \varphi(\nu_u)  \, \dh =   \overline{F}(\tilde{u}_{t_n},E)
\end{equation*}
for all $n \in \N$, with $J_u$ defined by \eqref{eq: general jump}.  Then, the upper inequality follows from the upper inequality in Proposition \ref{prop:relF} and a diagonal argument.
\end{proof}

\EEE

As a consequence, we obtain the  following  lower semicontinuity result in $GSBD^p_\infty$.
 
\begin{corollary}[Lower semicontinuity in $GSBD^p_\infty$]\label{cor: GSDB-lsc}
Let us suppose that a sequence $(u_n)_n \subset  GSBD^p_\infty(\Omega)$ converges weakly in $GSBD^p_\infty(\Omega)$ to  $u\in GSBD^p_\infty(\Omega)$,  see \eqref{eq: weak gsbd convergence}. Then for each   norm  $\phi$ on $\Rd$ 
there holds
$$\int_{J_u} \phi(\nu_u) \, \dh \le \liminf_{n \to \infty} \int_{J_{u_n}} \phi(\nu_{u_n}) \, \dh. $$
\end{corollary}

\begin{proof}
 Let $\eps>0$ and $f(\zeta) =  \eps|\zeta^T + \zeta|^p  $ for $\zeta \in \mathbb{M}^{d\times d}$.  The upper inequality in  
\CCC  Proposition~\ref{prop:relFinfty} \EEE
 (for $u_n$ and $E = \emptyset$) shows that for each $u_n \in GSBD_\infty^p(\Omega)$, we can find a   Lipschitz   set $E_n$ 
 with $\mathcal{L}^{d}(E_n) \le \frac{1}{n}$ and  $v_n \in L^0(\Omega;\R^d)$ with $v_n|_{\Omega\sm \ove E_n} \in W^{1,p}( \Omega \sm \ove E_n;\R^d)$,   $v_n|_{E_n}=0$, and  $\GGG\bar{d} \EEE(v_n,u_n) \le \frac{1}{n}$ (see \eqref{eq:metricd}) such that 
\begin{align}\label{eq: imm-ref}
{  \int_{\Omega \setminus E_n} \eps|e(v_n)|^p\, \dx +  \int_{\Omega \cap\partial E_n } \phi(\nu_{E_n}) \dh \le   \int_{\Omega }  \eps|e(u_n)|^p\, \dx  +   \int_{J_{u_n}} 2\phi(\nu_{u_n}) \dh + \frac{1}{n} \,.}
 \end{align}
 Observe that   $\GGG\bar{d} \EEE(v_n,u) \to 0$ as $n\to \infty$,  and thus $v_n$ converges weakly to $u$ in $GSBD^p_\infty(\Omega)$.  By applying  Theorem~\ref{thm:lower-semi}  on $(v_n,E_n)$ and using $E = \emptyset$ we get \[
 \int_{J_u} 2\phi(\nu_u) \, \dh \le \liminf_{n \to \infty} \int_{\Omega \cap\partial E_n } \phi(\nu_{E_n})\dh\,.
 \] 
  This, along with \eqref{eq: imm-ref}, $\sup_{n\in\N} \Vert e(u_n)\Vert_{L^p(\Omega)}<+\infty$, and the arbitrariness of $\eps$ yields the result. \end{proof}

 We now address the relaxation of $F_{\rm Dir}$, see \eqref{eq: FDir functional}, i.e., a version of $F$ with boundary data. 

We   take advantage of the following  approximation result  which is obtained by following the lines of \cite[Theorem~5.5]{CC17}, where an analogous approximation is proved for Griffith  functionals  with Dirichlet boundary   conditions.  The new feature with respect to \cite[Theorem~5.5]{CC17} is that,  besides the construction of approximating   functions   with the correct boundary data, also  approximating sets  are constructed.    For  convenience of the reader, we give a sketch of the proof in Appendix~\ref{sec:App}  highlighting the adaptations needed with respect to \cite[Theorem~5.5]{CC17}.  In the following,  we denote by $\ove F'_{\mathrm{Dir}}$ the functional on the right hand side of \eqref{1807191836}.

\begin{lemma}\label{le:0410191844}
 Suppose that $\partial_D \Omega\subset \partial \Omega$ satisfies \eqref{0807170103}.  Consider $(v, H) \in L^0(\Omega;\Rd) \times \M(\Omega)$ such that  $\ove F'_{\mathrm{Dir}}(v,H)<+\infty$.  Then there exist $(v_n, H_n) \in (L^p(\Omega;\Rd) \cap SBV^p(\Omega;\Rd)) \times \M(\Omega)$ such that $J_{v_n}$ is closed in $\Omega$ and included in a finite union of closed connected pieces of $C^1$ hypersurfaces, $v_n \in W^{1,p}(\Omega\sm J_{v_n}; \Rd)$, $v_n=u_0$ in a neighborhood $V_n \subset \Omega$ of $\dod$,  $H_n$  is a set of finite perimeter, and
\begin{itemize}
\item[{\rm (i)}] $v_n \to v$  in $L^0(\Omega;\Rd)$,
\item[{\rm (ii)}]$ \chi_{H_n} \to \chi_H  $ in $L^1(\Omega)$,
\item[{\rm (iii)}] $\limsup_{n\to \infty}  \ove F'_{\mathrm{Dir}}(v_n, H_n) \leq \ove F'_{\mathrm{Dir}}(v,H)$.
\end{itemize} 
\end{lemma}


\begin{proof}[Proof of Theorem \ref{thm:relFDir}]
First, we denote  by   $\Omega'$ a bounded open set with $\Omega \subset \Omega' $ and $\Omega' \cap \dom = \dod$.  By  $F'$ and ${\ove F}'$ we denote the analogs of the functionals $F$ and $\ove F$, respectively, defined on $L^0(\Omega';\Rd) {\times} \M(\Omega')$. Given $u \in L^0(\Omega;\R^d)$, we define the extension $u' \in L^0(\Omega';\R^d)$ by setting $u' = u_0$ on $\Omega' \setminus \Omega$ for fixed boundary values $u_0 \in W^{1,p}(\R^d;\R^d)$.  Then, we observe 
\begin{equation}\label{eq: extensi}
F'(u', E) = F_{\mathrm{Dir}}(u,E) + \int_{\Omega' \sm \Omega} f(e(u_0)) \dx\,, \quad {\ove F}'(u', E) = \ove F_{\mathrm{Dir}}(u,E) + \int_{\Omega' \sm \Omega} f(e(u_0)) \dx\,.
\end{equation}
Therefore, the lower inequality follows from Proposition \ref{prop:relF} applied for $F', \ove F'$ instead of $F, \ove F$.


We now address the upper inequality. 
In view of Lemma~\ref{le:0410191844} and by a diagonal argument, it is enough to prove the result in the case where, besides the assumptions in the statement, also $u \in L^p(\Omega;\Rd) \cap SBV^p(\Omega;\Rd)$ and $u = u_0$ in a neighborhood $U \subset \Omega$ of $\partial_D \Omega$.



 By $(u_n,E_n)_n$ we denote a recovery sequence for $(u,E)$ given by Proposition~\ref{th: Braides-upper bound}. In general, the functions $(u_n)_n$  do not satisfy the boundary conditions required in \eqref{eq: FDir functional}. Let $\delta>0$ and let $V \subset \subset U$ be a  smaller neighborhood of $\partial_D \Omega$. In view of \eqref{eq: growth conditions}--\eqref{eq: F functional}, by a standard diagonal argument in the theory of $\Gamma$-convergence,  it suffices to find a sequence $(v_n)_n \subset L^1(\Omega;\R^d)$ with $v_n|_{\Omega \setminus \ove E_n} \in W^{1,p}(  \Omega   \setminus \ove E_n;\R^d)$, $v_n=0$ on $E_n$, and $v_n = u_0$ on $V\setminus \ove E_n$ such that
\begin{align}\label{eq: vnunv}
\limsup\nolimits_{n \to \infty}\Vert v_n   - u\chi_{E^0} \Vert_{L^1(\Omega)} \le \delta, \qquad   \limsup\nolimits_{n \to \infty}\Vert e(v_n) - e(u_n)\chi_{\Omega \setminus E_n}\Vert_{L^p(\Omega)}  \le \delta.
\end{align}
To this end, choose  $\psi \in C^\infty(\Omega)$ with $\psi = 1$ in $\Omega \setminus U$ and $\psi = 0$ on $V$. \CCC The sequence $(u_n)_n$ converges to $u$ only in $L^1(\Omega;\R^d)$. Therefore, we introduce truncations to obtain $L^p$-convergence:   for $M>0$, we define \EEE $u^{M}$ by $u^M_i =  (-M \vee  u_i)\wedge M $, where $u_i$ denotes the $i$-th component, $i=1,\ldots,d$. In a similar fashion, we define $u_{n}^M$.  By Proposition~\ref{th: Braides-upper bound}  we then get  $\chi_{E_n} \to \chi_E$ in $L^1(\Omega)$  and 
\begin{align*}
 u_{n}^M \to u^M \ \ \ \text{in} \ \ \ L^p(\Omega;\R^d), \ \ \ \ \ \ \ \ \ \nabla u_{n}^M \chi_{\Omega \setminus E_n} \to  \nabla u^M \chi_{E^0} \ \ \ \text{in} \ \ \ L^p(\Omega;\mathbb{M}^{d\times d})\,.
 \end{align*}   
We define $v_n:= (\psi u_n^M + (1-\psi) u_0)\chi_{\Omega \setminus E_n}$. Clearly,  $v_n = u_0$ on $V\setminus \ove E_n$. By $\nabla v_n =  \psi \nabla u_{n}^M+ (1-\psi)\nabla u_{0} + \nabla \psi \otimes (u_{n}^M-u_0)$ on $\Omega \setminus E_n$,   $u=u_0$ on $U$, and  Proposition~\ref{th: Braides-upper bound}  we find
\begin{align*}
\limsup\nolimits_{n \to \infty}\Vert v_n - u\Vert_{L^1(\Omega)}&\le\Vert  u - u^M\Vert_{L^1(\Omega)}\,, \\
\limsup\nolimits_{n \to \infty}\Vert e(v_n) - e(u_n)\chi_{\Omega \setminus E_n}\Vert_{L^p(\Omega)}&\le \Vert \nabla u -\nabla u^M\Vert_{L^p(E^0)} +  \Vert \nabla \psi \Vert_\infty\Vert  u - u^M\Vert_{L^p(\Omega)}\,.
\end{align*} 
For $M$ sufficiently large, we obtain \eqref{eq: vnunv}  since $u=u\chi_{E^0}$.  This concludes the proof. 
\end{proof}

\CCC As done for the passage from Proposition~\ref{prop:relF} to Proposition~\ref{prop:relFinfty}, we may obtain the following characterization of the lower semicontinuous envelope of $F_{\mathrm{Dir}}$ with respect to the convergence induced by $\bar{d}$ on $L^0(\Omega; \bar{\R}^d)$. 
\begin{theorem}[Characterization of  the lower semicontinuous envelope $\ove{F}_{\mathrm{Dir},\infty}$]\label{thm:relFDirinfty}
Under the assumptions of Theorem~\ref{thm:relFDir}, the lower semicontinous envelope
\begin{equation*}\label{eq: Fdir-relainfty}
\ove{F}_{\mathrm{Dir},\infty}(u,E) = \Big\{ \liminf_{n\to\infty} F_{\mathrm{Dir}}(u_n,E_n)\colon  \,  \bar{d}(u_n, u) \to 0 \text{ and } \chi_{E_n} \to \chi_E \text{ in }L^1(\Omega) \Big\}
\end{equation*}
 for $ u\in L^0(\Omega; \bar{\R}^d)$ and $E \in \M(\Omega)$  is given by
\begin{equation}\label{1807191836'}
\ove{F}_{\mathrm{Dir},\infty}(u,E) = 
\ove{F}_\infty(u,E) +  \int\limits_{\dod \cap \partial^* E}  \hspace{-0.5cm} \varphi (\nu_E)    \dh + \int\limits_{ \{ \mathrm{tr}(u)  \neq \mathrm{tr}(u_0)  \} \cap  (\dod \setminus \partial^* E) }   \hspace{-0.5cm} 2 \, \varphi(  \nu_\Omega  ) \dh\,.
\end{equation}
Moreover, if $\mathcal{L}^d(E)>0$, then for any $(u,E) \in L^0(\Omega;\bar{\R}^d){\times}\M(\Omega)$ there exists a recovery sequence $(u_n,E_n)_n\subset L^0(\Omega;\Rd){\times}\M(\Omega)$ such that $\mathcal{L}^d(E_n) = \mathcal{L}^d(E)$ for all $n \in \N$.
\end{theorem}
 Notice that in \eqref{1807191836'} we wrote $\mathrm{tr}(u)$ also for $u\in GSBD^p_\infty(\Omega)$, with a slight abuse of notation: 
$\mathrm{tr}(u)$ should be intended as $\mathrm{tr}(\tilde{u}_t)$, cf.\ \eqref{eq: compact extension}, for any $t \in \R^d$ such that $\hd(\{ u_0 = t\} \cap \dod)=0$\EEE.

\subsection{Compactness and existence results for the relaxed functional}\label{sec: Fcomp}

We start with the following general compactness result. 
\begin{theorem}[Compactness]\label{thm:compF}
For  every  $(u_n, E_n)_n$ with $\sup_n F(u_n, E_n) < +\infty$, there exist a subsequence (not relabeled), $u \in GSBD^p_\infty(\Omega)$,  and  $E \in \M(\Omega)$ with $\hd(\partial^* E) < +\infty$ such that $u_n$  converges  weakly in $GSBD^p_\infty(\Omega)$ to $u$  and  $\chi_{E_n} \to \chi_E$ in $L^1(\Omega)$.
\end{theorem}
\begin{proof}
Let $(u_n, E_n)_n$ with $\sup_n F(u_n, E_n) < +\infty$.  As by the assumptions  on  $\varphi$ there holds  $\sup_{n\in \N} \hd(\partial E_n) < +\infty$,  a compactness result for sets of finite perimeter (see \cite[Remark 4.20]{AFP}) implies that   there exists $E \subset \Omega$ with $\mathcal{H}^{d-1}(\partial^* E)<+\infty$  such that $\chi_{E_n} \to \chi_E$ in $L^1(\Omega)$, up to a subsequence (not relabeled).

 Since the functions $u_n=u_n \chi_{\Omega \sm E_n}$ satisfy $J_{u_n} \subset \partial E_n$, we get $\sup_n \mathcal{H}^{d-1}(J_{u_n}) <+ \infty$. Moreover, by the growth assumptions on $f$ (see \eqref{eq: growth conditions})  we get that $\Vert e(u_n)\Vert_{L^p(\Omega)}$ is uniformly bounded. Thus, by Theorem \ref{th: GSDBcompactness},   $u_n$ converges (up to a subsequence) weakly in $GSBD^p_\infty(\Omega)$ to  some   $u \in GSBD^p_\infty(\Omega)$.  This concludes the proof. 
\end{proof}
\CCC We are now ready to prove Theorem~\ref{th: relF-extended}.
\begin{proof}[Proof of Theorem~\ref{th: relF-extended}.]
The existence of minimizers  for $\ove F_{\mathrm{Dir},\infty}$  follows by combining  Theorem~\ref{thm:relFDirinfty}  and Theorem~\ref{thm:compF},    by means of  \CCC general properties of relaxations, see e.g.\ \cite[Theorem~3.8]{DMLibro}.
To obtain minimizers for $\ove F_{\mathrm{Dir}}$, it is enough to observe (recall \eqref{eq: compact extension}, \eqref{eq:same})  that
\begin{equation*}
\ove F_{\mathrm{Dir},\infty}(u,E) \ge \ove F_{\mathrm{Dir}}(v_a, E)
\end{equation*}
for every $u\in GSBD^p_\infty(\Omega)$ and $v_a := u\chi_{\Omega\sm A^\infty_u} + a \chi_{A^\infty_u}$  (recall $A^\infty_u=\{u=\infty\}$), where $a\colon \Rd \to \Rd$ is an arbitrary affine function  with skew-symmetric gradient (usually called an \emph{infinitesimal rigid motion}).
Starting from a minimizer of $\ove F_{\mathrm{Dir},\infty}$, if $A^\infty_u \neq \emptyset$, we thus obtain a family of minimizers for $\ove F_{\mathrm{Dir}}$, parametrized by the infinitesimal rigid motions $a$. This concludes the proof.
\end{proof}
\EEE

\section{Functionals on domains with a subgraph constraint}\label{sec:GGG}

In this section we prove the the results announced in Subsection \ref{sec: results2}.

\subsection{Relaxation of the energy $G$}\label{subsec:RelG}
This subsection is devoted to the proof of Theorem~\ref{thm:relG}.  The lower inequality is obtained by exploiting the tool  of  $\sigma^p_{\mathrm{sym}}$-convergence introduced in Section~\ref{sec:sigmap}. The corresponding analysis will  prove to be  useful also for the compactness theorem in the next   subsection.  The proof of the upper inequality is quite delicate, and a careful procedure is needed to guarantee that the approximating   displacements are still defined on a domain which is the subgraph of a function. We only follow partially  the strategy in \cite[Proposition~4.1]{ChaSol07}, and employ  also other arguments in order to  improve the $GSBV$ proof which might fail in some pathological cases.

 Consider a Lipschitz set $\omega \subset \R^{d-1}$ which is uniformly star-shaped with respect to the origin, see \eqref{eq: star-shaped}.  We recall the notation $\Omega = \omega \times (-1, M  + 1 )$ and 
\begin{align}\label{eq: graphi}
\Omega_h = \lbrace x\in \Omega \colon \,   -1 <  x_d   <    h(x') \rbrace, \ \ \ \ \ \Omega_h^+ = \Omega_h \cap \lbrace x_d > 0 \rbrace
\end{align}
for $h\colon\omega \to  [0,M]$ measurable, where we write $x = (x',x_d)$ for $x \in \R^d$.  Moreover, we let  $\Omega^+ = \Omega \cap \lbrace x_d>0 \rbrace$.

\paragraph*{\textbf{The lower inequality.}} Consider  $(u_n, h_n)_n$ with $\sup_n G(u_n, h_n) < +\infty$. Then,  we have that $h_n \in C^1(\omega;  [0,M]  )$, $u_n|_{\Omega_{h_n}} \in W^{1,p}(\Omega_{h_n};\R^d )$, $u_n|_{\Omega\sm \Omega_{h_n}}=0$,  and   $u_n = u_0$ on $\omega\times (-1,0)$. 
Suppose that  $(u_n, h_n)_n$  converges in $L^0(\Omega;\Rd){\times}L^1(\omega)$ to $(u, h)$.
We let 
\begin{equation}\label{eq: Gamma-n}
\Gamma_n:=\partial \Omega_{h_n}  \cap \Omega = \lbrace x \in \Omega\colon \,  h_n(x') = x_d \rbrace 
\end{equation}
 be the graph of the function $h_n$.  Note that $\sup_n \mathcal{H}^{d-1}(\Gamma_n)<+\infty$.   We take $U = \omega \times (-\frac{1}{2},M) $  and $U' = \Omega = \omega \times (-1,M+1)$, and  apply Theorem~\ref{thm:compSps},  to  deduce that $(\Gamma_n)_n$ $\sps$-converges  (up to a subsequence)  to a pair $(\Gamma, G_\infty)$.   A fundamental step in the proof will be to show that
\begin{align}\label{eq: G is empty}
G_\infty = \emptyset\,.
\end{align}
We postpone the proof of this property to   Step~3  below. We first   characterize  the limiting set $\Gamma$ (Step 1) and prove the lower inequality  (Step 2), by following
   partially the  lines of \cite[Subsection~3.2]{ChaSol07}. 
In the whole proof, for simplicity we omit to write $\tilde{\subset}$ and $\tilde{=}$ to indicate that the inclusions hold up to  $\hd$-negligible sets.

\noindent \emph{Step 1:  Characterization   of the  limiting set $\Gamma$.} Let us prove that the set 
\begin{equation}\label{eq: sigma-low-def}
\Sigma:= \Gamma \cap \Omega_h^1
\end{equation}
is vertical, that is 
\begin{equation}\label{eq: Sigma}
(\Sigma+t e_d) \cap \Omega_h^1 \subset \Sigma \ \ \ \  \text{for any $t \geq 0$}\,.
\end{equation}
This follows as in \cite[Subsection~3.2]{ChaSol07}:  in fact, consider  $(v_n)_n$ and $v$ as in Definition~\ref{def:spsconv}(ii).   In particular,  $v_n = 0$ on $U' \setminus U$, $J_{v_n} {\subset} \Gamma_n$, and, in view of \eqref{eq: G is empty},   $ v$ is $\Rd$-valued with $\Gamma = J_v$.   The functions $v'_n(x):= v_n(x', x_d -t) \chi_{\Omega_{h_n}}(x)$ (with $ t >0  $, extended   by zero  in $\omega{\times}(-1, -1+t)$)  
converge to $v'(x):= v(x', x_d-t) \chi_{\Omega_h}(x)$   in measure   on $U'$.  Since $J_{v_n'} \subset \Gamma_n$,  Definition~\ref{def:spsconv}(i)  implies $J_{v'} \setminus \Gamma {\subset} (G_\infty)^1$. As $G_\infty = \emptyset$ by \eqref{eq: G is empty}, we get $J_{v'} {\subset} \Gamma$,     so that 
\[
(\Sigma + t e_d)  \cap \Omega_h^1  =  (\Gamma + t e_d)  \cap \Omega_h^1  =  (J_v + t e_d)  \cap \Omega_h^1 = J_{v'} \cap \Omega_h^1\subset \Gamma \cap \Omega_h^1=\Sigma\,,
\]
where we have used $\Gamma = J_v$. This shows \eqref{eq: Sigma}. In particular, $\nu_{\Sigma} \cdot e_d =0$ $\hd$-a.e.\ in $\Sigma$.  Next, we show that
\begin{equation}\label{2304191211}
\hd(\partial^* \Omega_h \cap \Omega) + 2 \hd(\Sigma) \leq \liminf_{n\to \infty} \int_\omega \sqrt{1 + |\nabla h_n(x')|^2} \, \d x'\,.
\end{equation}
To see  this, we again consider functions $(v_n)_n$ and  $v$ satisfying Definition~\ref{def:spsconv}(ii). In particular, we have  $J_{v_n} \subset \Gamma_n$ and  $J_v = \Gamma$. Since $\Gamma_n$ is the graph of a $C^1$ function,  we either get $v_n|_{\Omega_{h_n}} \equiv \infty$ or, by  Korn's inequality,   we have $v_n|_{\Omega_{h_n}} \in W^{1,p}(\Omega_{h_n}; \Rd)$. Since $v_n = 0$ on $U' \setminus U$, we obtain $v_n|_{\Omega_{h_n}} \in W^{1,p}(\Omega_{h_n}; \Rd)$.     We apply Theorem \ref{thm:lower-semi} for $E_n  = \Omega \setminus \Omega_{h_n}$, $E= \Omega \setminus \Omega_{h}$, and the sequence of functions $w_n := v_n\chi_{\Omega \setminus E_n} = v_n \chi_{\Omega_{h_n}}$.

 Observe that $\chi_{E_n} \to \chi_E$ in $L^1(\Omega)$. Moreover,  $w_n$  converges weakly in $GSBD^p_\infty(\Omega)$ to $w := v\chi_{\Omega_h}$ since $v_n$ converges weakly in $GSBD^p_\infty(\Omega)$ to $v$  and $\sup_n \mathcal{H}^{d-1}(\partial E_n)<+\infty$.  By \eqref{1405191304} for $\varphi \equiv 1$ on $\mathbb{S}^{d-1}$ there holds
$$
\hd(\partial^* \Omega_h \cap\Omega) + 2 \hd(J_w  \cap  \Omega_{h}^1) \leq \liminf_{n\to \infty} \mathcal{H}^{d-1}(\partial \Omega_{h_n} \cap \Omega) \,,
$$ 
where we used that $E^0 = \Omega_{h}^1$ and $\partial^* E \cap \Omega = \partial^* \Omega_h \cap \Omega$. Since $J_v = \Gamma$ and  $J_w \cap \Omega_h^1 = J_v \cap \Omega_h^1 = \Sigma$, we indeed  get \eqref{2304191211}, where for the right hand side we use that $\partial \Omega_{h_n}$ is the graph of the function $h_n\in C^1(\omega;   [0,M]  )$.   For later purposes in  Step~3,  we also note that by Corollary \ref{cor: GSDB-lsc} for $\phi(\nu)= |\xi \cdot \nu|$, with $\xi \in \mathbb{S}^{d-1}$ fixed, we get
\begin{align}\label{eq: aniso-lsc}
\int_\Gamma |\nu_\Gamma \cdot \xi|  \dh = \int_{J_v} |\nu_v \cdot \xi|  \dh \le \liminf_{n \to \infty}\int_{J_{v_n}} |\nu_{v_n} \cdot \xi| \dh \le  \liminf_{n \to \infty} \int_{\Gamma_n} |\nu_{\Gamma_n} \cdot \xi|  \dh\,. 
\end{align}
 (Strictly speaking, as $\phi$ is only a seminorm, we apply Corollary \ref{cor: GSDB-lsc} for $\phi +\eps$ for any $\eps>0$.)

\noindent \emph{Step 2: The lower inequality.} We now show the lower bound.  Recall that $(u_n, h_n)_n$ converges in $L^0(\Omega;\Rd){\times}L^1(\omega)$ to $(u, h)$ and that  $(G(u_n, h_n))_n$ is bounded. Then, \eqref{eq: growth conditions} and $\min_{\mathbb{S}^{d-1}} \varphi>0$ along with Theorem \ref{th: GSDBcompactness} and the fact that $\mathcal{L}^d(\lbrace x \in \Omega\colon |u_n(x)| \to \infty \rbrace)=0 $  imply that 
the limit $u= u\chi_{\Omega_h}$  lies in $GSBD^p(\Omega)$. 
There also holds  $u = u_0$ on $\omega\times (-1,0)$ by \eqref{eq: GSBD comp}(i) and the fact that  $u_n = u_0$ on $\omega\times (-1,0)$ for all $n \in \N$. In particular, we observe that $u_n = u_n \chi_{\Omega_{h_n}}$  converges weakly in  $GSBD^p_\infty(\Omega)$  to $u$, cf.\ \eqref{eq: weak gsbd convergence}. The fact that $h \in BV(\omega;  [0,M]  )$ follows from a standard compactness argument. This shows $\overline{G}(u,h) < + \infty$.

To obtain the lower bound for the energy,  we again apply Theorem \ref{thm:lower-semi} for $E_n  = \Omega \setminus \Omega_{h_n}$ and $E= \Omega \setminus \Omega_{h}$. Consider the sequence of functions  $v_n := \psi u_n\chi_{\Omega \setminus E_n} = \psi u_n$, where $\psi \in C^\infty(\Omega)$ with $\psi = 1$ in a neighborhood of $\Omega^+  = \Omega\cap\lbrace x_d >0\rbrace  $ and $\psi = 0$ on $\omega  \times (-1,-\frac{1}{2})$. We observe that $v_n = 0$ on $U'\setminus U  = \omega \times ((-1,-\frac{1}{2}]\cup[M,M+1)) $ and that $v_n$    converges to $ v:=  \psi u \in GSBD^p(\Omega)$  weakly in  $GSBD^p_\infty(\Omega)$. Now we apply Theorem \ref{thm:lower-semi}.     
First, notice that \eqref{1405191303},   $\psi = 1$ on $\Omega^+$,    and the fact that $A^\infty_u = \emptyset$ imply  $e(u_n)\chi_{\Omega^+_{h_n}} \weak e(u) \chi_{\Omega^+_h}$  weakly in $L^p(\Omega;\Mdd)$.    This along with the convexity of $f$  yields
\begin{equation}\label{2304191220}
\int_{\Omega_h^+} f(e(u)) \dx \leq \liminf_{n\to \infty} \int_{\Omega_{h_n}^+} f(e(u_n)) \dx\,.
\end{equation}
Moreover, applying Definition~\ref{def:spsconv}(i) on the sequence $(v_n)_n$,  which satisfies $v_n= 0$ on $U' \setminus U$ and $J_{v_n} \subset \Gamma_n$,   we observe $J_u   = J_v  \subset \Gamma$,  where we have also used \eqref{eq: G is empty}.   Recalling  the definition of $J_u' = \lbrace (x',x_d + t)\colon \, x \in J_u, \, t \ge 0 \rbrace$, see \eqref{eq: Ju'},  and using \eqref{eq: sigma-low-def}--\eqref{eq: Sigma} we find $J_u'  \cap \Omega^1_h\subset\Sigma$. Thus, by  \eqref{2304191211}, we obtain
\begin{equation}\label{2304191220-2}
\hd(\partial^* \Omega_h \cap \Omega) + 2 \hd(J_u'  \cap \Omega^1_h)  \leq \liminf_{n\to \infty} \int_\omega \sqrt{1 + |\nabla h_n(x')|^2} \, \d x'\,.
\end{equation}
Collecting \eqref{2304191220} and \eqref{2304191220-2} we conclude the lower inequality. 
To conclude the proof, it remains to confirm \eqref{eq: G is empty}.

\noindent \emph{Step  3:  Proof of $G_\infty =\emptyset$.}  Recall the definition of the graphs $\Gamma_n$ in  \eqref{eq: Gamma-n} and its $\sps$-limit $\Gamma$ on the sets  $U = \omega \times (-\frac{1}{2},M) $  and $U' = \Omega$. As before, consider $\psi \in C^\infty(\Omega)$ with $\psi = 1$ in a neighborhood of $\Omega^+$ and $\psi = 0$ on $\omega  \times (-1,-\frac{1}{2})$.     By employing (i) in Definition~\ref{def:spsconv} for the sequence  $v_n = \psi \chi_{\Omega_{h_n}} e_d$ and its limit  $v=  \psi \chi_{ \Omega_{h}  } e_d$,    we get  that $(\partial^*\Omega_h \cap \Omega) \setminus  \Gamma {\subset} (G_\infty)^1$. Since  $U' \cap \partial^* G_\infty \subset \Gamma$   by definition of $\sps$-convergence, we observe
\begin{align}\label{eq: summary5}
\Gamma \supset \big(\partial^* G_\infty\cap \Omega \big) \cup  \big( \partial^* \Omega_h  \cap \Omega  \cap  (G_\infty)^0 \big)\,. 
\end{align}
We estimate  the $\mathcal{H}^{d-1}$-measure of  the two terms on the right separately.  

\noindent \emph{The first term.}   We  define $\Psi = \partial^* G_\infty \cap \Omega$ for brevity.     Since $G_\infty$ is contained in $U = \omega \times (-\frac{1}{2},M)$ and $\Omega = \omega  \times (-1,M+1)$, we observe $\Psi = \partial^* G_\infty \cap (\omega \times \R)$.   Choose $\omega_\Psi \subset \omega$ such that $\omega_\Psi \times \lbrace 0 \rbrace $ is the orthogonal projection of $\Psi$ onto $\R^{d-1}\times \lbrace 0 \rbrace$. Note that $\Psi$ and $\omega_\Psi$ satisfy 
\begin{equation*}\label{eq: needed later3}
\mathcal{H}^0\big(  \Psi^{e_d}_y \big) \ge 2 \ \ \  \text{ for all $y \in \omega_\Psi \times \lbrace 0\rbrace$,}
\end{equation*} 
since $G_\infty$ is a set of finite perimeter. Thus
\begin{align}\label{eq: summary1}
 \int_{\omega \times \lbrace 0 \rbrace}  \mathcal{H}^0\big(   (\partial^* G_\infty\cap \Omega )^{e_d}_y \big)  \, \d \mathcal{H}^{d-1}(y)   \ge 2 \mathcal{H}^{d-1}(\omega_\Psi)\,. 
\end{align} 
\emph{The second term.} As $\partial^* \Omega_h \cap \Omega$ is the (generalized) graph of the function $h\colon \omega \to [0,M]$, we have 
\begin{align}\label{eq: summary2}
\int_{\omega \times \lbrace 0 \rbrace}  \mathcal{H}^0\Big(   \big( \partial^* \Omega_h \cap \Omega \big)^{e_d}_y \Big)  \, \d \mathcal{H}^{d-1}(y)  =  \mathcal{H}^{d-1}(\omega)\,. 
\end{align} 
In a similar fashion, letting $\Lambda_2 = (\partial^* \Omega_h \cap\Omega) \setminus (G_\infty)^0$ and  denoting by $\omega_{\Lambda_2}\subset \omega$ its orthogonal projection onto  $\R^{d-1} \times \lbrace 0 \rbrace$,  we get  
\begin{align}\label{eq: summary3}
 \int_{\omega \times \lbrace 0 \rbrace}  \mathcal{H}^0\Big(   \big((\partial^* \Omega_h \cap\Omega) \setminus (G_\infty)^0 \big)^{e_d}_y \Big)  \, \d \mathcal{H}^{d-1}(y)  =  \mathcal{H}^{d-1}(\omega_{\Lambda_2})\,. 
\end{align} 
 As $\Lambda_2$ is contained in $(G_\infty)^1 \cup \partial^* G_\infty$, we get  $\omega_{\Lambda_2} \subset \omega_\Psi$,  see Figure~1. 
\begin{figure}[h]\label{figp29copia}
\includegraphics[scale=2]{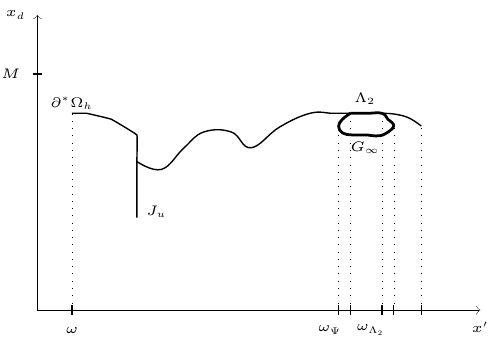}
 \caption{A picture of the situation in the argument by contradiction. We show that in fact $G_\infty=\emptyset$.} 
\end{figure} 
Therefore, by combining  \eqref{eq: summary2} and \eqref{eq: summary3} we find 
\begin{align}\label{eq: summary4}
 \int_{\omega \times \lbrace 0 \rbrace} \hspace{-0.2cm}  \mathcal{H}^0\Big(   \big(\partial^* \Omega_h  \cap \Omega   \cap (G_\infty)^0 \big)^{e_d}_y \Big)  \, \d \mathcal{H}^{d-1}(y) & = \mathcal{H}^{d-1}(\omega) -  \mathcal{H}^{d-1}(\omega_{\Lambda_2}) \ge \mathcal{H}^{d-1}(\omega) - \mathcal{H}^{d-1}(\omega_\Psi)\,.
\end{align} 
Now \eqref{eq: summary5}, \eqref{eq: summary1}, \eqref{eq: summary4}, and the fact that $\partial^*G_\infty \cap (G_\infty)^0 = \emptyset$ yield 
\begin{align}\label{eq: summary7}
\int_{\omega \times \lbrace 0 \rbrace} \mathcal{H}^{0}(\Gamma^{e_d}_y)  \, \d \mathcal{H}^{d-1}  & \ge  \int_{\omega \times \lbrace 0 \rbrace}\Big(\mathcal{H}^0\big(   (\partial^* G_\infty\cap \Omega )^{e_d}_y \big) +   \mathcal{H}^0\Big(   \big(\partial^* \Omega_h  \cap \Omega     \cap (G_\infty)^0 \big)^{e_d}_y \Big) \Big) \, \d \mathcal{H}^{d-1} \notag\\
& \ge \mathcal{H}^{d-1}(\omega) + \mathcal{H}^{d-1}(\omega_\Psi)\,.
\end{align}
Since $\Gamma_n$ are graphs of the functions $h_n \colon \omega\to [0,M]$, we get by  the area formula and \eqref{eq: aniso-lsc} that
\begin{align*}
\int_{\omega \times \lbrace 0 \rbrace} \mathcal{H}^{0}(\Gamma^{e_d}_y)  \, \d \mathcal{H}^{d-1}(y) & =  \int_\Gamma |\nu_\Gamma \cdot e_d| \d \mathcal{H}^{d-1} \le  \liminf_{n \to \infty} \int_{\Gamma_n} |\nu_{\Gamma_n} \cdot e_d|  \dh  =  \mathcal{H}^{d-1}(\omega) \,.
\end{align*}
This along with \eqref{eq: summary7} shows that $\mathcal{H}^{d-1}(\omega_\Psi) = 0$.  By recalling that $\omega_\Psi \times \lbrace 0 \rbrace$ is  the orthogonal projection of $\partial^* G_\infty \cap  (\omega \times \R)  = \Psi$ onto $\R^{d-1} \times \lbrace 0 \rbrace$, we conclude that $G_\infty = \emptyset$.  

  This  completes the proof of  the lower inequality in  Theorem~\ref{thm:relG}.  \eop

 \paragraph*{\textbf{The upper inequality.}} \label{page:upperineqbeg} To obtain the upper inequality, it clearly suffices to prove the following result.  
\begin{proposition}\label{prop: enough}
 Suppose that $f \ge 0$ is convex and   satisfies  \eqref{eq: growth conditions}. Consider $(u,h)$ with $ u = u \chi_{\Omega_h} \in~GSBD^p(\Omega)$, $u=u_0$ in $\omega{\times}(-1,0)$, and  $h \in BV(\omega;  [0,M]  )$. Then,  there exists a sequence $(u_n, h_n)_n$ with  $h_n \in C^1(\omega) \cap BV(\omega;  [0,M]  )$, $ u_n|_{\Omega_{h_n}}  \in W^{1,p}(\Omega_{h_n};  \R^d  )$,   $u_n=0$ in $\Omega \sm \Omega_{h_n}$,  and  $u_n=u_0$ in $\omega{\times}(-1,0)$ such that $u_n \to u$ in $L^0(\Omega;\R^d) $, $h_n \to h$ in $L^1(\omega)$, and 
\begin{subequations}\label{eq: sub}
\begin{align}
\limsup_{n\to \infty} \int_{\Omega_{h_n}} f(e(  u_n  ))  \, \dx  &\leq \int_{\Omega_h} f(e(u))  \, \dx  \,,\label{sub1}\\
\limsup\nolimits_{n\to \infty} \mathcal{H}^{d-1}(\partial \Omega_{h_n} \cap \Omega) & \leq \mathcal{H}^{d-1}(\partial^* \Omega_h \cap \Omega) + 2 \mathcal{H}^{d-1}(J_u' \cap \Omega_h^1)\,.\label{sub2}
\end{align}
\end{subequations}
\end{proposition}
In particular, it is not restrictive to assume that $f \ge 0$. In fact, otherwise we consider $\tilde{f} := f + c_2 \ge 0$  changing the value of the elastic energy  by the term $c_2\mathcal{L}^d(\Omega_h)$ which is continuous with respect to $L^1(\omega)$ convergence for $h$.  Moreover, the integrals $\Omega_{h_n}$ and $\Omega_h$ can be replaced by  $\Omega_{h_n}^+$ and $\Omega_h^+$, respectively, since all functions coincide with $u_0$ on $\omega \times (-1,0)$.

  \begin{remark}\label{rem:1805192117}
   The proof of the proposition will show that we can construct the sequence  $(u_n)_n$ also in such a way that  $u_n \in L^\infty(\Omega;\R^d)$ holds for all $n \in \N$. This, however, comes at the expense of the fact that the boundary data is only satisfied approximately, i.e.,  $u_n|_{\omega \times (-1,0)} \to u_0|_{\omega \times (-1,0)}$ in $W^{1,p}(\omega \times (-1,0);\R^d)$.  This slightly different version will be instrumental  in Subsection~\ref{sec:phasefield}.
\end{remark}

 As a preparation, we first state some auxiliary results. We recall two lemmas from \cite{ChaSol07}. The first is stated in \cite[Lemma~4.3]{ChaSol07}.
\begin{lemma}\label{le:4.3ChaSol}
Let $h \in BV(\omega; [0, +\infty))$, with $\partial^* \Omega_h$ essentially closed, i.e., $\hd(\ove{\partial^* \Omega_h} \sm \partial^* \Omega_h)=0$. Then, for any $\varepsilon>0$, there exists $g \in C^\infty(\omega; [0, +\infty))$ such that $g \leq h$ a.e.\ in $\omega$, $\|g-h\|_{L^1(\omega)} \leq \varepsilon$, and
\begin{equation*}
\bigg| \int_\omega \sqrt{1+|\nabla g|^2}\, \d  x'  - \hd(\partial^* \Omega_h \cap \Omega)   \bigg| \leq \varepsilon\,.
\end{equation*}
\end{lemma}

\begin{lemma}\label{lemma: graph approx}
Let $h \in BV(\omega; [0,M])$ and let $\Sigma \subset \R^{d}$ with $\mathcal{H}^{d-1}(\Sigma) < +\infty$ be vertical in the sense that $x=(x',x_d) \in \Sigma$ implies $(x',x_d + t) \in \Sigma$  as long as $(x',x_d + t) \in \Omega_h^1$.  Then, for each $\eps >0$ there exists $g \in C^\infty(\omega; [0,M])$   such that
\begin{subequations}
\begin{align}
\|g-h\|_{L^1(\omega)} &\leq \varepsilon\,,\label{2004192228}\\
\hd\big((\partial^* \Omega_h \cup \Sigma) \cap \Omega_{g}\big) & \leq \varepsilon\,,\label{2004192229}\\
\Big| \int_\omega \sqrt{1+|\nabla {g}|^2} \, \d  x'  - \big( \hd(\partial^*\Omega_h \cap \Omega) + 2 \hd(\Sigma)  \big)  \Big|& \leq \varepsilon\,.\label{2004192231}
\end{align}
\end{subequations}

\end{lemma}

\begin{proof}
We refer to the first step in the proof of \cite[Proposition~4.1]{ChaSol07}, in particular \cite[Equation (12)-(13)]{ChaSol07}. We point out that  the case of possibly unbounded graphs has been treated  there,  i.e., $h \in BV(\omega; [0,+\infty))$. The proof shows that the upper bound on $h$ is preserved and we indeed obtain $g \in C^\infty(\omega; [0,M])$ if $h \in BV(\omega; [0,M])$.
\end{proof}

\CCC Note that Lemma \ref{lemma: graph approx} states that $\partial^* \Omega_h \cup \Sigma$ can be approximated from below by a smooth graph $g$. However, this only holds \emph{up to a small portion}, see \eqref{2004192229}. Therefore, two additional approximation techniques are needed, one for graphs and one for $GSBD$ functions.     \EEE  To this end, we introduce some notation  which will also  be  needed  for the proof of Proposition \ref{prop: enough}. Let $k \in \N$,  $k>1$.  For any $z\in (2 \km) \mathbb{Z}^d$, consider  the hypercubes 
\begin{align}\label{eq: cube-notation}
q_z^k:=z+(-\km,\km)^d\,,  \qquad Q_z^k:=z+(-5\km,5\km)^d\,.
\end{align}
Given an open set $U \subset \R^d$, we also define the union of cubes well contained in $U$ by
\begin{align}\label{eq: well contained}
 (U)^k  := {\rm int}\Big( \bigcup\nolimits_{z\colon \, Q^k_z \subset U} \overline{q^k_z}\Big)\,. 
\end{align}
 (Here, ${\rm int}(\cdot)$ denotes the interior.  This definition   is unrelated  to the notation $E^s$ for the set of points with density $s \in [0,1]$.)  
 
 \CCC   We now address the two approximation results. We start by an approximation of graphs from which  a union of hypercubes has been removed. Recall $\Omega^+ = \Omega \cap \lbrace x_d>0 \rbrace$.


\begin{lemma}\label{lemma: new-lemma}
Let $g \in C^\infty(\omega;[0,M])$ and let $V_k \subset \Omega^+$ be a union of cubes $Q_z^k$, $z \in \mathcal{Z} \subset (2k^{-1}) \Z^d$, intersected with $(\Omega_g)^k$.  Suppose that  $V_k$ is vertical in the sense that $(x',x_d) \in V_k$ implies $(x',x_d + t) \in V_k$  for $t \ge 0$  as long as $(x',x_d + t) \in (\Omega_g)^k$.  Then, for $k \in \N$ sufficiently large,   we find a function $h_k\in C^\infty(\omega;[0,M])$ such that 
\begin{subequations} \label{eq: volume-a-b}
\begin{equation}\label{eq: volume2a}
\mathcal{L}^d( \Omega_g \, \triangle  \, \Omega_{{h}_k}   )  \le  \mathcal{L}^d(\Omega_g \cap V_k) + C_{g,\omega}k^{-1}\,,
\end{equation}
\begin{equation}\label{eq: volume2b}
 \hd(\partial \Omega_{{h}_k} \cap \Omega)  \le \hd(\partial \Omega_g \cap \Omega)  + \hd\big(\partial V_k \cap (\Omega_{g})^k\big)+ C_{g,\omega} k^{-1}\,,
\end{equation}
\end{subequations}
where $C_{g,\omega}>0$ depends on $d$, $g$, and $\omega$, but is independent of  $k$. Moreover, there are  constants $\tau_g,\tau_*>0$ only depending on $d$, $g$, and $\omega$ such that   
\begin{align}\label{eq: inclusion}
x =(x',x_d) \in \Omega_{h_k} \  \ \ \Rightarrow \ \ \   \big(  (1-\tau_*/k)\,  x', (1-\tau_*/k) \, x_d - 6\tau_g/k         \big)   \in  (\Omega_g)^k  \setminus V_k\,.
\end{align}

\end{lemma} 
 
We point out that \eqref{eq: inclusion} means that $h_k$ lies below the boundary of  $(\Omega_g)^k  \setminus V_k$, up to a slight translation and dilation. We suggest to omit the proof of the lemma on first reading.

 \begin{proof}
 The proof relies on slight lifting and dilation of the set  $(\Omega_{g})^k\sm V_k $ along with an application of Lemma \ref{le:4.3ChaSol}.  Recall definition \eqref{eq: well contained}, and define $\omega_k \subset \omega \subset \R^{d-1}$ such that $(\omega \times \R)^k = \omega_k \times \R$. Since $\omega$ is  uniformly star-shaped  with respect to the origin,  see \eqref{eq: star-shaped},  there exists a universal constant $\tau_\omega>0$ such that 
\begin{align}\label{eq: stretching}
 \omega_k \supset (1-\tau k^{-1}) \, \omega \ \ \  \text{ for $\tau \ge \tau_\omega$}\,.
\end{align} 
Define $\tau_g := 1+ \sqrt{d} \max_{\omega} |\nabla {g}|$. For $k$ sufficiently large, it is elementary to check that
\begin{equation}\label{2104191048}
\Omega_{g} \cap (\omega_k \times (0,\infty)) \subset \big( (\Omega_{g})^k + 6 \,\tau_g k^{-1} e_d\big)\,.
\end{equation}
We now  ``lift''  the set $(\Omega_g)^k \setminus V_k$ upwards: define the function 
\begin{equation}\label{eq: htilde-def}
g'_k(x'):= \sup\big\{ x_d < g(x') \colon (x', x_d-6\,\tau_g/k) \in  (\Omega_{g})^k\sm V_k \big\}  \qquad\text{for }x' \in \omega_k\,.
\end{equation}
We observe that $g_k' \in BV(\omega_k;[0,M])$. Define $(\Omega)^k$ as in \eqref{eq: well contained} and, similar to \eqref{eq: graphi}, we let $\Omega_{g_k'} = \lbrace x\in \omega_k \times (-1, M+1)\colon \,  -1 <   x_d   <    g_k'(x') \rbrace$. Since $V_k$ is vertical, we note  that $\partial \Omega_{g_k'}\cap (\Omega)^k$ is made of two parts: one part is contained in the smooth graph of $g$ and the rest in the boundary of $V_k + 6\tau_g  k^{-1}  e_d$.  In particular, by  \eqref{2104191048} we get  
\begin{equation*}
\partial \Omega_{g'_k} \cap (\Omega)^k \subset (\partial \Omega_{g}\cap\Omega) \cup \big(\partial \Omega_{g'_k}  \cap (\Omega)^k\cap \Omega_{g} \big) \subset (\partial \Omega_{g}\cap\Omega) \cup \big(( \partial V_k  \cap (\Omega_{g})^k) + 6\,\tau_g k^{-1}\, \,e_d  \big)\,.
\end{equation*}
Then, we deduce
\begin{equation}\label{2104192337}
 \hd(\partial \Omega_{g'_k} \cap (\Omega)^k  ) \le \hd(\partial \Omega_g \cap \Omega)  + \hd\big(\partial V_k \cap (\Omega_{g})^k\big)\,.
\end{equation}
Since by \eqref{2104191048} and \eqref{eq: htilde-def} there holds $(\Omega_g \setminus \Omega_{g_k'})\cap (\Omega)^k \subset V_k  + 6\,\tau_g k^{-1} e_d$,   the fact that  $V_k$ is vertical  implies
\begin{align}\label{eq: htilde}
\mathcal{L}^d\big((\Omega_g \triangle \Omega_{g_k'})\cap (\Omega)^k\big) &  \le \mathcal{L}^d\big(\Omega_g \cap (V_k  + 6\,\tau_g k^{-1} e_d)\big)  \le \mathcal{L}^d(\Omega_g \cap V_k)\,.
\end{align}
As $g_k'$ is only defined on $\omega_k$, we further need a dilation: letting $\tau_* := \tau_\omega \vee  (6\tau_g +  6)   $ and recalling \eqref{eq: stretching} we define $g_k'' \in BV(\omega;[0,M])$ by 
\begin{align}\label{eq: g''}
g_k''(x') = g_k'(  (1-\tau_* k^{-1})  \,   x') \ \ \text{ for } \ \ x' \in \omega\,.
\end{align}
 (The particular choice of $\tau_*$  will become clear  in the proof of  \eqref{eq: inclusion} below.)   By  \eqref{eq: htilde} we get
\begin{subequations} \label{eq: volume3}
\begin{equation}
\mathcal{L}^d( \Omega_g \, \triangle\,   \Omega_{g_k''}   )  \le  \mathcal{L}^d(\Omega_g \cap V_k)  + C_{g,\omega}k^{-1}\,,
\end{equation}
\begin{equation}
\big| \hd(\partial \Omega_{{g}''_k} \cap\Omega) -  \hd(\partial \Omega_{g'_k} \cap  (\Omega)^k   )  \big| \leq C_{g,\omega} k^{-1}\,,
\end{equation}
\end{subequations} 
where the constant $C_{g,\omega}$ depends only on $d$,  $g$, and $\omega$.   We  also notice that  $\mathcal{H}^{d-1} \big( \overline{\partial^* \Omega_{g_k''}} \setminus \partial^* \Omega_{g_k''} \big) = 0$. Then by Lemma \ref{le:4.3ChaSol} applied for $\eps = 1/k$ we find a function $h_k\in C^\infty(\omega;[0,M])$ with $h_k \le g''_k$ on $\omega$ such that 
\begin{align}\label{eq: htilde2}
\|g_k''-h_k\|_{L^1(\omega)} &\leq k^{-1}\,, \ \ \  \ \ \ \ \ \ \ \big|  \mathcal{H}^{d-1}(\partial {\Omega_{h_k}} \cap \Omega    )       - \mathcal{H}^{d-1}(\partial^* {\Omega_{g_k''}} \cap \Omega    )       \big| \le k^{-1} \,.
\end{align}
By passing to a larger constant $C_{\omega,g}$ and by using  \eqref{2104192337},  \eqref{eq: volume3}, and \eqref{eq: htilde2},    we get \eqref{eq: volume-a-b}. We finally show \eqref{eq: inclusion}.  
In view of the definitions of $g_k'$ and $g_k''$ in \eqref{eq: htilde-def} and \eqref{eq: g''}, respectively, and the fact that  $h_k \le g_k''$, we get
$$ x =(x',x_d) \in \Omega_{h_k} \  \ \ \Rightarrow \ \ \   \big(  (1-\tau_*/k)\,  x',  \, x_d - 6\tau_g/k         \big)   \in  (\Omega_g)^k  \setminus V_k\,. $$ 
   Recall $\tau_* = \tau_\omega \vee  (6\tau_g +  6)  $ and observe that $-(1-\tau_*/k)  - 6\tau_g/k \ge - 1 + 6  /k$.  Also note that $(\Omega_g)^k \supset (\Omega)^k \cap (\omega \times (-1+6/k,0))$, cf.\ \eqref{eq: well contained}.  This along with the verticality of $V_k\subset \Omega^+$  shows    
$$
x =(x',x_d) \in \Omega_{h_k} \  \ \ \Rightarrow \ \ \   \big(  (1-\tau_*/k)\,  x', (1-\tau_*/k) \, x_d - 6\tau_g/k         \big)   \in  (\Omega_g)^k  \setminus V_k\,.
$$
This concludes the proof. 
\end{proof}

 \EEE

\CCC Next, we  present an approximation technique for $GSBD$ functions \EEE based on  \cite{CC17}.  In the following, $\psi\colon  [0,\infty)  \to  [0,\infty) $ denotes the function $\psi(t) = t \wedge 1$. 
\begin{lemma}\label{lemma: vito approx}
 Let $U \subset \R^d$  be  open, bounded, $p >1$, and $k \in \N$,   $\theta\in (0,1)$ with $k^{-1}$, $\theta$  small enough.  
   Let $\mathcal{F} \subset GSBD^p(U)$ be such that $\psi(|v|) + |e(v)|^p$ is equiintegrable for $v \in \mathcal{F}$. 
  Suppose that for $v \in \mathcal{F}$  there exists a 
set of finite perimeter
 $V \subset U$ such that 
 for each $q^k_z$, $z \in (2k^{-1})\Z^d$, intersecting $(U)^k \setminus V$, there  holds that 
\begin{equation}\label{eq: cond1}
\hd\big(Q^k_z \cap  J_v\big)  \leq \theta k^{1-d}\,.
\end{equation}
Then there exists a  function $w_k \in W^{1,\infty}( (U)^k \setminus V; \R^d)$  such that  
%
\begin{subequations}\label{rough-dens}
\begin{equation}\label{rough-dens-1}
\int_{(U)^k \setminus V} \psi(|w_k-v| )\dx \leq R_k\,,
\end{equation}
\begin{equation}\label{rough-dens-2}
\int_{(U)^k \setminus V} |e(w_k)|^p \dx \leq \int_U  |e( v)|^p \dx + R_k\,.
\end{equation}
\end{subequations}   
 where $(R_k)_k$ is a sequence independent of $v \in \mathcal{F}$ with $R_k \to 0$ as $k \to \infty$. 
\end{lemma}

%

The lemma is essentially a consequence of the rough estimate proved in \cite[Theorem 3.1]{CC17}. 
For the convenience of the reader, we include a short proof 
in   Appendix~\ref{sec:App}.

\CCC After having collected auxiliary lemmas, we now give a short outline of the proof. \EEE Recall that $\Omega = \omega \times (-1,M+1)$ for given $M>0$.  Consider a pair $(u,h)$ as in Proposition~\ref{prop: enough}. We work with $u \chi_{\Omega_h} \in  GSBD^p(\Omega)$ in the following, without specifying each time that $u=0$ in the complement of $\Omega_h$.  Recall $J_u'$ defined in \eqref{eq: Ju'}, and, as before, set $\Sigma:= J_u' \cap \Omega_h^1$.  This implies  
$
J_u \subset (\partial^* \Omega_h  \cap \Omega) \cup \Sigma$. Since $\Sigma$ is vertical, we can   approximate $(\partial^*\Omega_h \cap \Omega ) \cup \Sigma$ by  the graph  of a smooth function $g\in C^\infty(\omega;[0,M])$ in the sense of Lemma \ref{lemma: graph approx}.

Our goal is to construct a regular approximation of $u$ in  (most of)  $\Omega_{g}$ by means of Lemma~\ref{lemma: vito approx}. The main step is to identify suitable exceptional sets $(V_k)_k$ such that for the cubes outside of $(V_k)_k$ we can verify  \eqref{eq: cond1}.  In this context, we emphasize that it is crucial that each $V_k$  is vertical \CCC since this allows us to apply Lemma \ref{lemma: new-lemma} and to approximate the boundary of $(\Omega_g)^k \setminus V_k$ from below by a smooth graph. \EEE    Before we start with the actual proof of  Proposition~\ref{prop: enough}, \CCC we address the construction of $(V_k)_k$. To this end, \EEE we introduce the notion of  good and bad nodes, and collect some important properties.

Define the set of nodes
\begin{align}\label{eq: nodes}
\mathcal{N}_k := \lbrace z \in (2 \km) \mathbb{Z}^d\colon \,  \overline{q^k_z}  \subset \Omega_g \rbrace\,.
\end{align}
Let us introduce the families of \emph{good nodes}  and \emph{bad nodes}  at level $k$. 
Let $\rho_1,\rho_2>0$ to be specified below.  By $\mathcal{G}_k$ we denote the set of good nodes $z \in \mathcal{N}_k$,  namely those  satisfying  
\begin{equation}\label{0306191212-1}
\hd\big(\overline{q^k_z} \cap (\partial^* \Omega_h \cup  \Sigma) \big)   \leq \rho_1 k^{1-d} 
\end{equation}
or having the property that there exists a set of finite perimeter  $F^k_z \subset q^k_z$, 
such that 
\begin{equation}\label{0306191212-2}
q^k_z \cap \Omega^0_h  \subset (F^k_z)^1, \ \ \ \ \ \  \mathcal{H}^{d-1}\big(\partial^* F^k_z  \big) \le  \rho_2 k^{1-d}, \ \  \ \ \  \hd\big(\overline{q^k_z} \cap   \Sigma  \cap  (F^k_z)^0  \big)    \leq \rho_2 k^{1-d}\,.      
\end{equation}
We define the set of bad nodes by $\mathcal{B}_k = \mathcal{N}_k \setminus \mathcal{G}_k$.  Moreover, let
\begin{equation}\label{1906191702}
 {\mathcal{G}}_k^*  :=\{ z \in   \mathcal{G}_k \colon \eqref{0306191212-1} \text{ does not hold}  \}\,.
\end{equation}  
For an illustration of the cubes in $\mathcal{G}_k$  we refer to Figure~2.  

\begin{figure}[h]\label{figp32}
\hspace{-3em}
\begin{minipage}[c]{0.5\linewidth}
\hspace{8em}
\includegraphics[scale=2]{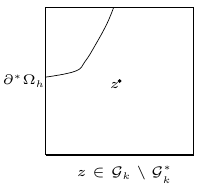}
\end{minipage}
\hfill
\begin{minipage}[c]{0.5\linewidth}
\includegraphics[scale=2]{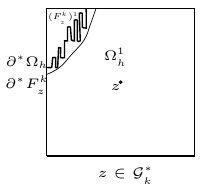}
\end{minipage}
\caption{A simplified representation of nodes in $\mathcal{G}_k$, for $d=2$ and with $\Sigma=\emptyset$.  The set  ${\mathcal{G}}_k \setminus  {\mathcal{G}}_k^*$ corresponds to the cubes containing only a small portion of $\partial^* \Omega_h \cup  \Sigma$, see first picture. For the cubes ${\mathcal{G}}_k^*$, the portion of  $\partial^* \Omega_h$  is contained in a set $F^k_z$ with small boundary, see second picture. Intuitively, this along with the fact that \eqref{0306191212-1} does not hold means that   $\partial^* \Omega_h$    is highly oscillatory in such cubes.} 
\end{figure} 
We partition the set of good nodes $\mathcal{G}_k$ into 
\begin{equation}\label{0406191145}
\mathcal{G}_k^1 = \big\{ z \in  \mathcal{G}_k\colon\, \Ld(\qz \cap \oho) \le \Ld(\qz \cap \ohi) \big\},  \ \ \ \ \ \ \ \mathcal{G}_k^2 =\mathcal{G}_k \setminus \mathcal{G}_k^1.
\end{equation}
We introduce the terminology ``$q^k_{z'}$ is above $q^k_z$'' meaning that $q^k_{z'}$ and $q^k_z$ have the same vertical projection onto $\R^{d-1} {\times} \{0\}$ and $z'_d > z_d$.

 We remark that bad nodes have been defined differently in \cite{ChaSol07}, namely as the cubes having an edge which intersects $\partial^*\Omega_h \cap \Sigma$. This definition is considerably easier than our definition. It may, however, fail in some pathological  situations  since, in this case, the union of cubes with bad nodes as centers does not necessarily form a ``vertical set''.

\begin{lemma}[Properties of good and bad nodes]\label{lemma: bad/good}
Given $\Omega_h$ and $\Sigma$, define $\Omega_g$ as in Lemma \ref{lemma: graph approx} for $\eps>0$ sufficiently small.  
We can choose  $0<\rho_1<\rho_2$  satisfying  $\rho_1,\rho_2 \le \frac{1}{2}5^{-d}  \theta$  such that the following properties hold for the good and bad nodes defined in \eqref{eq: nodes}--\eqref{0406191145}:
\begin{align*}
{\rm (i)} & \ \  \text{ if  $q^k_{z'}$ is above $q^k_z$ and $ z  \in  \mathcal{B}_k \cup \mathcal{G}^2_k$, then  $ z'  \in \mathcal{B}_k \cup \mathcal{G}^2_k$.} \\
{\rm (ii)} & \ \ \text{ if $z,z' \in \mathcal{G}_k$ with $\hd(\partial \qz \cap \partial q^k_{z'})>0$, then $z,z' \in \mathcal{G}^1_k$ or $z,z' \in \mathcal{G}^2_k$.}\\
{\rm (iii)} & \ \ \# \mathcal{B}_k  +  \# \mathcal{G}_k^* \le  2\rho_1^{-1}  k^{d-1} \varepsilon.\\
{\rm (iv)} &   \ \ \sum\nolimits_{z \in \mathcal{G}_k^2}\mathcal{L}^d(\Omega_h \cap q_z^k) \le   \varepsilon.  
\end{align*}

\end{lemma}

We suggest to omit the proof of the lemma on first reading and to proceed directly with the proof  of  Proposition \ref{prop: enough}.  

\begin{proof}
By $c_\pi \ge 1$ we denote the maximum of the constants appearing in the isoperimetric inequality and the relative isoperimetric inequality on a cube in dimension $d$. We will show the statement for $\eps$ and $0 < \rho_2 < 1$ sufficiently small satisfying $\rho_2 \le \frac{1}{2} 5^{-d} \theta$, and for $\rho_1 = ((3d+1)c_\pi)^{-1} \rho_2$.

\emph{Preparations.} First, we observe that for $\rho_2$ sufficiently small we have that  $\mathcal{G}_k^* \subset \mathcal{G}_k^1$.  Indeed,  since for $z \in \mathcal{G}_k^*$ property \eqref{0306191212-2} holds,  the isoperimetric inequality   implies
\begin{align}\label{eq: G1G2}
\Ld(\qz \cap \oho) \le \Ld(  F^k_z)  \le c_\pi \big( \mathcal{H}^{d-1}(\partial^*  F^k_z  ) \big)^{d/(d{-}1)} \le c_\pi \rho_2^{d/(d{-}1)} k^{-d}.
\end{align}
Then, for $\rho_2$ sufficiently small we get $\Ld(\qz \cap \oho)<\frac{1}{2} \Ld(\qz)$, and thus $z \in \mathcal{G}_k^1$, see \eqref{0406191145}. 

As a further preparation, we show that for each $z \in \mathcal{G}_k^1$ there exists a set of finite perimeter $H^k_z$ with   $\Omega_h^0 \cap q^k_z \subset H^k_z \subset q^k_z$ such that
\begin{align}\label{eq: setH}
\mathcal{L}^d(     H^k_z      )  \le c_\pi \rho_2^{d/(d{-}1)} k^{-d}, \ \ \  \mathcal{H}^{d-1}( \partial^* H^k_z) \le \rho_2 k^{1-d}, \ \ \ \   \hd\big(\overline{q^k_z} \cap   \Sigma  \cap  (H^k_z)^0  \big)    \leq \rho_2 k^{1-d}\,.   
\end{align}
Indeed, if  \eqref{0306191212-2} holds, this follows directly from \eqref{0306191212-2} and \eqref{eq: G1G2} for $H^k_z := F^k_z$.

Now suppose that $z \in \mathcal{G}_k^1$ satisfies \eqref{0306191212-1}. In this case, we define $H^k_z := \Omega_h^0 \cap q_z^k$. To control the volume, we use the relative isoperimetric inequality on $q^k_z$ to find by \eqref{0306191212-1}
\begin{align}\label{eq: vol part}
\mathcal{L}^d(     H^k_z      )  =  \mathcal{L}^d( \Omega^0_h \cap q^k_z    ) \wedge   \mathcal{L}^d( \Omega^1_h \cap q^k_z    ) \le c_\pi \big(\mathcal{H}^{d-1}(\partial^* \Omega_h \cap q^k_z)\big)^{d/(d{-}1)} \le c_\pi \rho_1^{d/(d{-}1)} k^{-d}\,,
\end{align}
i.e., the first part of \eqref{eq: setH} holds since  $\rho_1 = ((3d+1)c_\pi)^{-1} \rho_2$.  To obtain the second estimate in \eqref{eq: setH}, the essential step is to control $\mathcal{H}^{d-1}( \partial q^k_z \cap \Omega_h^0)$. For simplicity, we only estimate $\mathcal{H}^{d-1}( \partial_d q^k_z \cap \Omega_h^0)$ where $ \partial_d q^k_z$ denotes the two faces of  $\partial q^k_z$ whose normal vector is  parallel to  $e_d$. The other faces can be treated in a similar fashion. Write $z=(z',z_d)$ and define $\omega_z = z' + (-k^{-1},k^{-1})^{d-1}$.  By $\omega_* \subset \omega_z$ we denote the largest measurable set such that the cylindrical set $(\omega_* \times \R) \cap q^k_z$ is contained in $\Omega^0_h$. Then by the area formula (cf.\ e.g.\ \cite[(12.4) in Section~12]{Sim84}) and by recalling notation \eqref{eq: vxiy2}  we get
\begin{align}\label{eq: area-est}
\mathcal{H}^{d-1}( \partial_d q^k_z \cap \Omega_h^0) &\le 2 \mathcal{H}^{d-1}(\omega_*) + 2  \int_{ (\omega_z \setminus \omega_*) \times \lbrace 0 \rbrace} \mathcal{H}^0\big(  (\partial^* \Omega_h)^{e_d}_y \big) \, \d \mathcal{H}^{d-1}(y)\notag \\
& \le 2 \mathcal{H}^{d-1}(\omega_*)  + 2\int_{\partial^* \Omega_h \cap q^k_z} |\nu_{\Omega_h} \cdot e_d| \dh \notag \\
& \le 2 \mathcal{H}^{d-1}(\omega_*)  +2\mathcal{H}^{d-1}(\partial^* \Omega_h \cap q^k_z)\,.
\end{align}
As $(\omega_* \times \R) \cap q^k_z \subset \Omega^0_h \cap q^k_z$ and $\mathcal{L}^d(\Omega_h^0 \cap q^k_z)  \le c_\pi \rho_1^{d/(d{-}1)} k^{-d}$  by \eqref{eq: vol part}, we deduce $2k^{-1}\mathcal{H}^{d-1}(\omega_*) \le  c_\pi \rho_1^{d/(d{-}1)} k^{-d}$. This along with \eqref{0306191212-1} and \eqref{eq: area-est} yields
$$\mathcal{H}^{d-1}( \partial_d q^k_z \cap \Omega_h^0) \le c_\pi \rho_1^{d/(d{-}1)} k^{1-d} + 2\rho_1 k^{1-d} \le 3c_\pi\rho_1 k^{1-d}\,.$$
By repeating this argument for the other faces and by recalling $\mathcal{H}^{d-1}(\partial^* \Omega_h \cap q^k_z) \le \rho_1 k^{1-d}$, we conclude that $H^k_z = \Omega_h^0 \cap q_z^k$ satisfies
$$\mathcal{H}^{d-1}(\partial^* H^k_z) = \mathcal{H}^{d-1}(\partial^* \Omega_h \cap q^k_z) + \mathcal{H}^{d-1}(\partial q^k_z \cap \Omega_h^0) \le    \rho_1 k^{1-d} +  d   \cdot 3c_\pi\rho_1 k^{1-d}\le\rho_2k^{1-d}\,, $$
where the last step follows from  $\rho_1 = ((3d+1)c_\pi)^{-1} \rho_2$. This concludes the  the second part of  \eqref{eq: setH}. The third part follows from \eqref{0306191212-1} and $\rho_1 \le\rho_2$.   We are now in a position to prove the statement.

\emph{Proof of (i).} We need to show that for $z' \in \mathcal{G}_k^1$ there holds $z \in \mathcal{G}_k^1$ for   all  $z \in \mathcal{N}_k$ such that  $q^k_{z'}$ is above $q^k_{z}$. Fix such cubes $q_z^k$ and $q^k_{z'}$. 


Consider the set  $H^k_{z'}$  with   $\Omega_h^0 \cap q^k_{z'} \subset H^k_{z'} \subset q^k_{z'}$ introduced in \eqref{eq: setH}, and define  $F^k_{z}  := H^k_{z'} - z' + z$.   Since $\Omega_h$ is a generalized graph, we get $(F^k_z)^1 \supset \Omega_h^0 \cap q^k_z$.
 Moreover, since $\Sigma  = J_u' \cap \Omega_h^1  $ is vertical in $\Omega_h$,  see \eqref{eq: Ju'},  and $(H^k_{z'})^0 \subset \Omega_h^1 \cap q^k_{z'}$,  we have 
\begin{equation*}
\Sigma \cap (F^k_z)^0 = \Sigma \cap \Omega_h^1 \cap (F^k_z)^0 \subset (\Sigma \cap \Omega_h^1 \cap (H^k_{z'})^0) + z-z'= (\Sigma \cap (H^k_{z'})^0)+z-z'\,.
\end{equation*} 
 By \eqref{eq: setH} we thus get     $\hd( \ove{q^k_z} \cap \Sigma \cap (F^k_z)^0) \leq \hd( \ove{q^k_{z'}} \cap \Sigma \cap (H^k_{z'})^0)\leq  \rho_2  k^{1-d}$.
Then the third property in \eqref{0306191212-2} is satisfied for $z$. 
 Again  by \eqref{eq: setH} we note that  also the first two properties of  \eqref{0306191212-2} hold, and thus $z \in \mathcal{G}_k$. Using  once more  that $\Omega_h$ is a generalized graph, we get $\mathcal{L}^d(\Omega_h\cap q_z^k) \ge \mathcal{L}^d(\Omega_h\cap q_{z'}^k)$. Then $z' \in \mathcal{G}_k^1$  implies $z \in \mathcal{G}_k^1$, see \eqref{0406191145}.  This shows (i).

\emph{Proof of (ii).} Suppose by contradiction that there exist  $z\in \mathcal{G}_k^1$ and $z' \in \mathcal{G}_k^2$ satisfying     $\hd(\partial \qz \cap \partial q^k_{z'})>0$. Define the set $F := H^k_z \cup (\Omega^0_h \cap q^k_{z'})$ with $H^k_z$ from \eqref{eq: setH}, and observe that $F$ is contained in the cuboid $q^k_* = {\rm int}(\overline{q^k_z} \cup \overline{q^k_{z'}})$. Since $H^k_z \supset \Omega_h^0 \cap q^k_z$, we find
$$\mathcal{H}^{d-1}(q^k_* \cap \partial^*F) \le \mathcal{H}^{d-1}(\partial^* H^k_z) + \mathcal{H}^{d-1}(\partial^* \Omega_h \cap q^{k}_{z'}) \,.$$
As $\mathcal{G}^2_k \cap \mathcal{G}_k^* = \emptyset$, cf.\ \eqref{eq: G1G2},  for $z' \in \mathcal{G}^2_k$ estimate \eqref{0306191212-1} holds true. This along with  \eqref{eq: setH} yields
$$\mathcal{H}^{d-1}(q^k_* \cap \partial^*F)    \le  \rho_2 k^{1-d}  +  \rho_1 k^{1-d} \le 2\rho_2 k^{1-d}    \,.$$
 Then, the relative isoperimetric inequality on $q^k_*$ yields
\begin{align}\label{eq: contradiction}   
\Ld(q^k_* \cap F) \wedge \Ld(q^k_* \setminus F) \le C_* \big(\mathcal{H}^{d-1}(q^k_* \cap \partial^*F)\big)^{d/(d{-}1)} \le C_* (2\rho_2)^{d/(d{-}1)} k^{-d} 
\end{align}
for some universal $C_*>0$. On the other hand,  there holds $\Ld(q^k_* \cap F) \ge  \mathcal{L}^d(\Omega^0_h \cap q^k_{z'}) \ge \frac{1}{2} (2k^{-1})^d $ and $\Ld(q^k_* \setminus F)\ge \mathcal{L}^d(q^k_z \setminus H^k_z ) \ge (2k^{-1})^d - c_\pi \rho_2^{d/(d{-}1)} k^{-d}$ by \eqref{eq: setH}. However, for $\rho_2$ sufficiently small, this contradicts \eqref{eq: contradiction}. This concludes the proof of (ii).

\emph{Proof of (iii).}  Note that $\mathcal{H}^{d-1}$-a.e.\ point in $\R^d$ is contained in at most two different closed cubes $\overline{q^k_z}$, $\overline{q^k_{z'}}$. Therefore, since the cubes with centers in  $\mathcal{G}_k^*$ and    $\mathcal{B}_k$  do not satisfy \eqref{0306191212-1}, we get 
\begin{align*}
  \# \mathcal{B}_k  + \# \mathcal{G}_k^*   \leq  \rho_1^{-1} k^{d-1}  \hspace{-0.2cm}  \sum_{z \in \mathcal{B}_k \cup \mathcal{G}_k^*} \hd\big(\overline{q^k_z} \cap (\partial^* \Omega_h \cup  \Sigma) \big)  \le   2 \rho_1^{-1} k^{d-1}   \hd\big((\partial^* \Omega_h \cup \Sigma) \cap  {\Omega_g}    \big)\,,
\end{align*}
 where the last step follows from \eqref{eq: nodes}.  This along with \eqref{2004192229}   shows  (iii).

\emph{Proof of (iv).} Recall that each $z \in \mathcal{G}^2_k$ satisfies \eqref{0306191212-1}, cf.\  \eqref{1906191702} and before  \eqref{eq: G1G2}. The relative isoperimetric inequality,   \eqref{eq: nodes},   and \eqref{0406191145} yield
\begin{align*}
\sum\nolimits_{z \in \mathcal{G}_k^2}\mathcal{L}^d(\Omega_h \cap q_z^k) & = \sum\nolimits_{z \in \mathcal{G}_k^2}\mathcal{L}^d(\Omega^0_h \cap q_z^k) \wedge \mathcal{L}^d(\Omega^1_h \cap q_z^k) \le  c_\pi \sum\nolimits_{z \in \mathcal{G}_k^2}  \big(\mathcal{H}^{d-1}(\partial^* \Omega_h \cap q^k_z)\big)^{\frac{d}{d-1}}  \\
&\le    c_\pi  \Big(\sum\nolimits_{z \in \mathcal{G}_k^2}  \mathcal{H}^{d-1}(\partial^* \Omega_h \cap q^k_z)\Big)^{d/(d{-}1)} \le c_\pi \big( \mathcal{H}^{d-1}(\partial^*\Omega_h \cap \Omega_g) \big)^{d/(d{-}1)}\,.
\end{align*}
By \eqref{2004192229} we conclude for $\eps$ small enough that $\sum\nolimits_{z \in \mathcal{G}_k^2}\mathcal{L}^d(\Omega_h \cap q_z^k) \le c_\pi \eps^{d/(d{-}1)} \le\eps$.
 \end{proof}

\begin{proof}[Proof of  Proposition \ref{prop: enough}]
 Consider a pair $(u,h)$ and set $\Sigma := J_u' \cap \Omega_h^1$  with $J_u'$ as in \eqref{eq: Ju'}.  Given $\eps>0$, we  approximate $(\partial^*\Omega_h \cap \Omega ) \cup \Sigma$ by  the graph  of a smooth function $g\in C^\infty(\omega;[0,M])$ in the sense of Lemma \ref{lemma: graph approx}. Define the good and bad nodes as in  \eqref{eq: nodes}--\eqref{0406191145} for $0<\rho_1,\rho_2 \le \frac{1}{2}5^{-d}  \theta$ such that the properties in Lemma \ref {lemma: bad/good} hold. We will first define approximating regular graphs (Step 1) and  regular functions (Step 2) for fixed $\eps >0$. Finally, we let $\eps \to 0$ and obtain the result by a diagonal argument (Step 3). In the whole proof, $C>0$ will denote a constant depending only on $d$, $p$,  $\rho_1$,  and $\rho_2$.  

\noindent\emph{Step 1: Definition of regular graphs.} Recall \eqref{eq: well contained}. For each $k \in \N$, we define the set
\begin{equation}\label{0406191233}
V_k:= \bigcup\nolimits_{z \in \mathcal{G}_k^2 \cup \mathcal{B}_k} Q^k_z \cap  (\Omega_g)^k\,. 
\end{equation}
We observe that
\begin{equation}\label{0506192329}
\partial V_k \cap (\Omega_{g})^k \subset  \bigcup\nolimits_{z \in \mathcal{B}_k} \partial Q_z^k\,.
\end{equation}
In fact, consider  $z \in \mathcal{B}_k \cup \mathcal{G}_k^2$ such that $Q_z^k \cap V_k \neq \emptyset$ and one face of  $\partial Q^k_z$ intersects $\partial V_k \cap (\Omega_{g})^k$. In view of \eqref{0406191233}, there  exists an adjacent cube $q^k_{z'}$ satisfying  $\hd(\partial \qz \cap \partial q^k_{z'})>0$ and ${z'}  \in \mathcal{G}_k^1$ since otherwise $\partial Q^k_z \cap \partial V_k \cap (\Omega_{g})^k =\emptyset$. As  ${z'}  \in \mathcal{G}_k^1$, Lemma \ref{lemma: bad/good}(ii) implies $z  \notin \mathcal{G}_k^2$ and therefore $z \in \mathcal{B}_k$.  This shows  \eqref{0506192329}. A similar argument
yields 
\begin{equation}\label{0506192329-2}
V_k= \Big(\bigcup\nolimits_{z \in  \mathcal{B}_k} Q^k_z  \cup \bigcup\nolimits_{z \in \mathcal{G}_k^2}  q^k_z \Big) \cap (\Omega_g)^k
\end{equation}
up to a negligible set.  Indeed,  since  $V_k$ is a union of cubes of sidelength $2k^{-1}$ centered in nodes in $\mathcal{N}_k$,  it suffices to prove that  for a fixed 
$z \in \mathcal{N}_k \cap V_k$
 there holds   (a)  $z \in \mathcal{G}_k^2$ or that  (b)   there exists $z' \in \mathcal{B}_k$ such that $z \in Q_{z'}^k$.
 Arguing by contradiction, if $z \in \mathcal{N}_k \cap V_k$ and  neither (a) nor (b) hold,  we deduce that $z \in \mathcal{G}_k^1$ and $Q_z^k \cap \mathcal{B}_k=\emptyset$.  Then all  $z' \in \mathcal{N}_k\cap Q_z^k$ lie in $\mathcal{G}_k$. More precisely, by $z \in \mathcal{G}_k^1$ and Lemma~\ref{lemma: bad/good}(ii) we get that all  $z' \in \mathcal{N}_k\cap Q_z^k$ lie in $\mathcal{G}_k^1$.    Then $Q_z^k \cap (\mathcal{G}_k^2 \cup \mathcal{B}_k)=\emptyset$, so  that $q_z^k \cap V_k = \emptyset$ by \eqref{0406191233}.    This contradicts   $z \in  V_k$.

%

Let us now estimate the surface and volume of $V_k$. By \eqref{0506192329} and Lemma \ref{lemma: bad/good}(iii) we get 
\begin{equation}\label{0506192340}
\hd(\partial V_k \cap (\Omega_{g})^k) \leq  \sum\nolimits_{z \in \mathcal{B}_k} \mathcal{H}^{d-1}(\partial Q_z^k) \le C k^{1-d} \# \mathcal{B}_k \le   C \varepsilon\,,
\end{equation}
where $C$ depends on $\rho_1$. In a similar fashion, by  \eqref{0506192329-2} and  Lemma \ref{lemma: bad/good}(iii),(iv) we obtain
\begin{align}\label{eq: V volume}
\mathcal{L}^d(V_k\cap \Omega_h) \le Ck^{-d} \, \# \mathcal{B}_k  +  \sum\nolimits_{z \in \mathcal{G}_k^2} \mathcal{L}^d(q^k_z\cap \Omega_h) \le Ck^{-1}\eps + \eps \le C\eps\,.   
\end{align}
 Note that $V_k$ is vertical in the sense that $(x',x_d) \in V_k$ implies $(x',x_d + t) \in V_k$  for $t \ge 0$  as long as $(x',x_d + t) \in (\Omega_g)^k$.     This follows from Lemma \ref{lemma: bad/good}(i) and \eqref{0406191233}.  

\CCC 
We apply Lemma \ref{lemma: new-lemma} for $g$ and $V_k$  to find functions $h_k\in C^\infty(\omega;[0,M])$ satisfying \eqref{eq: volume-a-b} and \eqref{eq: inclusion}. Therefore, by 
  \eqref{2004192228}, \eqref{eq: volume-a-b}, and \eqref{eq: V volume}      we get
\begin{align}\label{eq: volume2a-new}
\mathcal{L}^d( \Omega_h \, \triangle  \, \Omega_{{h}_k}   ) & \le  \mathcal{L}^d( \Omega_g \, \triangle  \, \Omega_{{h}_k}   ) + \mathcal{L}^d(\Omega_g \triangle \Omega_h)  \le   \mathcal{L}^d(\Omega_g \cap V_k) + C_{g,\omega}k^{-1} + \mathcal{L}^d(\Omega_g \triangle \Omega_h) \notag \\ & \le  \mathcal{L}^d(\Omega_h \cap V_k) + C_{g,\omega}k^{-1} + 2\mathcal{L}^d(\Omega_g \triangle \Omega_h)  \le  C\eps + C_{g,\omega}k^{-1}\,.
\end{align}
Moreover, by  \eqref{2004192231}, \eqref{eq: volume-a-b}, and    \eqref{0506192340} we obtain
\begin{align}\label{eq: volume2b-new}
 \hd(\partial \Omega_{{h}_k} \cap \Omega)  & \le  \hd(\partial \Omega_g \cap \Omega)  + \hd\big(\partial V_k \cap (\Omega_{g})^k\big)+ C_{g,\omega} k^{-1} \notag \\ & \le \hd(\partial^* \Omega_h \cap \Omega) + 2 \, \hd(\Sigma)  + C\varepsilon + C_{g,\omega} k^{-1}\,.
\end{align}
\EEE
\noindent\emph{Step 2: Definition of regular functions.}  
Recall \eqref{0306191212-2}--\eqref{1906191702},  and observe that Lemma~\ref{lemma: bad/good}(iii)  implies
\begin{equation*}
\Ld(F^k)\leq  \sum\nolimits_{z \in {\mathcal{G}}_k^*} \mathcal{L}^d(q^k_z) \le  Ck^{-d} \# {\mathcal{G}}_k^* \le  C\varepsilon\,k^{-1}\,, \quad \text{ where  } \  F^k:= \bigcup\nolimits_{z \in {\mathcal{G}}_k^*} (F^k_z)^1\,.
\end{equation*} 
We define the  functions $v_k \in GSBD^p(\Omega)$  by   
\begin{equation}\label{2006191034}
 v_k:=  u (1-\chi_{F^k }) \, \chi_{\Omega_g}\,. 
\end{equation}
Since $u=0$ in $\Omega\sm \Omega_h$ and  $v_k=0$  in $\Omega\sm \Omega_g$, we get by \eqref{2004192228} and   \eqref{2006191034} 
\begin{equation}\label{2006190839}
 \limsup_{k\to \infty}  \Ld(\{ v_k \neq u\}) \leq \limsup_{k\to \infty} \Ld\big(F^k \cup (\Omega_h \sm \Omega_g)  \big)  \leq C\varepsilon\,.
\end{equation}
 We also obtain  
\begin{equation}\label{eq: smalljump}
\hd(Q^k_z \cap  J_{v_k} ) \leq \theta k^{1-d}
\end{equation}
for each $q^k_z$ intersecting $(\Omega_g)^k \sm V_k$.   To see this, note that the definitions of $\mathcal{N}_k$ in \eqref{eq: nodes} and of $V_k$ in \eqref{0406191233} imply that for each $q^k_z$ with $q^k_z\cap( (\Omega_g)^k \setminus V_k) \neq \emptyset$, each $z'\in \mathcal{N}_k$  with $q^k_{z'} \cap Q^k_z \neq \emptyset$  satisfies $z' \in \mathcal{G}_k$.  In view of  $\rho_1 <  \rho_2  \leq \frac{1}{2}5^{-d} \theta$ (see Lemma ~\ref{lemma: bad/good}), the property then  follows from  \eqref{0306191212-1}, \eqref{0306191212-2}, 
$ J_u  \cap \Omega_g \subset \partial^* \Omega_h \cup \Sigma$,  and  the fact that $Q^k_z$ consists of  $5^d$ different cubes $q^k_{z'}$. 

Notice that  $|v_k|\le |u|$ and $|e(v_k)| \le |e(u)|$ pointwise a.e., i.e., the functions $ \psi(|v_k|)  + |e(v_k)|^p$ are equiintegrable,   where $\psi(t) = t \wedge 1$.  In view of \eqref{eq: smalljump}, we can   apply Lemma \ref{lemma: vito approx} on $U = \Omega_g$   for the function $v_k \in GSBD^p(\Omega_g)$ and the sets $V_k$,    to get functions $w_k \in W^{1,\infty}( (\Omega_g)^k \setminus V_k;  \R^d)$ such that \eqref{rough-dens-1} and \eqref{rough-dens-2}  hold for a sequence $R_k \to 0$.

%
%
We now define the function  $\hat{w}_k\colon \Omega\to \R^d$  by   
 \begin{equation*}
\hat{w}_k(x):= \begin{dcases}
w_k\big((1-\tau_*/k) \, x', (1-\tau_*/k) \, x_d - 6\,\tau_g/k\big) \qquad&\text{if } -1 < x_d < {h}_k(x')\,,\\
0 &\text{otherwise.}
\end{dcases}
\end{equation*}
\CCC Note that, in view of \eqref{eq: inclusion}, \EEE the mapping is well defined and  satisfies  $ \hat{w}_k  |_{\Omega_{h_k}} \in W^{1,\infty}(\Omega_{h_k};\R^d)$. By   \eqref{2004192228} 
\eqref{rough-dens-1},  \eqref{eq: volume2a-new},   
\eqref{2006190839},  and $\psi \le 1$ 
we get
\begin{align}\label{eq: smalldiff}
\limsup_{k\to \infty}\Vert \psi(|\hat{w}_k - u|) \Vert_{L^1(\Omega)}  \le \limsup_{k\to \infty}\big(\Vert \psi(|\hat{w}_k - v_k|)\Vert_{L^1(\Omega)}  +   \Ld(\{ v_k \neq u\})\big)  \le C\eps\,. 
\end{align} 
 In a similar fashion, by  employing  \eqref{rough-dens-2}  in place of \eqref{rough-dens-1}   and by the fact that $\Vert  e(\hat{w}_k)\Vert_{L^p(\Omega)} \le (1+C_Mk^{-1}) \Vert e(w_k)\Vert_{L^p((\Omega_g)^k \setminus V_k)}$ for some $C_M$ depending on $M$ and $\tau_*$, we obtain   
\begin{align}\label{eq: elastic energy estimate}
\limsup_{k\to \infty} \int_{\Omega} |e(\hat{w}_k)|^p \dx & \le \limsup_{k\to \infty} \int_{(\Omega_g)^k \setminus V_k} |e(w_k)|^p \dx      \notag \\
& \le     \limsup_{k\to \infty}  \int_{\Omega_g} |e(v_k)|^p \dx  \le \int_{\Omega_h}  |e(u)|^p \dx  \,,
\end{align} 
 where  the  last step follows from \eqref{2006191034}.

\noindent\emph{Step 3: Conclusion.} Performing the construction above for $\eps = 1/n$, $n \in \N$, and choosing for each $n \in \N$ an index $k=k(n)\in \N$ sufficiently large, we obtain a sequence $( \hat{w}_{n},  h_{n})$ such that by  \eqref{eq: volume2a-new}    and  \eqref{eq: smalldiff}  we get  
\begin{equation}\label{2709192115}
\hat{w}_n \to u=u\chi_{\Omega_h}\text{ in }L^0(\Omega;\R^d) \quad \text{ and }\quad h_n \to h\text{ in }L^1(\omega)\,.
\end{equation} 
By \eqref{eq: volume2b-new}  and the definition $\Sigma=J_u' \cap \Omega_h^1$  we obtain \eqref{sub2}. By $GSBD^p$ compactness (see Theorem \ref{th: GSDBcompactness}) applied on $\hat{w}_n = \hat{w}_n\chi_{\Omega_{h_n}} \in  GSBD^p(\Omega)$ along with $\hat{w}_n \to u$ in $L^0(\Omega;\R^d)$  we get 
$$   \int_{\Omega_h} |e(u)|^p \dx \le \liminf_{n \to \infty} \int_{\Omega_{h_n}} |e(\hat{w}_n)|^p \dx.$$ 
 This along with \eqref{eq: elastic energy estimate} and the strict convexity of the norm $\Vert \cdot \Vert_{L^p(\Omega)}$  gives  
\begin{equation}\label{2006191122}
e(\hat{w}_n) \to e(u) \quad \text{in }L^p(\Omega; \Mdd)\,.
\end{equation}
 In view of \eqref{eq: growth conditions}, this shows the statement apart from the fact that the configurations $\hat{w}_n$ do possibly not satisfy the boundary data. (I.e., we have now proved the version described in Remark \ref{rem:1805192117} since $\hat{w}_n \in L^\infty(\Omega;\R^d)$.) It remains to adjust the boundary values. 

To this end,  choose a continuous extension operator from $W^{1,p}(\omega \times (-1,0);\R^d)$ to $W^{1,p}(\Omega;\R^d)$ and denote by $(w_n)_n$ the extensions of  $(\hat{w}_n-  u_0)|_{\omega \times (-1,0)}$  to $\Omega$. Clearly, $w_n \to 0$ strongly in $W^{1,p}(\Omega;\R^d)$ since $ (\hat{w}_n-  u_0)|_{\omega \times (-1,0)}   \to 0$  in $W^{1,p}(\omega \times (-1,0);\R^d)$. We now define the sequence $(u_n)_n$ by $u_n:= (\hat{w}_n - w_n)\chi_{\Omega_{h_n}}$.  By \eqref{2709192115} we immediately deduce $u_n \to u$ in $L^0(\Omega;\R^d)$.  Moreover,  $ u_n|_{\Omega_{h_n}}  \in W^{1,p}(\Omega_{h_n}; \Rd)$, $u_n=0$ in $\Omega \sm \Omega_{h_n}$,  $u_n = u_0$ a.e.\ in $\omega \times (-1,0)$ and 
\eqref{2006191122} 
still holds with $u_n$ in place of  $\hat{w}_n$.  Due to \eqref{eq: growth conditions},   this shows \eqref{sub1} and concludes the proof. 
\end{proof}

\begin{remark}[Volume constraint]\label{rem: volume constraint}
Given a volume constraint $\mathcal{L}^d(\Omega_h^+) = m$ with $0 <m < M\mathcal{H}^{d-1}(\omega) $, one can construct  the sequence $(u_n,h_n)$ in Proposition \ref{prop: enough}  such that also $h_n$ satisfies the volume constraint, cf.\ \cite[Remark~4.2]{ChaSol07}.   Indeed, if $\Vert h\Vert_{\infty}<M$,  we consider  $h^*_n(x') = r_n^{-1} h_n(x')$  and $u^*_n(x',x_d) = u_n(x', r_n  x_d)$, where $r_n := m^{-1}\int_\omega h_n \,\dx$. Then $\int_\omega h_n^* \,\dx =m$. Note that we can assume $\Vert h_n \Vert_\infty \le \Vert h \Vert_\infty$ (apply Proposition~\ref{prop: enough} with $\Vert h\Vert_{\infty}$ in place of $M$). Since $r_n \to 1$, we then find $h_n\colon \omega \to [0,M]$ for $n$ sufficiently large,  and \eqref{eq: sub} still holds.

If $\Vert h\Vert_{L^\infty(\omega)}=M$ instead, we need to perform a preliminary approximation: given $\delta>0$, define $\hat{h}^{\delta,M} = h \wedge (M- \delta)$ and   $h_\delta(x') = r_\delta^{-1} \hat{h}^{\delta,M}(x')$,  where $r_\delta = m^{-1} \int_\omega\hat{h}^{\delta,M} \, \dx$. Since $\Omega_h$  is  a subgraph and $m < M\mathcal{H}^{d-1}(\omega)$,  it is easy to check that $r_\delta > (M-\delta)/M$ and therefore $\Vert h_\delta\Vert_\infty < M$. Moreover, by construction we have $\int_\omega h_\delta \, \dx = m$. We  define $u_\delta(x',x_d) = u(x', r_\delta x_d) \chi_{\Omega_{h_\delta}}$. We now apply the above approximation on fixed $(u_\delta, h_\delta)$, then consider a sequence $\delta \to 0$, and use a diagonal argument.
\end{remark}  \label{page:upperineqend}


\begin{remark}[Surface tension]
We remark that, similar to \cite{BonCha02, ChaSol07, FonFusLeoMor07}, we could also derive a relaxation result for more general models where the surface tension  $\sigma_S$ for the substrate can be different from the the surface tension $\sigma_C$ of the crystal. This corresponds to surface energies of the form 
$$\sigma_S \,\hd(\{h=0\}) +  \sigma_C \hd\big(\partial\Omega_h \cap (\omega{\times}(0,+\infty))\big)\,.$$ 
In the relaxed  setting, the  surface energy is then given by 
\begin{equation*}
(\sigma_S \wedge \sigma_C) \,\hd(\{h=0\}) + \sigma_C \Big(\hd\big(\partial^*\Omega_h \cap (\omega{\times}(0,+\infty))\big) + 2\,\hd(J_u' \cap \Omega_h^1) \Big)\,.
\end{equation*}
We do not prove this fact here for simplicity, but refer to \cite[Subsection~2.4, Remark~4.4]{ChaSol07} for details how the proof needs to be adapted to deal with such a situation.
\end{remark}

\subsection{Compactness and existence of minimizers} \label{subsec:compactness}
In this short subsection we give the proof of the compactness result stated in Theorem~\ref{thm:compG}. As discussed in Subsection \ref{sec: results2}, this immediately implies the existence of minimizers for problem \eqref{eq: minimization problem2}.

\begin{proof}[Proof of Theorem~\ref{thm:compG}]
 Consider  $(u_n, h_n)_n$ with $\sup_n G(u_n, h_n) < +\infty$. First, by \eqref{eq: Gfunctional} and a standard compactness argument we find $h \in BV(\omega;  [0,M]  )$  such that $h_n \to h$ in $L^1(\omega)$, up to a subsequence (not relabeled). Moreover, by \eqref{eq: growth conditions}, \eqref{eq: Gfunctional}, and the fact that $J_{u_n} \subset \partial \Omega_{h_n}\cap \Omega$ we can apply  Theorem \ref{th: GSDBcompactness} to obtain some $u \in GSBD^p_\infty(\Omega)$ such that  $u_n \to u$ weakly in $GSBD^p_\infty$. We also observe that $u =u\chi_{\Omega_h}$ and $u = u_0$ on $\omega\times (-1,0)$ by \eqref{eq: GSBD comp}(i), $u_n =u_n\chi_{\Omega_{h_n}}$, and   $u_n = u_0$ on $\omega\times (-1,0)$ for all $n \in \N$. It remains to show that $u \in GSBD^p(\Omega)$, i.e., $ \lbrace u = \infty \rbrace = \emptyset$.
 
 To this end, we take $U = \omega \times (-\frac{1}{2},M) $  and $U' = \Omega = \omega \times (-1,M+1)$, and  apply Theorem~\ref{thm:compSps} on the sequence $\Gamma_n = \partial \Omega_{h_n} \cap \Omega$  to  find that $\Gamma_n$ $\sps$-converges  (up to a subsequence) to a pair $(\Gamma, G_\infty)$. Consider  $v_n = \psi u_n$, where $\psi \in C^\infty(\Omega)$ with $\psi = 1$ in a neighborhood of $ \omega \times (0,M+1) $ and $\psi = 0$ on $\omega  \times (-1,-\frac{1}{2})$. Clearly, $v_n$ converges weakly in $GSBD^p_\infty(\Omega)$ to $v:=\psi u$. As $J_{v_n} \subset \Gamma_n$ and $v_n = 0$ on $U' \setminus U$ for all $n \in \N$, \CCC we also obtain  $\lbrace v = \infty \rbrace \subset G_\infty$ (up to a $\Ld$-negligible set), see Definition~\ref{def:spsconv}(i). \EEE As by definition of $v$ we have $\lbrace u = \infty \rbrace  = \lbrace v = \infty \rbrace$, we deduce $\lbrace u = \infty \rbrace \subset G_\infty$. It now suffices to recall $G_\infty =  \emptyset$, see \eqref{eq: G is empty}, to conclude $\lbrace u = \infty \rbrace = \emptyset$. 
 \end{proof}

%
%
%

\subsection{Phase field approximation of $\ove G$}\label{sec:phasefield}

This final subsection  is  devoted to  the phase-field approximation of the functional  $\ove{G}$. Recall the  functionals  introduced in  \eqref{eq: phase-approx}.

\begin{proof}[Proof of Theorem~\ref{thm:phasefieldG}]  
 Fix a decreasing sequence $(\eps_n)_n$ of positive numbers converging to zero.  We first prove the liminf and then the  limsup inequality.

\noindent \emph{Proof of (i).} Let $(u_n, v_n)_n$ with  $\sup_n G_{\varepsilon_n}(u_n, v_n)< +\infty$.  Then, $v_n$ is nonincreasing in $x_d$,  and therefore 
\begin{equation*}
\widetilde{v}_n  (x)  :=0 \vee( v_n(x)-\delta_n x_d) \wedge 1 \quad  \text{ for $x \in \Omega  = \omega \times (-1,M+1)$}  
\end{equation*}
is strictly decreasing  on $\lbrace 0 < \widetilde{v}_n < 1\rbrace$,   where $(\delta_n)_n$ is a decreasing sequence of positive  numbers converging to zero. For a suitable choice of $(\delta_n)_n$, depending on $(\eps_n)_n$ and $W$, we obtain   $\| v_n - \widetilde{v}_n\|_{L^1(\Omega)} \to 0$  and 
\begin{equation}\label{1805192049} 
G_{\varepsilon_n}(u_n, v_n)= G_{\varepsilon_n}(u_n, \widetilde{v}_n) + O(1/n)\,.
\end{equation}     
By using  the implicit function theorem and  the coarea formula for $\widetilde{v}_n$, we can see, exactly as in  the proof of   \cite[Theorem~5.1]{ChaSol07}, that for a.e.\ $s \in (0,1)$  and $n \in \N$  the superlevel  set  $\{\widetilde{v}_n > s\}$ is the subgraph of a function $h^s_n \in H^1(\omega; [0,M])$. (Every $h^s_n$ takes values in $[0,M]$ since $\widetilde{v}_n=0$ in $\omega{\times} (M, M+1)$.) By the coarea formula for $\widetilde{v}_n$,  $\partial^*  \lbrace \widetilde{v}_n > s \rbrace\cap \Omega =\partial^* \Omega_{h_n^s} \cap \Omega$, and   Young's inequality we obtain
\begin{align*}
\int_0^1 \sqrt{2W(s)}\,  \mathcal{H}^{d-1}(\partial^* \Omega_{h_n^s} \cap \Omega) \,  \d s & \leq \int_\Omega \sqrt{2W(\widetilde{v}_n)} \,|\nabla \widetilde{v}_n|  \dx  \le  \int_\Omega \Big( \frac{\varepsilon_n}{2} |\nabla \widetilde{v}_n|^2 + \frac{1}{\varepsilon_n} W(\widetilde{v}_n)  \Big) \dx\,.
\end{align*}
Then, by Fatou's lemma we get 
\begin{align}\label{eq: coarea, fatou}
\int_0^1 \sqrt{2W(s)}\Big( \liminf_{n\to \infty} \int_\omega \sqrt{1+|\nabla h^s_n(x')|^2} \d x'\Big) \d s  \le \liminf_{n\to \infty} \int_\Omega \Big( \frac{\varepsilon_n}{2} |\nabla \widetilde{v}_n|^2 + \frac{1}{\varepsilon_n} W(\widetilde{v}_n)  \Big) \dx  < + \infty 
\end{align}
and  thus  $\liminf_{n\to \infty} \int_\omega \sqrt{1+|\nabla h^s_n(x')|^2} \,\d x'$ is finite for a.e.\ $s\in (0,1)$. 
By a diagonal argument, we can find a subsequence (still denoted by $(\varepsilon_n)_n$) and $(s_k)_k \subset (0,1)$ with $\lim_{k\to \infty} s_k =0$ such that for every $k \in \N$ there holds
\begin{equation}\label{eq: liminfequalinf}
\lim_{n\to \infty} \int_\omega \sqrt{1+|\nabla h^{s_k}_n(x')|^2} \,  \d x'=\liminf_{n\to \infty} \int_\omega \sqrt{1+|\nabla h^{s_k}_n(x')|^2} \,  \d x'  < + \infty \,.
\end{equation}
 Up  to a  further  (not relabeled) subsequence, we may     thus  assume  that $h^{s_k}_n$ converges in $L^1(\omega)$ to some function $h^{s_k}$ for every $k$.  Since $ \sup_n  G_{\varepsilon_n}(u_n, \widetilde{v}_n) < +\infty$ and thus $W(\widetilde{v}_n) \to 0$ a.e.\ in $\Omega$, we obtain $\widetilde{v}_n \to 0$ for a.e.\ $x$ with $x_d > h^{s_k}(x')$ and $\widetilde{v}_n \to 1$ for a.e.\ $x$ with $x_d < h^{s_k}(x')$.  (Recall $W(t) = 0 \Leftrightarrow t \in \lbrace 0,1\rbrace$.)  This shows that the functions $h^{s_k}$ are independent of $k$, and will be denoted simply by $h\in BV(\omega; [0,M])$.

Let us  denote by $u_n^k \in  GSBD^p(\Omega)$ the function given by 
\begin{align}\label{1805192011}
u^k_n(x) = \begin{cases} u_n(x) & \text{if } x_d < h^{s_k}_n(x')\,,\\ 0 & \text{else}\,. \end{cases}
\end{align}
Then $(u_n^k)_n$ satisfies the hypothesis of Theorem~\ref{th: GSDBcompactness} for every $k \in \N$. Indeed, $J_{u^k_n} \subset \partial^* \Omega_{h^{s_k}_n}$  and $\mathcal{H}^{d-1}(\partial^* \Omega_{h^{s_k}_n})$ is uniformly bounded in $n$ by \eqref{eq: liminfequalinf}. Moreover,   $(e(u^k_n))_n$ is uniformly bounded in $L^p(\Omega; \Mdd)$   by \eqref{eq: growth conditions} and the fact that   
\[
G_{\varepsilon_n}(u_n, \widetilde{v}_n) \geq (\eta_{\varepsilon_n} + s_k^2) \int _{\Omega}   f(e(u^k_n))  \dx\,.
\]
Therefore, Theorem~\ref{th: GSDBcompactness}  implies  that, up to a subsequence, $u^k_n$ converges weakly in $GSBD^p_\infty(\Omega)$ to a function $u^k$. Furthermore, we infer, arguing exactly as  in the proof of Theorem~\ref{thm:compG} above,  that  actually  $u^k \in GSBD^p(\Omega)$, i.e., the exceptional set   $\lbrace u^k  = \infty \rbrace$   is empty.  By \eqref{eq: GSBD comp}(i) this yields $u^k_n \to u^k$ in  $L^0(\Omega;\Rd)$. By a diagonal argument we get (up to a further subsequence)  that   $u_n^k \to u^k$ pointwise a.e.\ as $n \to \infty$ for all $k\in\N$. 

 Recalling now the definition of $u^k_n$  in  \eqref{1805192011} and the fact that $\lim_{n\to  \infty} \Vert h_n^{s_k} - h\Vert_{L^1(\omega)} = 0$ for all $k \in \N$, we deduce that the functions $u^k$ are independent of $k$.  This function will simply be denoted by $u \in GSBD^p(\Omega)$ in the following. Note that $u = u \chi_{\Omega_h}$  and that $u = u_0$ on $\omega \times (-1,0)$ since  $u_n = u_0$ on $\omega \times (-1,0)$ for all $n \in \N$.

 For the proof of \eqref{1805192018}, we can now follow exactly the lines of the lower bound in \cite[Theorem~5.1]{ChaSol07}. We sketch the main arguments for convenience of the reader. We first observe that 
$$\int_\Omega \widetilde{v}_n \, f(e(u_n)) \,\dx = \int_\Omega \Big(2\int_0^{\widetilde{v}_n(x)}s\,\d s \Big) \, f(e(u_n)(x)) \,\dx \ge \int_0^1 2s \Big(\int_{\lbrace\widetilde{v}_n >s \rbrace} f(e(u_n)) \, \d x\Big)\, \d s\,. $$
This along with  \eqref{eq: coarea, fatou} and  Fatou's lemma yields 
\begin{equation}\label{eq: last equation}
\int_0^1 \liminf_{n\to \infty}  \Big( 2s \int_{\{ \widetilde{v}_n > s \} } f(e(u_n)) \,\dx + c_W \sqrt{2 W(s)} \int _\omega \sqrt{1+|\nabla h^s_n|^2}\, \d x' \Big)\, \d s\leq \liminf_{n \to \infty} G_{\varepsilon_n}(u_n, \widetilde{v}_n) \,.
\end{equation}
Thus, the  integrand  
\[
I^s_n:= 2s \int_{\{ \widetilde{v}_n > s \} } f(e(u_n)) \dx + c_W \sqrt{2 W(s)} \int _\omega \sqrt{1+|\nabla h^s_n|^2} \,  \d x' 
\]
is finite for a.e.\ $s \in (0,1)$.  We then take $s$ such that $h_n^s \in H^1(\omega)$ for all $n$,  and consider a subsequence $(n_m)_m$  such that  $\lim_{m\to \infty} I_{n_m}^s = \liminf_{n\to \infty} I_n^s$. Exactly as in \eqref{1805192011}, we let $u^s_{n_m}$ be the function given by $u_{n_m}$ if $x_d< h^{s}_{n_m}(x')$ and by zero otherwise.  Repeating the compactness argument below  \eqref{1805192011},  we get $u^s_{n_m} \to u$ a.e.\ in $\Omega$  and  $h^s_{n_m} \to h$ in $L^1(\omega)$ as $m \to \infty$.  We observe that this can be done for a.e.\ $s \in (0,1)$, for a subsequence depending on $s$.

By   $\int_{\{ \widetilde{v}_{n_m} > s \} } f(e(u_{n_m})) \dx = \int_{\Omega} f(e(u^s_{n_m})) \dx $ and   the (lower inequality in the) relaxation result Theorem~\ref{thm:relG}  (up to different constants in front of the elastic energy and surface energy)  we obtain
\begin{equation*}
2s \int _{ \Omega^+_h  } f(e(u)) \dx + c_W \sqrt{2W(s)}\,\big( \hd(\partial^*\Omega_h  \cap \Omega  ) + 2 \, \hd(J_u' \cap \Omega_h^1) \big) \leq \lim_{n_m\to \infty} I_{n_m}^s  = \liminf_{n \to \infty} I_n^s
\end{equation*}
 for a.e.\ $s\in(0,1)$.  
 We obtain \eqref{1805192018} by integrating the above inequality  and by using  \eqref{1805192049} and \eqref{eq: last equation}. Indeed,   the integral on the left-hand side gives exactly $\ove G(u,h)$  as  $c_W=(\int_0^1 \sqrt{2 W(s)} \,\d s)^{-1}$. 

\noindent \emph{Proof of (ii).}  Let $(u, h)$ with   $\ove G(u,h) < + \infty$.  By the construction in the upper inequality for Theorem~\ref{thm:relG},  see Proposition~\ref{prop: enough} and Remark~\ref{rem:1805192117},  
we find $h_n \in C^1(\omega; [0,M]  )$ with $h_n \to h$ in $L^1(\omega)$ and $u_n \in L^\infty(\Omega;\R^d)$ with $u_n|_{\Omega_{h_n}} \in W^{1,p}(\Omega_{h_n}; \Rd)$   and $u_n \to u$ a.e.\ in $\Omega$ such that
\begin{equation}\label{1805192140}
\ove G(u,h)= \lim_{n\to \infty}   H(u_n, h_n)  \quad \text{for } H(u_n, h_n):= \int_{\Omega^+_{h_n}} f(e(u_n)) \dx + \hd(\partial \Omega_{h_n} \cap \Omega)  
\end{equation}
as well as
\begin{equation}\label{1805192141}
 (u_n -  u_0)|_{\omega \times (-1,0)} \to 0 \quad\text{in }W^{1,p}(\omega{\times}(-1,0); \Rd)\,.
\end{equation}
 For each $(u_n,h_n)$, we can use the construction in \cite{ChaSol07} to find sequences $ (u_n^k)_k  \subset W^{1,p}(\Omega;\R^d)$ and $(v_n^k)_k \subset H^1(\Omega;[0,1])$ with $u_n^k = u_n$ on $\omega \times (-1,0)$,  $u_n^k \to u_n$ in $L^1(\Omega;\Rd)$, and  $v^k_n \to \chi_{\Omega_{h_n}}$ in $L^1(\Omega)$  such that  (cf.\ \eqref{1805192140}) 
\begin{equation}\label{1805192154}
\limsup_{k \to \infty} \int_\Omega \bigg( \big( (v_n^k)^2+\eta_{\varepsilon_k}\big) f(e(u_n^k)) + c_W\Big(\frac{W( v_n^k)}{\varepsilon_k} + \frac{\varepsilon_k}{2} |\nabla  v_n^k|^2 \Big) \bigg)\dx  \leq  H(u_n,h_n)\,. 
\end{equation}
In particular, we refer to  \cite[Equation (28)]{ChaSol07} and mention that the functions  $(v_n^k)_k$ can be constructed such that $v_n^k =1$ on $\omega \times (-1,0)$ and $v_n^k = 0$ in $\omega \times (M,M+1)$. We also point out that for this construction the assumption  $\eta_\varepsilon \varepsilon^{1-p} \to 0$ as $\varepsilon \to 0$ is needed.

By \eqref{1805192140}, \eqref{1805192154}, and  a standard diagonal extraction argument we find sequences $(\hat{u}^k)_k \subset (u_n^k)_{n,k}$ and $(v^k)_k \subset (v^k_n)_{n,k}$ such that $\hat{u}^k \to u$ a.e.\ in $\Omega$, $v^k \to \chi_{\Omega_h}$ in $L^1(\Omega)$, and   
\begin{equation}\label{1805192154.2}
\limsup_{k \to \infty} \int_\Omega \bigg( \big( ({v}^k)^2+\eta_{\varepsilon_k}\big) f(e(\hat{u}^k)) + c_W\Big(\frac{W( v^k)}{\varepsilon_k} + \frac{\varepsilon_k}{2} |\nabla  v^k|^2 \Big) \bigg)\dx  \leq  \ove G(u,h)\,. 
\end{equation}
By using \eqref{1805192141} and the fact that $u^k_n = u_n$ for all $k,n \in \N$, we can modify $(\hat{u}^k)_k$ as described at the end of the proof of Proposition \ref{prop: enough}  (see below \eqref{2006191122}): we find a sequence $(u^k)_k$ which satisfies $u^k = u_0$ on $\omega \times (-1,0)$, converges to $u$ a.e.\ in $\Omega$, and \eqref{1805192154.2} still holds, i.e., $\limsup_{k \to \infty}G_{\eps_k}(u^k, v^k) \le \ove G(u,h)$.  This concludes the proof. 
\end{proof}

\begin{appendices}

\section{Auxiliary results}\label{sec:App}
In this appendix, we prove two technical  approximation  results employed in Sections~\ref{sec:FFF} and \ref{sec:GGG},  based  on tools from \cite{CC17}.
 \begin{proof}[Proof of Lemma~\ref{le:0410191844}]
Let $(v, H)$ be given as in the statement of the lemma.
  Clearly, it suffices to prove the following statement: for every   $\eta>0$, there exists $( {v}^\eta, H^\eta) \in L^p(\Omega;\R^d){\times}\M(\Omega)$ with the  regularity and the  properties  required in the   statement of the lemma   (in particular, $ {v}^\eta = u_0$ in a neighborhood $V^\eta \subset \Omega$ of $\partial_D \Omega$), such that, for a universal constant $C$, one has $\GGG\bar{d}\EEE(  v^\eta, v  )\le C\eta$ (cf.\ \eqref{eq:metricd} for \GGG $\bar{d}$\EEE), $\Ld(H\triangle H^\eta)\le C\eta$, and
\[
\ove F'_{\mathrm{Dir}}( {v}^\eta, H^\eta) \le  \ove F'_{\mathrm{Dir}}(v,H) + C\eta\,.
\]
 We start by recalling the main steps of the construction in \cite[Theorem~5.5]{CC17}  and we  refer to \cite{CC17} for  details (see also \cite[Section~4, first part]{CC19b}).   Based on this, we then explain how to construct $( {v}^\eta, H^\eta)$ simultaneously, highlighting particularly the steps needed for constructing  $H^\eta$.

 Let $\varepsilon>0$, to be chosen small with respect to $\eta$. By using     the assumptions on $\partial\Omega$ 
 given  before \eqref{0807170103}, a preliminary step is to find  cubes $(Q_j)_{j=1}^{ J  }$ with pairwise disjoint closures and hypersurfaces $(\Gamma_j)_{j=1}^J$ with the following properties: each $Q_j$ is centered at $x_j \in \partial_N \Omega$ with sidelength $\varrho_j$, $ {\rm dist}  (Q_j, \dod)> d_\varepsilon  >0  $ with $\lim_{\varepsilon \to 0} d_\varepsilon =0$,  and 
\begin{align}\label{eq: guarantee2}
\hd(\don \sm \widehat{Q}) + \Ld(\widehat{Q}) \le \varepsilon,\qquad\text{for }\widehat{Q}:= \bigcup\nolimits_{j=1}^J \ove Q_j\,.
\end{align}  
Moreover, each $\Gamma_j$ is a $C^1$-hypersurface with $x_j \in \Gamma_j \subset \ove Q_j$, 
\begin{equation*}\label{2212181446}
\begin{split}
 \hd\big((\don\triangle \Gamma_j)\,\cap \,  \ove{Q_j}  \big) \le \varepsilon (2\varrho_j)^{d-1}\le \, \frac{\varepsilon}{1-\varepsilon}  \hd(\don\cap \ove{Q_j})\,,
 \end{split}
 \end{equation*}
and  $\Gamma_j$  is  a $C^1$-graph with respect to $\nu_{\dom}(x_j)$ with Lipschitz constant less than $\varepsilon/2$. (We can say that $\don \cap Q_j$  is ``almost'' the intersection of $Q_j$ with the hyperplane passing  through  $x_j$ with normal $\nu_{\dom}(x_j)$.)
We can also guarantee that 
\begin{align}\label{eq: guarantee}
\hd\big((\partial^* H \cup J_u) \cap \Omega \cap \widehat{Q}\big) \le  \varepsilon, \ \ \ \ \ \ \ \ \hd\big((\partial^* H \cup J_u) \cap \partial  Q_j  \big)= 0 
\end{align}
for all   $j=1,\ldots,J$.  
 To each $Q_j$, we associate    the following rectangles:
\begin{equation*}
R_{j}:=\Big\{x_{j}+\sum\nolimits_{i=1}^{d-1} y_i\, b_{j,i}+y_d\, \nu_{j} \colon y_i\in (-\varrho_{j},\varrho_{j}),\, y_d \in (-3\varepsilon \varrho_{j}-t, -\varepsilon \varrho_{j}) \Big\}\,,
\end{equation*}
$$
R'_{j}:=\Big\{x_{j}+\sum\nolimits_{i=1}^{ d-1  } y_i\, b_{j,i}+y_d\, \nu_{j} \colon y_i\in (-\varrho_{j},\varrho_{j}),\, y_d \in (-\varepsilon \varrho_{j}, \varepsilon \varrho_{j}+t) \Big\}\,,
$$
and $\widehat{R}_{j}:=R_{j} \cup R'_{j}$,  where  $\nu_{j}=-\nu_{\dom}(x_{j})$  denotes the generalized outer normal,  
$(b_{j,i})_{i=1}^{d-1}$  is  an orthonormal basis of $(\nu_{j})^\perp$,  and  $t>0$ is small with respect to $\eta$.     We remark that $\Gamma_j \subset R'_j$ and that
$R_j$ is a small strip  adjacent to $R_j'$, which is  included in  $\Omega \cap Q_j$.  (We use here the notation $_j$ in place of $_{h,N}$ adopted in \cite[Theorem~5.5]{CC17}.)

After this preliminary part,  the approximating function  $u^\eta$ was constructed in \cite[Theorem~5.5]{CC17} starting from a given function $u$ through the following three  steps:
\begin{itemize}
\item[(i)] definition of an extension $\widetilde{u} \in GSBD^p(\Omega + B_t(0))$ which is obtained by a reflection argument \emph{à la} Nitsche \cite{Nie81} inside $\widehat{R}_j$,  equal  to $u$ in $\Omega\sm \bigcup_j \widehat{R}_j$, and  equal  to $u_0$ elsewhere.  This can be done  such that, for $t$ and  $\varepsilon$ small,  there holds  (see below \cite[(5.13)]{CC17}) 
\begin{equation}\label{eq: SSS2}
\int\limits_{(\Omega+B_t(0)) \sm \Omega} \hspace{-2em} |e(u_0)|^p \dx + \int\limits_{\widehat{R}}  |e(\widetilde{u})|^p  \dx +    \int\limits_{{R}}  |e(u)|^p \dx    + \hd\big(J_{\widetilde{u}} \cap \widehat{R}\big)   \le \eta\,,
\end{equation}
where  $R:= \bigcup_{j=1}^J {R}_j $ and  $\widehat{R}:=  \bigcup_{j=1}^J  \widehat{R}_j \cap (\Omega+B_t(0))$.   
\item[(ii)]  application of  Theorem~\ref{thm:densityGSBD}  on the function  $\widetilde{u}^\delta:= \widetilde{u}\circ (O_{\delta,x_0})^{-1} + u_0 - u_0 \circ  (O_{\delta,x_0})^{-1}$  (for some $\delta$ sufficiently small)  to get approximating functions $\widetilde{u}^\delta_n$  with the required regularity  which are equal to $u_0 \ast  \psi_n  $ in a neighborhood of $\dod$ in $\Omega$, where   $\psi_n$ is a suitable mollifier. Here, assumption \eqref{0807170103} is crucial.  
\item[(iii)]  correcting the boundary values by defining  $u^\eta$ as $ u^\eta := \widetilde{u}^\delta_n + u_0 - u_0 \ast  \psi_n  $, for $\delta$  and  $1/n$ small enough.
\end{itemize}


 After having recalled the main steps of the construction in \cite[Theorem~5.5]{CC17}, let  us now construct $ {v}^\eta$ and $H^\eta$ at the same time, following the lines of the steps (i), (ii),  and  (iii) above. The main  novelty  is the analog of step (i) for the approximating sets, while the  approximating  functions  are constructed in a very similar way. For this reason, we   do  not recall more details 
from \cite[Theorem~5.5]{CC17}.

\emph{Step (i).} Step (i) for $ {v}^\eta$ is the same done before for  $u^\eta$,  starting from $v$ in place of $u$.  Hereby, we   get a function $\widetilde{v}  \in GSBD^p(\Omega + B_t(0)) $. 

 For the construction of 
$H^\eta$, we introduce a set $\widetilde{H} \subset \Omega+B_t(0)$ as follows:
in $R'_j$, we define a set $H'_j$ by a simple reflection of the set $H \cap R_j$ with respect to the common hyperface between  $R_j$ and $R'_j$. 
Then, we let  
$
\widetilde{H}:= H \cup \bigcup_{j=1}^J   (H'_j \cap (\Omega+B_t(0)))$.  Since  $H$ has finite perimeter,  also $\widetilde{H}$ has  finite perimeter.  By \eqref{eq: guarantee} we get $\hd(\partial^*\widetilde{H} \cap  \widehat{R}  )\le  \eta/3 $ for $\varepsilon$  small,  where as before  $\widehat{R}:=  \bigcup_{j=1}^J  \widehat{R}_j \cap (\Omega+B_t(0))$.   
We choose $\delta$, $\eps$, and $t$ so small that
\begin{align}\label{eq: deltassmall}
\mathcal{H}^{d-1}  \Big(O_{\delta,x_0}\Big( \bigcup\nolimits_{j=1}^J \partial R_j' \setminus \partial R_j \Big) \cap \Omega    \Big) \le \frac{\eta}{ 3 }\,. 
\end{align}
  We let  $ {H}^\eta  :=O_{\delta,x_0}(\widetilde{H})$. Then, we get    $\Ld({H}^\eta\triangle H)\le \eta$ for $\varepsilon$, $t$,  and  $\delta$  small enough.    By \eqref{eq: guarantee2}, \eqref{eq: deltassmall}, and $\hd(\partial^*\widetilde{H} \cap  \widehat{R}  )\le  \eta/3  $  we  also  have (again take suitable $\varepsilon$, $\delta$)
\begin{align}\label{eq: SSS}
\int_{\partial^*{H}^\eta} \varphi(\nu_{{H}^\eta}) \dh \le \int_{  \partial^*  H \cap (\Omega\cup \dod)} \varphi(\nu_H) \dh + \eta\,.
\end{align}
Moreover,  in view of \eqref{0807170103}  and $ {\rm dist}  (Q_j, \dod)> d_\varepsilon  >0  $ for all $j$,   ${H}^\eta$ does not intersect a suitable neighborhood of $\dod$.  Define  $\widetilde{v}^\delta:= \widetilde{v}\circ (O_{\delta,x_0})^{-1} + u_0 - u_0 \circ  (O_{\delta,x_0})^{-1}$   and observe that the function $\widetilde{v}^\delta\chi_{({H}^\eta)^0}$ coincides with $u_0$ in a suitable neighborhood of $\dod$.  By \eqref{eq: SSS}, by the properties recalled for $\widetilde{u}$,  see \eqref{eq: SSS2},  and the fact that $v = v \chi_{H^0}$,   it is elementary to check that 
\begin{align}\label{eq: LLLL}
\ove F'_{\mathrm{Dir}}(\widetilde{v}^\delta \chi_{({H}^\eta)^0}, {H}^\eta)  \le  \ove F'_{\mathrm{Dir}}(v\chi_{H^0}, H) + C \eta =  \ove F'_{\mathrm{Dir}}(v, H) + C \eta\,.
\end{align}
Notice that here it is important to take the same $\delta$ both for $\widetilde{v}^\delta$ and ${H}^\eta$, that is to  ``dilate''   the function and the set at the same time. 

\emph{Step~2.} We apply Theorem~\ref{thm:densityGSBD} to $\widetilde{v}^\delta \chi_{({H}^\eta)^0}$, to get approximating functions $\widetilde{v}^\delta_n$   with the required regularity. For $n$ sufficiently large, we obtain $\GGG\bar{d} \EEE(\widetilde{v}^\delta_n \chi_{({H}^\eta)^0}, \widetilde{v}^\delta \chi_{({H}^\eta)^0})\le \eta$ and    
\[
|\ove F'_{\mathrm{Dir}}(\widetilde{v}^\delta_n \chi_{({H}^\eta)^0}, {H}^\eta) - \ove F'_{\mathrm{Dir}}(\widetilde{v}^\delta \chi_{({H}^\eta)^0}, {H}^\eta)| \le \eta\,.
\]

\emph{Step 3.}  Similar to item  (ii)   above, we obtain  $\widetilde{v}^\delta_n= u_0 \ast  \psi_n  $  in a neighborhood of $\dod$.  Therefore, it is enough  
 to define $ {v}^\eta$ as $  {v}^\eta := \widetilde{v}^\delta_n + u_0 - u_0 \ast  \psi_n  $.  Then  by \eqref{eq: LLLL} and  Step~2 we obtain  $\GGG\bar{d} \EEE(  v^\eta, v  )\le C\eta$ and $\ove F'_{\mathrm{Dir}}( {v}^\eta, H^\eta) \le  \ove F'_{\mathrm{Dir}}(v,H) + C\eta$ for $n$ sufficiently large. 
\end{proof}

 We now proceed with the proof of Lemma~\ref{lemma: vito approx} which  relies strongly  on \cite[Theorem~3.1]{CC17}.  Another main ingredient is the following Korn-Poincar\'e inequality in $GSBD^p$, see \cite[Proposition~3]{CCF16}. 

\begin{proposition}\label{prop:3CCF16}
Let $Q =(-r,r)^d$, $Q'=(-r/2, r/2)^d$, $u\in GSBD^p(Q)$, $p\in [1,\infty)$. Then there exist a Borel set $\omega\subset Q'$ and an affine function $a\colon \Rd\to\Rd$ with $e(a)=0$ such that $\Ld(\omega)\leq cr \hd(J_u)$ and
\begin{equation}\label{prop3iCCF16}
\int_{Q'\setminus \omega}(|u-a|^{p}) ^{1^*} \dx\leq cr^{(p-1)1^*}\Bigg(\int_Q|e(u)|^p\dx\Bigg)^{1^*}\,.
\end{equation}
If additionally $p>1$, then there  exists  $q>0$ (depending on $p$ and $d$) such that, for a given mollifier $\varphi_r\in C_c^{\infty}(B_{r/4})\,, \varphi_r(x)=r^{-d}\varphi_1(x/r)$, the function $ w=u \chi_{Q'\setminus \omega}+a\chi_\omega$ obeys
\begin{equation}\label{prop3iiCCF16}
\int_{Q''}|e( w  \ast \varphi_r)-e(u)\ast \varphi_r|^p\dx\leq c\left(\frac{\hd(J_u)}{ r^{d-1}  }\right)^q \int_Q|e(u)|^p\dx\,,
\end{equation}
where $Q''=(-r/4,r/4)^d$.
The constant in \eqref{prop3iCCF16} depends only on $p$ and $d$, the one in \eqref{prop3iiCCF16} also on $\varphi_1$.  
\end{proposition}


\begin{proof}[Proof of Lemma~\ref{lemma: vito approx}]
We recall the definition of the hypercubes
\begin{equation*}
\begin{split}
q_z^k:=z+(-\km,\km)^d\,,\qquad\tq_z^k:= z+(-2\km,2\km)^d\,, \qquad Q_z^k:=z+(-5\km,5\km)^d\,,
\end{split}
\end{equation*}
 where in addition  to the  notation in \eqref{eq: cube-notation},  we have also defined  the hypercubes $\tqz$.   In contrast to \cite[Theorem~3.1]{CC17}, the cubes $\Qz$ have  sidelength $10k^{-1}$ instead of $8k^{-1}$. This, however, does not affect the estimates. 
 We point out that at some points in \cite[Theorem~3.1]{CC17} cubes of the form $z+(-8\km,8\km)^d$  are used. By a slight alternation of the argument, however, it suffices to take cubes $Q^k_z$.  In particular it is enough to show the inequality \cite[(3.19)]{CC17}  for  a cube $Q_j$ (of sidelength $10k^{-1}$) in place of $\widetilde{Q}_j$ (of sidelength $16k^{-1}$), which may be done by employing rigidity properties of affine functions. 
  Let us fix a smooth radial function $\varphi$ with compact support on the unit ball $B_1(0)\subset \Rd$, and  define  $\varphi_k(x):=k^d\varphi(kx)$.  We choose $\theta< (16c)^{-1}$,  where $c$ is  the constant in Proposition~\ref{prop:3CCF16} (cf.\  also  \cite[Lemma~2.12]{CC17}).     Recall \eqref{eq: well contained} and set 
\begin{equation*}
\mathcal{N}'_k:=\{ z \in  (2k^{-1})  \Z^d \colon \qz \cap (U)^k \sm V \neq \emptyset  \}\,.
\end{equation*}
We apply Proposition~\ref{prop:3CCF16}   for $r = 4k^{-1}$,  for any $z \in \mathcal{N}'_k$  by
taking $v$ as the reference function and $z+(-4k^{-1}, 4k^{-1})^d$ as $Q$ therein. (In the following, we may then use the bigger cube $\Qz$ in the estimates from above.) Then, there exist $\omega_z \subset \tqz$ and $a_z\colon \Rd \to \Rd$ affine with $e(a_z)=0$ such that by   \eqref{eq: cond1},   \eqref{prop3iCCF16}, and H\"older's inequality  there holds 
\begin{subequations}
\begin{equation}\label{1005171230}
\Ld(\omega_z)\leq  4  c k^{-1} \hd(J_{v} \cap \Qz) \leq  4  c \theta k^{-d} \,,
\end{equation}
\begin{equation}\label{prop3iCCF16applicata}
\|v-a_z\|_{L^{p}(\tqz\setminus \omega_z)} \leq  4  ck^{-1} \|e(v)\|_{L^p(\Qz)}\,.
\end{equation}
Moreover,  by \eqref{eq: cond1} and \eqref{prop3iiCCF16}  there holds
\begin{equation*}
\begin{split}
\int_{\qz}|e(\hat{v}_z\ast \varphi_k)-e(v)\ast \varphi_k|^p\dx  \leq c\left(\hd(J_v \cap \Qz)\,k^{d-1}\right)^q \int_{\Qz}|e(v)|^p\dx \leq c \theta^q \int_{\Qz}|e(v)|^p\dx
\end{split}
\end{equation*}
for $\hat{v}_z:= v\chi_{\tqz\setminus \omega_z}+a_z \chi_{\omega_z}$ and a suitable $q>0$ depending on $p$ and  $d$.  
\end{subequations}
Let us set
\begin{equation*}
\omega^k:= \bigcup\nolimits_{ z \in \mathcal{N}_k'  } \, \omega_{z}\,.
\end{equation*}
We order (arbitrarily) the nodes $z \in \mathcal{N}_k'$,  and denote the set by $(z_j)_{j\in J}$.  We define 
\begin{equation}\label{eq:defappr1}
\widetilde{w}_k:=
\begin{cases}
 v \quad &\text{in }\big(\bigcup_{z \in \mathcal{N}'_k} \Qz \big) \setminus \omega^k\,,\\
 a_{z_j}\quad &\text{in }\omega_{z_j}\setminus \bigcup_{i<j}\omega_{z_i}\,,
\end{cases}
\end{equation}
and
\begin{equation}\label{eq:defapprox}
 w_k:=  \widetilde{w}_k \ast \varphi_k \quad \text{in }(U)^k \sm V\,.
\end{equation}
We have that $w_k$ is smooth since $(U)^k \sm V + \mathrm{supp} \,\varphi_k \subset \bigcup_{z \in \mathcal{N}'_k} \tqz  \subset U $  (recall  \eqref{eq: well contained})  and $  v|_{\tqz \sm \omega^k}  \in L^p(\tqz \sm \omega^k;  \R^d  )$ for any $z \in \mathcal{N}'_k$,
by \eqref{prop3iCCF16applicata}.  

 We define the  sets $G^k_1:=\{ z \in \mathcal{N}'_k \colon \hd(J_v \cap \Qz)\leq k^{1/2 - d}\}$ and  $G^k_2:= \mathcal{N}'_k \sm G^k_2$.  By  $\widetilde{G}^k_1$ and  $\widetilde{G}^k_2$, respectively, we denote their  ``neighbors'',  see \cite[(3.11)]{CC17} for the exact definition.  We let 
\begin{equation*}
 \widetilde{\Omega}^k_{g,2}:=   \bigcup\nolimits_{z \in \widetilde{G}^k_2} \,  \Qz\,.
\end{equation*}
 There holds  (cf.\ \cite[(3.8), (3.9), (3.12)]{CC17})
\begin{equation}\label{eq: small vol--}
\lim_{k\to \infty} \big( \Ld(\omega^k) +  \Ld(\widetilde{\Omega}^k_{g,2})\big) = 0\,.
\end{equation}
At this point, we  notice that the set  $E_k$   in \cite[(3.8)]{CC17} reduces to $\omega^k$ since in our situation all nodes are ``good'' (see \eqref{eq: cond1} and \cite[(3.2)]{CC17}) and therefore $\widetilde{\Omega}^k_b$ therein is empty. 

The proof of (3.1a), (3.1d), (3.1b) in \cite[Theorem~3.1]{CC17} may be followed exactly, with the modifications described just above and the suitable slight change of notation.  More precisely,  by \cite[equation below (3.22)]{CC17} we obtain
 \begin{equation}\label{2006191938}
 \|w_k-v\|_{L^p( ((U)^k \sm V) \sm \omega^k)} \leq Ck^{-1} \|e(v)\|_{L^p(U)}\,,
 \end{equation}
  for a  constant $C>0$ depending only on $d$ and $p$, 
 and \cite[equation before (3.26)]{CC17} gives
\begin{equation}\label{2006191940}
\int_{\omega^k} \psi(|w_k-v|) \dx \leq C \Big( \int_{\omega^k \cup \widetilde{\Omega}^k_{g,2}} \big(1+\psi(|v|)\big) \dx + k^{-1/2} \int_U \big( 1+\psi(|v|)\big) \dx  +  k^{-p}\int_U  |e(v)|^p    \,\dx \Big)\,,
\end{equation}
 where $\psi(t) = t \wedge 1$.   Combining \eqref{2006191938}-\eqref{2006191940}, using \eqref{eq: small vol--}, and recalling  that  $\psi$ is sublinear, we obtain \eqref{rough-dens-1}. Note that the sequence $R_k \to 0$ can be chosen independently of $v \in \mathcal{F}$ since $\psi(|v|) + |e(v)|^p$ is equiintegrable for $v \in \mathcal{F}$.

Moreover, recalling \eqref{eq:defappr1}-\eqref{eq:defapprox}, we sum \cite[(3.34)]{CC17} for $z=z_j \in \widetilde{G}^k_2$ and \cite[(3.35)]{CC17} for $z=z_j \in \widetilde{G}^k_1$ to obtain 
$$\int_{(U)^k \setminus V} |e(w_k)|^p \, \dx \le \int_U |e(v)|^p \, \dx + Ck^{-q'/2} \int_U |e(v)|^p \, \dx + C\int_{\widetilde{\Omega}^k_{g,2}} |e(v)|^p \, \dx $$
for some $q' >0$. This along with \eqref{eq: small vol--} and the equiintegrability of $|e(v)|^p$ shows   \eqref{rough-dens-2}.  
\end{proof}

%

\end{appendices}

\section*{Acknowledgements} 
VC is supported by the Marie Sk\l odowska-Curie Standard European Fellowship No.\ 793018.
MF acknowledges support by the DFG project FR 4083/1-1 and by the Deutsche Forschungsgemeinschaft (DFG, German Research Foundation) under Germany's Excellence Strategy EXC 2044 -390685587, Mathematics M\"unster: Dynamics--Geometry--Structure.

\medskip
\noindent {\bf Conflict of interest and ethical statement.} The authors declare that they have no conflict of interest and guarantee the compliance with the Ethics Guidelines of the journal.
\bigskip


\end{document}